\documentclass[twoside,11pt]{article}
\usepackage{graphicx}
\usepackage{color}
\usepackage{amsmath,amssymb,amsthm,amsfonts}
\usepackage{graphicx}
\usepackage{color}
\usepackage{multicol}

\usepackage{mathptmx}
\usepackage{enumerate}
\usepackage[numbers,sort&compress]{natbib}
\numberwithin{equation}{section}
\newcommand{\R}{{\mathbb R}}

\newcommand{\N}{{\mathbb N}}

\newcommand{\be}{\begin{equation}}
\newcommand{\ee}{\end{equation}}

\newcommand{\ben}{\begin{eqnarray*}}
\newcommand{\enn}{\end{eqnarray*}}

\newcommand{\G}{\Gamma}

\newcommand{\Om}{\Omega}

\newcommand{\al}{\alpha}

\newcommand{\va}{\varepsilon}
\newcommand{\ti}{\tilde}

\newcommand{\var}{\varphi}

\newcommand{\De}{\Delta}
\newcommand{\na}{\nabla}
\newcommand{\de}{\delta}
\newtheorem{theorem}{\textbf Theorem}[section]
\newtheorem{lemma}{\textbf Lemma}[section]
\newtheorem{rem}{\textbf Remark}[section]

\newtheorem{prop}{\textbf Proposition}[section]

\def\endProof{{\hfill$\Box$}}

\begin{document}
\title{{\textbf{Global solutions to the free boundary value problem of a chemotaxis-Navier-Stokes system} }}
\author{Qianqian Hou\thanks{Institute for Advanced Study in Mathematics, Harbin Institute of Technology, Harbin 150001, People's Republic of China
({\tt qianqian.hou@hit.edu.cn}).}}
\date{}
\maketitle

\begin{quote}
{\sc Abstract} In this paper, we investigate the global solvability of the chemotaxis-Navier-Stokes system on a three-dimensional moving domain of finite depth, bounded below by a rigid flat bottom and bounded above by the free surface. Completing the system with boundary conditions that match the boundary descriptions in the experiments and numerical simulations, we establish the global existence and uniqueness of solutions near a constant state $(0,\hat{c},\mathbf{0})$, where $\hat{c}$ is the saturation value of the oxygen on the free surface. To the best of our knowledge, this is the first analytical work for the well-posedness of chemotaxis-Navier-Stokes system on a time-dependent domain.



\end{quote}
\maketitle

\section{Introduction}
\textbf{Background and literature review.}
When suspension of an oxytactic bacteria denser than water like \emph{Bacillus subtilis} is placed in a chamber with its upper surface open to the atmosphere, bacterial cells swim up the gradient of the oxygen which diffuses to the suspension through the air-fluid interface and quickly get densely packed below the interface in a relatively thin layer. Subsequently,  Rayleigh-Taylor type instabilities occur due to the buoyancy effect and evolve ultimately into the bioconvection patterns observed in the experiments \cite{dombrowski-cisneros-chatkaew2004,hillesdon-pedley-kessler1995,
hillesdon-pedley1996}. To describe this chemotaxis-diffusion-convection  process the following chemotaxis-(Navier-)Stokes system has been proposed in \cite{Tuval}:
\be\label{cns}
\left\{
\begin{array}{lll}
m_t+\vec{u}\cdot\na m+\nabla\cdot(m\chi(c) \na c)=D_m\, \Delta m \quad \text{for}\ \ x\in \Om_t \ \ \text{and}\ \ t>0,
\\
c_t+\vec{u}\cdot \na c+mf(c)=D_c\, \Delta c,\\
\vec{u}_t+\kappa \vec{u}\cdot\na \vec{u}+\nabla p+m\nabla \Phi=D\Delta\vec{u},\\
\na\cdot\vec{u}=0,\\
\end{array}
\right.
\ee
where $\Om_t$ is a domain in $\R^d$ that may evolve with time $t$. The unknown functions $m(x,t)$ and $c(x,t)$ are the bacteria density and oxygen concentration. $\vec{u}(x,t)$ and $p(x,t)$ denote the fluid velocity and the pressure above the hydrostatic value. The positive constants $D_m$, $D_c$ and $D$ are diffusion rates of the bacterial cells, the oxygen and the velocity, respectively.  The first two equations in \eqref{cns} describe the chemotactic movement of bacteria
towards increasing gradients of the oxygen concentration with chemotactic intensity $\chi(c)>0$ and oxygen consumption rate $f(c)>0$, where both bacteria and oxygen
diffuse and are convected with the fluid. In turn, the influence of the bacterial cells on the fluid is through the given function $\Phi(x,t)$, which denotes the potential function produced by different physical mechanisms, e.g. the gravity or centrifugal forces.


  The striking feature of \eqref{cns} is that it couples the well-known obstacles in theory of hydrodynamics to the typical difficulties arising in the study of chemotaxis system. Indeed, due to the lack of effective mathematical tools handling the cross-diffusive term $\nabla\cdot(m\chi(c) \na c)$, the answer is still unavailable to the question whether the global weak solutions (constructed in \cite{tao-winkler2012})
   of the three-dimensional chemotaxis-only subsystem obtained from \eqref{cns} by letting $\vec{u}=\mathbf{0}$
   may blow up at a finite time or not. On the other hand, the global well-posedness to the three-dimensional incompressible Navier-Stokes equations with arbitrary large initial data remains a prominent open problem in hydrodynamics. In spite of these challenges, extensive studies have been conducted numerically and analytically in the last decades and most of the results achieved are focusing on the pattern formation of bacteria cells and global well-posedness for the corresponding initial-boundary value problem on fixed spatial domains $\Om$ independent of $t$. Here, we only mention the previous results related to the present paper.

In the case $\Om=\R^2$, for the chemotaxis-Stokes system obtained from \eqref{cns} by taking $\kappa=0$ in the fluid evolution, the global weak solutions have been constructed
(cf. \cite{duan-lorz-markowich2010}) under some smallness assumptions on either the potential function or the initial oxygen concentration along with certain structural conditions on $\chi$ and $f$.
For the same two-dimensional Cauchy problem, Liu and Lorz  (cf. \cite{liu-lorz2011}) removed the above smallness assumptions and showed global weak solvability even for the
chemotaxis-Navier-Stokes system (i.e. $\kappa=1$) under basically the same conditions on $\chi$ and $f$ as made in \cite{duan-lorz-markowich2010}. Uniqueness of such solutions was justified later (cf. \cite{zhang-zheng2014}) by taking advantage of a coupling structure of the equations and using the Fourier localization technique. To include the prototypical choices $\chi=const.$ and $f(c)=c$ (cf. \cite{chertock-fellner2012,hillen2008,Tuval}), Chae-Kang-Li (cf. \cite{chae-kang-li2012,chae-kang-li2014}) relaxed the structural conditions on $\chi$, $f$ and demonstrated the global well-posedness for the chemotaxis-Navier-Stokes system under smallness requirement on $\|c_0\|_{L^\infty}$. They also obtained some blow-up criteria that allow the weak solutions derived in \cite{liu-lorz2011} to become a classical one upon improving the regularity of initial data.
In the case $\Om=\R^3$, the problem of global well-posedness seems to be more delicate: to the best of our knowledge, results available so far are merely confined to local and global small solutions (cf. \cite{chae-kang-li2012,chae-kang-li2014,
duan-lorz-markowich2010}).

When $\Om$ is a bounded domain (independent of $t$) in $\R^d$, $d=2,3$ with smooth boundary, system \eqref{cns} subject to
 the following boundary conditions:
\be\label{bc}
\na m\cdot \vec{n}=0,\qquad \na c\cdot \vec{n}=0,\qquad \vec{u}=\mathbf{0},
\ee
with $\vec{n}$ the outward unit normal to the boundary $\partial \Om$, was investigated in \cite{lorz2010} and local weak solutions were constructed in the situation $\chi$ being a constant and $f$ being monotonically increasing with $f(0)=0$. Under the structural hypotheses
 $(\frac{f(s)}{\chi(s)})^{'}>0$, $(\frac{f(s)}{\chi(s)})^{''}\leq 0$ and $(\chi(s)f(s))^{'}\geq 0$, Winkler derived the global existence of weak solutions in the 3D case with $\kappa=0$ and of smooth solutions in the 2D case with $\kappa\in \R$ (cf. \cite{winkler2012}). Such smooth solutions in the latter 2D case stabilize to the spatially uniform equilibria $(\bar{m}_0,0,\mathbf{0})$ with $\bar{m}_0=\frac{1}{|\Om|}\int_\Om m(x,0)dx$ in the large time limit (cf. \cite{winkler2014arma}) at an exponential convergence rate (cf. \cite{zhang-li2015}). Their convergence in small-convection limit $\kappa\rightarrow 0$ was later justified in \cite{wang-winkler-xiang2018}.
Global existence of weak solutions for the three-dimensional chemotaxis-Navier-Stokes system was established in \cite{winkler2016} under certain structural requirements on $\chi$ and $f$. Similar to the 2D case, such weak solutions enjoy eventual smoothness and approach the unique spatially homogeneous steady state $(\bar{m}_0,0,\mathbf{0})$ as $t$ goes to infinity (cf. \cite{winkler2017}). Recently, global solvability of the chemotaxis-Navier-Stokes system in a three-dimensional unbounded domain $\Om$ with infinite extent and finite depth was justified in \cite{peng-xiang2018} under appropriate smallness assumptions on initial data.

~\\
\textbf{Goals and motivations.}
As aforementioned, most of the previously analytical studies devoted to the chemotaxis-Navier-Sotkes system in the literature are confined to the fixed domain setting. However, the domain is normally deformable in natural conditions. For instance, considering that a large variety of swimming micro-organisms live in the vast ocean lying above a rigid bottom it is realistic to investigate the dynamics of cell-fluid interactions in domains with upper surface evolving in time. Allowing the motion of the upper surface and completing system \eqref{cns} with appropriate boundary conditions, the linear and nonlinear stability analysis have been carried out in \cite{chakraborty2018} along with supporting numerical simulations in a 2D shallow chamber and the effect of free-surface on bacterial plume patterns and their dynamics in both 2D and 3D cases have been recently explored numerically in \cite{Ivan2019,Ivan2020}.
 The rigorous mathematical analysis for \eqref{cns} in time-dependent domains are lack of investigations even on the natural first question of its well-posedness.

Motivated by the above numerical results and the experiments conducted in \cite{hillesdon-pedley-kessler1995,
hillesdon-pedley1996,Tuval}, we shall study the global well-posedness of system \eqref{cns} in the following three-dimensional moving domain above a rigid bottom defined by:
\ben
\Om_t=\{(x_1, x_2, y)\in \mathbb{R}^3: \ -b<y<\eta(x_1,x_2,t)\},
\enn
where $b$ is a positive constant and the surface function $\eta(x_1,x_2,t)$ depends on the horizontal variable $(x_1,x_2)\in \R^2$ and temporal variable $t$. For illustration, we denote $S_F=\{(x_1,x_2,y)\in \mathbb{R}^3:\  y=\eta(x_1,x_2,t)\}$, $S_B=\{(x_1,x_2,y)\in \mathbb{R}^3:\  y=-b\}$ and $\Om_0=\{(x_1,x_2,y)\in\mathbb{R}^3:\ -b<y<\eta_0(x_1,x_2)\}$ with given $\eta_0$ defined on $\R^2$.

 The choice on chemotactic intensity and oxygen consumption rate  in the present paper is $\chi(c)=\frac{1}{c}$ and $f(c)=c$. Substituting $\chi(c)=\frac{1}{c}$ into the chemotactic sensitivity function $\chi(c)\na c$ in the first equation of \eqref{cns} leads to the logarithmic sensing $\na \ln c$, which has been experimentally verified in \cite{Kalinin}. This logarithm results in a mathematically unfavorable singularity which, however, has been adopt in the literature to describe various chemotaxis process e.g.  dynamical interactions between vascular endothelial cells and signaling molecules vascular endothelial growth factor in the initiation of tumor angiogenesis (cf. \cite{LSN}), the boundary movement of chemotactic bacterial populations (cf. \cite{Nossal72}) and the travelling band behavior of bacteria (cf. \cite{KS71b,lui-wang2010}). With such choice on $\chi(c)$ and $f(c)$, the initial-boundary value problem studied in the present paper reads

 \be\label{e04}
\left\{
\begin{array}{lll}
m_t+\vec{u}\cdot \na m+\nabla\cdot(m\nabla \ln c)=\Delta m\quad \text{for}\ \ (x_1,x_2,y)\in \Om_t \ \ \text{and}\ \ t>0,\\
c_t+\vec{u}\cdot \na c+mc=\Delta c,\\
\vec{u}_t+\vec{u}\cdot\na \vec{u}+\nabla p+m\nabla \Phi=\Delta\vec{u}, \\
\na \cdot\vec{u}=0,\\
(m, c, \vec{u})(x_1,x_2,y,0)=(m_0, c_0, \vec{u}_0)(x_1,x_2,y) \quad \text{for}\ \ (x_1,x_2,y)\in \Om_0,
\end{array}
\right.
\ee
with the following boundary conditions on $S_F$:
\be\label{e05}
\left\{
\begin{array}{lll}
(\na m -m\nabla \ln c)\cdot \vec{n}=0,\\
c=\hat{c},\\
\eta_t=u_3-u_1\partial_1\eta -u_2\partial_2\eta,\\
p n_i-(\partial_j u_i+\partial_i u_j)n_j=\left\{
\gamma\eta-\sigma\na\cdot\left(
\frac{\na \eta}{\sqrt{1+|\na\eta|^2}}
\right)
\right\}n_i
\end{array}
\right.
\ee
and the following boundary conditions on $S_B\times(0,\infty)$:
\be\label{e06}
m=0,\quad \partial_3 c=0, \quad \vec{u}=\mathbf{0},
\ee
where the constant $\hat{c}>0$ is the saturation value of the oxygen on the free surface $S_F$, $\vec{n}=(n_1,n_2,n_3)$ is the outward unit normal to $S_F$ and we sum upon repeated indices following the Einstein convention (this convention will be used in the remaining part of this paper without further clarification). $D_m,D_c,D$ have been taken to be $1$ without loss of generality and initial-boundary conditions have been imposed.

We next briefly introduce the derivation of the boundary conditions \eqref{e05} on $S_F$.
\begin{itemize}
\item the kinematic condition: denote the free surface by $d(x_1,x_2,y,t)=y-\eta(x_1,x_2,t)=0$. Since fluid particles do not cross the free surface, we have $(\partial_t+\vec{u}\cdot \na)(y-\eta(x_1,x_2,t))=0$, which results in $\eta_t=u_3-u_1\partial_1\eta -u_2\partial_2\eta$. Further discussion of this condition can be found, e.g., in \cite[page 451]{wehausen-laiton}.
\item the normal force balance condition: the last boundary condition in \eqref{e05} states a discontinuity in the normal stress on two sides of the free surface which, is proportional to the mean curvature of the surface and produced by the effect of surface tension, where $(p-\gamma\eta) n_i-(\partial_j u_i+\partial_i u_j)n_j$ is the stress difference between two sides of the surface, $\na\cdot\left(
\frac{\na \eta}{\sqrt{1+|\na\eta|^2}}
\right)$ is the mean curvature of the surface, $\gamma>0$ is the  acceleration of gravity and $\sigma>0$ denotes the coefficient of surface tension. We refer the reader to \cite[page 451-454]{wehausen-laiton} for detailed derivation of this condition.
\item
the zero-flux boundary condition on $m$ and the Dirichlet boundary condition on $c$ in \eqref{e05} are physical ones that match the boundary descriptions in the experiments conducted in \cite{hillesdon-pedley-kessler1995,
hillesdon-pedley1996,Tuval} and the numerical analysis in \cite{chakraborty2018, Ivan2019, Ivan2020}.
\end{itemize}

As mentioned before, our goal is to establish the global solvability of system \eqref{e04}-\eqref{e06} under appropriate small assumptions on initial data. To this end,
we shall first apply the following transformation (cf. \cite{winkler2017radial})
\be\label{tr}
\tilde{c}=-\ln c+\ln \hat{c}
\ee
to
system \eqref{e04}-\eqref{e06} to resolve the logarithmic singularity in its first equation and study the global well-posedness of the following transformed system
\be\label{e01}
\left\{
\begin{array}{lll}
m_t+\vec{u}\cdot\na m-\nabla\cdot(m\nabla \tilde{c})=\Delta m\quad \text{for}\ \ (x_1,x_2,y)\in \Omega_{t}\ \ \text{and}\ \ t>0,\\
\tilde{c}_t+\vec{u}\cdot \na\tilde{c}+|\na \tilde{c}|^2-m=\Delta \tilde{c},\\
\vec{u}_t+\vec{u}\cdot\na \vec{u}+\nabla p+m\nabla \Phi=\Delta\vec{u},\\
\na \cdot\vec{u}=0,\\
(m, \ti{c}, \vec{u})(x_1,x_2,y,0)=(m_0, \tilde{c}_0, \vec{u}_0)(x_1,x_2,y)\quad \text{for}\ \ (x_1,x_2,y)\in \Om_0,
\end{array}
\right.
\ee
with
\be\label{e02}
\left\{
\begin{array}{lll}
(\na m +m\nabla \tilde{c})\cdot \vec{n}=0,\\
\ti{c}=0,\\
\eta_t=u_3-u_1\partial_1\eta -u_2\partial_2\eta,\\
p n_i-(\partial_j u_i+\partial_i u_j)n_j=\left\{
\gamma\eta-\sigma\na\cdot\left(
\frac{\na \eta}{\sqrt{1+|\na\eta|^2}}
\right)
\right\}n_i
\end{array}
\right.
\ee
on $S_F$ and
\be\label{e03}
m=0,\quad \partial_3 \tilde{c}=0, \quad \vec{u}=\mathbf{0}
\ee
on $S_B\times(0,\infty)$.

\section{Notation and main results}
\textbf{Notation.} For clarity, we specify some notations below.
\begin{itemize}
  \item $\Omega_t=\{(x_1, x_2, y)\in \mathbb{R}^3: \ -b<y<\eta(x_1,x_2,t)\}.$
  \item $S_B=\{(x_1,x_2,y)\in \mathbb{R}^3: \ y=-b\}.$
  \item $S_F=\{(x_1,x_2,y)\in \mathbb{R}^3: \ y=\eta(x_1,x_2,t)\}.$
\item $\Om_0=\{(x_1,x_2,y)\in\mathbb{R}^3:\ -b<y<\eta_0(x_1,x_2)\}.$
\item $S_0=\{(x_1,x_2,y)\in\mathbb{R}^3:\ y=\eta_0(x_1,x_2)\}.$
\item $\Omega=\{(x_1,x_2,y)\in\mathbb{R}^3: \  -b<y<0\}.$
\item $\Gamma=\{(x_1,x_2,y)\in \mathbb{R}^3:\  y=0\}.$
\item We denote $dx=dx_1dx_2$ for $\vec{x}=(x_1,x_2)\in \mathbb{R}^2$ and denote $dxdy=dx_1dx_2dy$ for $(x_1,x_2,y)\in \Om$.
\item $[\vec{v},\vec{u}]=\frac{1}{2}\int_{\Om}(\partial_jv_i+\partial_iv_j)
    (\partial_ju_i+\partial_iu_j)dxdy$ for $\vec{v}=(v_1,v_2,v_3)$ and $\vec{u}=(u_1,u_2,u_3)$.

\item $\na_0$ denotes tangential gradient along the $x_1-x_2$ plane.
\item $H^m\,(m\geq 1)$ represents $H^m(\Om)$ and $L^p\,(1\leq p\leq \infty)$ stands for $L^p(\Om)$. In particular, $H^0=H^0(\Om)=L^2(\Om).$
\item ${_{0}}{H}^1$ and
${^0}{H}^1$ represent the subspace of $H^1(\Om)$ consisting of functions which vanish on $S_B$ and $\G$, respectively.

\item  For Banach space $\mathbf{X}$, we use $\mathbf{X}^{'}$ to denote its dual space and use $\|\cdot\|_{L^q_t\mathbf{X}}\,(1\leq q\leq \infty)$ to denote $\|\cdot\|_{L^q(0,t;\mathbf{X})}$ for $t>0$. $H^{-\frac{1}{2}}(\G)$ represents the dual space of $H^{\frac{1}{2}}(\G)$.
\item $\mathcal{H}$ is the harmonic extension operator, also denoted by $\bar{\eta}=\mathcal{H}(\eta)$, extending functions defined on $\G$ to harmonic functions on $\Om$ with zero Neumann boundary condition on $S_B$. Specifically, for any $\eta(x_1,x_2)$ defined on $\G$, its harmonic extension $\bar{\eta}$ solves
\be\label{eee9}
\left\{
\begin{array}{lll}
\De \bar{\eta}=0\quad \ \text{in}\,\,\,\Omega,
\\
\bar{\eta}=\eta\quad \ \ \ \text{on}\,\,\,\G,\\
\partial_3\bar{\eta}=0\quad \text{on}\,\,\,S_B.
\end{array}
\right.
\ee
\item We use $C$ to denote a generic constant which is independent of $t$, and may change from one line to another.
\end{itemize}

~\\
\textbf{Compatibility conditions.} Since the problem \eqref{e01}-\eqref{e03} is posed on a domain with boundary, it is natural to impose on initial data the following compatibility conditions:
\be\label{cb}
\left\{
\begin{array}{lll}
[(\partial_j u_{0i}+\partial_i u_{0j})n_{0j}]_{\text{tan}}=0\quad \text{on}\ \ \ S_0,\\
\na \cdot \vec{u}_0=0 \quad \text{in}\quad \Om_0,\\
(\na m_0 +m_0\nabla \tilde{c}_0)\cdot \vec{n}_0=0, \quad  \ti{c}_0=0\ \ \ \text{on}\ \ \ S_0,\\
m_0=0,\quad \partial_3 \tilde{c}_0=0,\quad \vec{u}_0=\mathbf{0}\quad \text{on} \quad S_B,
\end{array}
\right.
\ee
where ``tan'' means the tangential component, $\vec{n}_0=(n_{01},n_{02},n_{03})$ is the outward unit normal to the initial surface $S_0$ and the first condition is obtained by taking the inner product of the stress difference between two sides of the initial surface with any tangential vector.

We are now in a position to state the main results of this paper.

\begin{theorem}\label{t2}
Suppose that the initial data $\vec{u}_0=(u_{01},u_{02},u_{03}),\,m_0, \, \ti{c}_0 \in H^2(\Om_0)$ and $\eta_0\in H^3(\R^2)$ fulfill the compatibility conditions \eqref{cb}.
Assume further
\be\label{e004}
\na \Phi,\ \na^2 \Phi\in L^\infty(0,\infty; L^\infty(\R^3)),\qquad \na \Phi_t\in L^2(0,\infty; L^2(\R^3)).
\ee
Then there exists a constant $\va_0>0$ suitably small such that if $\|m_0\|_{H^2(\Om_0)}+\| \ti{c}_0\|_{H^2(\Om_0)}+\|\vec{u}_0\|_{H^2(\Om_0)}
+\|\eta_0\|_{H^3(\R^2)}<\va_0$,
system \eqref{e01}-\eqref{e03} admits a unique solution $(m,\tilde{c},\vec{u},p,\eta)$ satisfying
\be\label{e07}
\begin{split}
&\ \sup_{t>0}\,\big(\|m(t)\|_{H^2(\Om_t)}
+\|\ti{c}(t)\|_{H^2(\Om_t)}
+\|\vec{u}(t)\|_{H^2(\Om_t)}
+\| \na p(t)\|_{L^2(\Om_t)}
+\|\eta(t)\|_{H^3(\R^2)}\big)^2\\
&\ \ +\int_0^\infty \big(\|m(t)\|_{H^3(\Om_t)}
+\|\tilde{c}(t)\|_{H^3(\Om_t)}
+\|\vec{u}(t)\|_{H^3(\Om_t)}
+\| \na p(t)\|_{H^1(\Om_t)}
+\|\na_0\eta(t)\|_{H^\frac{5}{2}(\R^2)}\big)^2\,dt\\
&\leq C \big(\|m_0\|_{H^2(\Om_0)}+\|\ti{c}_0\|_{H^2(\Om_0)}
+\|\vec{u}_0\|_{H^2(\Om_0)}
+\|\eta_0\|_{H^3(\R^2)}\big)^2
\end{split}
\ee
for some positive constant $C$.
\end{theorem}

With \eqref{tr} and the results obtained for the transformed system \eqref{e01}-\eqref{e03}, we have the following assertions for the initial-boundary value problem \eqref{e04}-\eqref{e06}.

\begin{theorem}\label{t1}
Let $\hat{c},\,c_0>0$ and $m_0\geq 0$. Suppose the assumptions in Theorem \ref{t2} hold with $\ti{c}_0=-\ln c_0+\ln \hat{c}$.
Then there exists a unique solution $(m,c,\vec{u},p,\eta)$ to system \eqref{e04}-\eqref{e06} satisfying that
\ben
m(x_1,x_2,y,t)\geq0\quad\text{and}\quad c(x_1,x_2,y,t)>0
\enn
for $(x_1,x_2,y)\in \Om_t$ and $t>0$, and that
\be\label{e08}
\begin{split}
&\ \sup_{t>0}\,\big(\|m(t)\|_{H^2(\Om_t)}
+\|c(t)-\hat{c}\|_{H^2(\Om_t)}
+\|\vec{u}(t)\|_{H^2(\Om_t)}
+\|\na p(t)\|_{L^2(\Om_t)}
+\|\eta(t)\|_{H^3(\R^2)}\big)^2\\
&\ \ +\int_0^\infty \big(\|m(t)\|_{H^3(\Om_t)}
+\|c(t)-\hat{c}\|_{H^3(\Om_t)}
+\|\vec{u}(t)\|_{H^3(\Om_t)}
\big)^2\,dt\\
&\ \ +\int_0^\infty
(\big\| \na p(t)\|_{H^1(\Om_t)}
+\|\na_0\eta(t)\|_{H^\frac{5}{2}(\R^2)}\big)^2\,dt\\
&\leq C \big[\sum_{k=1}^3\big(\|m_0\|_{H^2(\Om_0)}+\|\ln c_0-\ln\hat{c}\|_{H^2(\Om_0)}
+\|\vec{u}_0\|_{H^2(\Om_0)}
+\|\eta_0\|_{H^3(\R^2)}\big)^{2k}\big]\\
&\ \quad \ \ \times \exp\{C \big(\|m_0\|_{H^2(\Om_0)}+\|\ln c_0-\ln\hat{c}\|_{H^2(\Om_0)}
+\|\vec{u}_0\|_{H^2(\Om_0)}
+\|\eta_0\|_{H^3(\R^2)}\big)\}
\end{split}
\ee
for some $C>0$.
\end{theorem}

\begin{rem}
As mentioned in the introduction, in the experiments (cf. \cite{dombrowski-cisneros-chatkaew2004,
hillesdon-pedley1996, Tuval}) the potential function $\Phi$ denotes the gravity or centrifugal forces, both of which are independent of the temporal variable $t$. Although
in Theorem \ref{t2} and Theorem \ref{t1} we have treated the more general case, where the given potential functions $\Phi$ are allowed to depend on $t$, our results can be applied to the above two cases, where $\Phi(x_1,x_2,y)=\lambda\gamma y$ for some constant $\lambda\in\R$ depending on the fluid mass density and the cell density represents the gravity potential, and $\Phi(x_1,x_2,y)=\Phi(x_1^2+x^2_2+y^2)\to 0$ as $(x_1^2+x^2_2+y^2)\to\infty$ stands for the centrifugal potential, since the $\Phi$ in both cases satisfy
the assumption \eqref{e004}.
\end{rem}

\begin{rem} Setting $\eta_0=\eta=p=\Phi=0$ and $\vec{u}_0=\vec{u}=\mathbf{0}$ in Theorem \ref{t1}, one derives the unique global solution $(m,c)\in L^\infty(0,\infty;H^2)\times L^2(0,\infty;H^3)$ to the chemotaxis-only subsystem on $\Omega$. To the best of our knowledge, this is the first analytical result on the well-posedness of the chemotaxis system with  logarithmic sensitivity in an unbounded domain with boundary.
\end{rem}

We next briefly introduce the main ideas and organisation of this paper.
Due to the lack of well-posedness theory for parabolic systems on a moving domain, we shall first transform \eqref{e01}-\eqref{e03} into the initial-boundary value problem \eqref{e001}-\eqref{e003} on an equilibrium domain in Section 3, whose detailed derivations are postponed to be given in the Appendix. Solvability of the linearized version of \eqref{e001}-\eqref{e003} is demonstrated in Section 4.
When solving the linear system in Section 4, the main challenge that we are face of is that we can not address the higher energy estimates in the vertical direction for the velocity field $\vec{v}$ by directly taking the partial derivatives in the equations of \eqref{e1} since the equilibrium domain is not translation-invariant in the vertical spatial variable.
While this difficulty can be circumvented in the situation where  the velocity field subject to the no-sip boundary condition as in \cite{peng-xiang2018}, by employing the elliptic theory on the corresponding stationary Stokes system with the aid of the estimation of the time derivative of $\vec{v}$, resorting to the elliptic estimates on system \eqref{pe1} to address the $L^2_tH^3$ regularity of $\vec{v}$ claimed in Proposition \ref{p1} seems rather hopeless in the present case due to the boundary condition $q-2\partial_3 v_3=\gamma\eta-\sigma\Delta_0\eta-g_3$ on $\G$, which incorporates the normal stress with the surface function $\eta$. In fact, to obtain this $L^2_tH^3$ regularity of $\vec{v}$ by utilizing Lemma \ref{pl6}, the boundary term $\gamma\eta-\sigma\Delta_0\eta$ is required to belongs to the space $L^2_tH^{\frac{3}{2}}(\G)$, however, because of the fact that $\eta_t=v_3$ on $\G$, one only derives the $L^\infty$ in-time estimate of $\eta$ by applying the standard testing procedure to system \eqref{e1} (see Lemma \ref{nl2}), so that we only expect the $L^\infty_tH^2$ regularity on $\vec{v}$ if we make use of the elliptic estimates presented in Lemma \ref{pl6}. Accordingly, our approach to establish this $L^2_tH^3$ regularity on the velocity field will be based on an alternative method to verify the intuitive idea that the higher energy estimates on the tangential direction of $\vec{v}$ should provide suitable regularity properties on its vertical direction thanks to the divergence-free property of $\vec{v}$.
 Noting that the domain is translation-invariant in the horizontal direction, we start by taking tangential derivatives in the first equation of \eqref{e1} to gain energy bounds for the tangential derivatives of the velocity field $\vec{v}$. By taking advantage of the divergence-free condition on $\vec{v}$, the boundary estimates (on $\G$)
for its third component $v_3$ readily follows from these tangential bounds (see the proof of Lemma \ref{nl5}). Then the desired $L^2_tH^3$ regularity of $\vec{v}$ is achieved by employing the elliptic estimates on the stationary Stokes system derived from \eqref{pe1} by replacing the third boundary condition with this boundary estimates on $v_3$ (see the proof of Lemma \ref{nl5}). Based on the results derived for the linear system, we construct approximation solutions for the nonlinear problem \eqref{e001}-\eqref{e003} by iteration and prove that such approximation solutions are uniformly bounded in the desired energy space in Section 5. The proof of Theorem \ref{t2} is given in Section 6 by first applying the contraction mapping theorem to the approximation sequence to derive the unique solution of \eqref{e001}-\eqref{e003} and then reversing the transformation defined in Section 3 to obtain the solution of \eqref{e01}-\eqref{e03}. The proof of Theorem \ref{t1} is also exhibited in Section 6.

\section{Transformation to an equilibrium domain}
Since we do not know how to solve system \eqref{e01}-\eqref{e03} locally-in-time on the moving domain $\Om_t$, following \cite{beale1984} we shall transform this system to one on the equilibrium domain $\Om=\{(x_1,x_2,y)\in \R^3:\  -b<y<0 \}$. For $t>0$, define $\theta(t): \Om\to \Om_t=\{(x_1,x_2,y)\in \R^3:\  -b<y<\eta(x_1,x_2,t) \}$ by
\be\label{aae1}
\theta(x_1,x_2,y,t):=(\theta_1,\theta_2,\theta_3)
(x_1,x_2,y,t)
=(x_1,x_2,\bar{\eta}+y(1+\bar{\eta}/b)),
\ee
where $\bar{\eta}$ is the harmonic extension of $\eta$.
Then
\be\label{aae2}
d\theta =\left(
\begin{array}{lll}
1\ \ \quad0\ \ \quad 0\\
0\ \ \quad 1\ \ \quad 0\\
\al\quad \ \beta\ \quad J
\end{array}
\right),
\quad (\xi_{ij})_{3\times 3}:=(d\theta)^{-1}
=\left(
\begin{array}{ccc}
\ \ 1\  \qquad  \qquad 0\  \qquad \qquad 0\\
\ \ 0\   \qquad  \qquad 1\  \qquad \qquad 0\\
-J^{-1}\al\quad \  -J^{-1}\beta\ \  \quad  \ \ \ \ \ J^{-1}
\end{array}
\right)
\ee
with
\be\label{aae18}
\al:=(1+y/b)\partial_1\bar{\eta},
\quad\
\beta:=(1+y/b)\partial_2\bar{\eta},
\quad\
J:=1+\bar{\eta}/b+\partial_3 \bar{\eta}(1+y/b).
\ee
For $(x_1,x_2,y,t)\in \Om\times(0,\infty)$ we define
\be\label{aae3}
\begin{split}
w(x_1,x_2,y,t)&=m(\theta(x_1,x_2,y,t),t),\quad
h(x_1,x_2,y,t)=\tilde{c}(\theta(x_1,x_2,y,t),t),\\
q(x_1,x_2,y,t)&=p(\theta(x_1,x_2,y,t),t),\quad
\phi(x_1,x_2,y,t)=\Phi(\theta(x_1,x_2,y,t),t).
\end{split}
\ee
The velocity field $\vec{v}(x_1,x_2,y,t)=(v_1,v_2,v_3)(x_1,x_2,y,t)$ on $\Om\times (0,\infty)$ is defined in the following way
\be\label{aae4}
u_i(\theta(x_1,x_2,y,t),t)=J^{-1}(\partial_j \theta_i)v_j(x_1,x_2,y,t),
\ee
that is,
\be\label{aae8}
\begin{split}
&u_1(\theta(x_1,x_2,y,t),t)=J^{-1}v_1(x_1,x_2,y,t),
\quad
u_2(\theta(x_1,x_2,y,t),t)=J^{-1}v_2(x_1,x_2,y,t),\\
&u_3(\theta(x_1,x_2,y,t),t)=J^{-1}\al v_1(x_1,x_2,y,t)
+J^{-1}\beta v_2(x_1,x_2,y,t)+v_3(x_1,x_2,y,t)
\end{split}
\ee
to preserve the divergence-free condition.
In particular, when $t=0$, the transformation $\theta(0): \Om\to \Om_0=\{(x_1,x_2,y)\in \R^3:\  -b<y<\eta_0(x_1,x_2)\}$ is defined as follows
 \be\label{ne52}
 \theta(x_1,x_2,y,0):=(x_1,x_2,\bar{\eta}_0+y(1+\bar{\eta}_0/b)),
 \ee
 where $\bar{\eta}_0$ is the harmonic extension of $\eta_0$.
 Corresponding to \eqref{aae18}-\eqref{aae4} we set
 \be\label{ne53}
 \al_0=(1+y/b)\partial_1\bar{\eta}_0,
\quad\
\beta_0=(1+y/b)\partial_2\bar{\eta}_0,
\quad\
J_0=1+\bar{\eta}_0/b+\partial_3 \bar{\eta}_0(1+y/b)
\ee
and
\be\label{ne56}
\begin{split}
&w_0(x_1,x_2,y)=m_0(\theta(x_1,x_2,y,0)),\quad
h_0(x_1,x_2,y)=\tilde{c}_0(\theta(x_1,x_2,y,0)),\\
&u_{0i}(\theta(x_1,x_2,y,0))=J^{-1}_0(\partial_j \theta_i(x_1,x_2,y,0))v_{0j}(x_1,x_2,y)
\end{split}
\ee
for $(x_1,x_2,y)\in \Om$.
 Then by the chain rule and direct computations one can deduce from \eqref{e01}-\eqref{e03} that
\be\label{e001}
\left\{
\begin{array}{lll}
w_t-\Delta w-\nabla\cdot(w\nabla h)=F_4(w,h,\vec{v},\bar{\eta}),\quad (x_1,x_2,y,t)\in \Omega\times(0,\infty),\\
h_t-\Delta h-w=F_5(h,\vec{v},\bar{\eta}),\\
\vec{v}_t-\Delta\vec{v}+\nabla q +w\nabla \phi=\vec{F}(w,\vec{v},\na q,\bar{\eta}),\\
\na\cdot\vec{v}=0,\\
(w,h,\vec{v})(x_1,x_2,y,0)=(w_0,h_0,\vec{v}_0)(x_1,x_2,y),\quad \eta(x_1,x_2,0)=\eta_0(x_1,x_2),
\end{array}
\right.
\ee
with the following boundary conditions on $\Gamma\times(0,\infty)$:
\be\label{e002}
\left\{
\begin{array}{lll}
\partial_3 w+w\partial_3 h=G_4(w,h,\bar{\eta}),\quad h=0,\\
\partial_3 v_1+\partial_1 v_3=G_1(\vec{v},\bar{\eta}),\quad
\partial_3 v_2+\partial_2 v_3=G_2(\vec{v},\bar{\eta}),\\
\eta_t=v_3,\quad
q-2\partial_3 v_3=\gamma\eta-\sigma\Delta_0\eta-G_3(\vec{v},\bar{\eta})\\
\end{array}
\right.
\ee
and the following boundary conditions on $S_B\times(0,\infty)$:
\be\label{e003}
w=0,\quad\partial_3 h=0,\quad\vec{v}=\mathbf{0}.
\ee
The detailed derivation of \eqref{e001}-\eqref{e003} is given in appendix with nonlinear terms $\vec{F}=(F_1,F_2,F_3)$, $\vec{G}=(G_1,G_2,G_3)$ and $F_4$, $F_5$, $G_4$ defined in \eqref{aae15}, \eqref{aae13}-\eqref{aae17} and \eqref{aae9}-\eqref{aae21}, respectively.

~\\

\section{Solvability of the linear system}
This section is devoted to proving the solvability of the  linearized version of \eqref{e001}-\eqref{e003}. We first split the corresponding linear system into the following two initial-boundary value problems:
\be\label{e2}
\left\{
\begin{array}{lll}
w_t-\Delta w-\nabla\cdot(a\nabla h)=f_4,\quad (x_1,x_2,y,t)\in \Omega\times(0,\infty),\\
h_t-\Delta h-w=f_5,\\
(w,h)(x_1,x_2,y,0)=(w_0,h_0)(x_1,x_2,y),\\
\partial_3 w+a\partial_3 h=g_4,\quad h=0\quad\text{on}\ \  \Gamma\times(0,\infty),\\
w=0,\quad\partial_3 h=0 \quad \text{on}\ \ S_B\times(0,\infty)
\end{array}
\right.
\ee
and
\be\label{e1}
\left\{
\begin{array}{lll}
\vec{v}_t-\Delta\vec{v}+\nabla q +w\nabla \phi=\vec{f}, \quad (x_1,x_2,y,t)\in \Omega\times(0,\infty),\\
\na \cdot\vec{v}=0,\\
\vec{v}(x_1,x_2,y,0)=\vec{v}_0(x_1,x_2,y),\quad \eta(x_1,x_2,0)=\eta_0(x_1,x_2),\\
\partial_3 v_1+\partial_1 v_3=g_1,\quad
\partial_3 v_2+\partial_2 v_3=g_2\quad \text{on}\ \  \Gamma\times(0,\infty),\\
\eta_t=v_3,\quad
q-2\partial_3 v_3=\gamma\eta-\sigma\Delta_0\eta-g_3\quad \text{on}\ \  \Gamma\times(0,\infty),\\
\vec{v}=\mathbf{0} \quad \text{on}\ \  S_B\times(0,\infty).
\end{array}
\right.
\ee
We introduce two notations for later use:
\ben
|||f|||:=\|f\|_{L^\infty_tH^2}+\|f_t\|_{L^\infty_tL^2}
+\|f\|_{L^2_tH^3}+\|f_t\|_{L^2_tH^1}
\enn
and
\be\label{ee00}
\begin{split}
\|\left\{
w,h,\vec{v},q,\eta\right\}\|
:=&|||w|||+|||h|||+|||\vec{v}|||
+\|\nabla\vec{v}_{t}\|_{L^2_t H^{-\frac{1}{2}}(\G)}\\
&+\|\nabla q\|_{L^\infty_t L^2}
+\|\nabla q\|_{L^2_t H^1}
+\|\nabla q_t\|_{L^2_t({_0}H^1)^{'}}\\
&+\|\eta\|_{L^\infty_tH^3(\R^2)}
+\|\na_0\eta\|_{L^2_tH^{\frac{5}{2}}(\R^2)}
+\|\na^2\mathcal{H}(\eta)\|_{L^2_tH^2}.
\end{split}
\ee

To solve system \eqref{e2}-\eqref{e1}, the initial and boundary data are required to satisfy the following compatibility conditions:
\be\label{cb1}
\left\{
\begin{array}{lll}
\partial_3 w_0+a(x_1,x_2,y,0)\partial_3 h_0=g_4(x_1,x_2,0),
\quad h_0=0\quad \text{on}\ \ \G,\\
w_0=0,\quad \partial_3h_0=0\quad \text{on}\ \ S_B
\end{array}
\right.
\ee
and
\be\label{cb2}
\left\{
\begin{array}{lll}
\partial_3 v_{01}+\partial_{1}v_{03}=g_1(x_1,x_2,0),
\quad \partial_3 v_{02}+\partial_{2}v_{03}=g_2(x_1,x_2,0)
\quad  \text{on}\ \ \G,\\
\na\cdot \vec{v}_0=0\quad \text{in}\quad \Om,\\
\vec{v}_0=\mathbf{0} \quad \text{on}\ \ S_B.
\end{array}
\right.
\ee
Then the solvability of the linear system \eqref{e2}-\eqref{e1} is as follows.
\begin{prop}\label{p1} Let $\vec{v}_0=(v_{01},v_{02},v_{03}),w_0,h_0\in H^2(\Om)$, $\eta_0\in H^{3}(\R^2)$ and
\be\label{ee01}
\left\{\begin{array}{lll}
\na \phi,\ \na^2 \phi\in L^\infty(0,\infty; L^\infty),\quad \na \phi_t\in L^2(0,\infty; L^2),\\
\vec{g},\ g_4\in
L^2(0,\infty;H^{\frac{3}{2}}(\Gamma)),\quad
\vec{g}_t,\ g_{4t}\in L^2(0,\infty;H^{-\frac{1}{2}}(\Gamma)),\\
\vec{f},\ f_4,\ f_5 \in L^2(0,\infty;H^{1}),\quad
\vec{f}_t,\ f_{4t}\in L^2(0,\infty;({_0}H^1)^{'}),\quad
f_{5t}\in L^2(0,\infty;({^0}H^1)^{'})
\end{array}
\right.
\ee
satisfy the compatibility conditions \eqref{cb1}-\eqref{cb2}. Assume further that the function $a$ fulfills
\be\label{eee0}
C_1(C_2+1) |||a|||^2\leq \frac{1}{2}\qquad \text{for}\ \ \text{all} \ \  t> 0,
\ee
where the constants $C_1$ and $C_2$ are independent of t, given in \eqref{ce1} and $\eqref{ce14}.$
Then system \eqref{e2}-\eqref{e1} admits a unique global solution $(w,h,\vec{v},q,\eta)$ such that
\be\label{ee78}
\begin{split}
\|&\left\{
w,h,\vec{v},q,\eta\right\}\|^2\\
&\leq  C\big(\|w_0\|_{H^2(\Om)}^2+
\|h_0\|_{H^2(\Om)}^2+\|\vec{v}_0\|_{H^2(\Om)}^2
+\|\eta_0\|_{H^3(\R^2)}^2
\big)\\
&\ \quad +C\big(
\|f_{4}\|_{L^2_tH^{1}}^2
+\|f_{4t}\|_{L^2_t({_0}H^1)^{'}}^2
+\|f_{5}\|_{L^2_tH^{1}}^2
+\|f_{5t}\|_{L^2_t({^0}H^1)^{'}}^2
+\|g_4\|_{L^2_tH^{\frac{3}{2}}(\Gamma)} ^2
\big)\\
&\ \quad +C\big(\|g_{4t}\|_{L^2_tH^{-\frac{1}{2}}(\Gamma)}^2
+\|\vec{f}\|_{L^2_tH^1}^2
+\|\vec{f}_{t}\|_{L^2_t ({_0}H^1)^{'}}^2
+\|\vec{g}\|_{L^2_t H^{\frac{3}{2}}(\G)}^2
+\|\vec{g}_t\|_{L^2_tH^{-\frac{1}{2}}(\G)}^2
\big)
\end{split}
\ee
for all $t>0$, where the constant $C$ is independent of $t$.
\end{prop}

The remaining part of this section is organized as follows. In next subsection, we shall introduce some preliminaries for later use. The solvability of subsystem \eqref{e2} and
 subsystem \eqref{e1} will be established in subsection \ref{ss4} and subsection \ref{ss6}, respectively. Subsection \ref{ss3} is to proving Proposition \ref{p1}.
\subsection{Preliminaries}

 Noting the divergence-free condition $\na\cdot\vec{v}=0$ and the gradient form $\na q$ of the pressure in \eqref{e1}, and using the following identity
\be\label{ee98}
\int_{\Om} \vec{v}\cdot \na q\,dxdy
=-\int_{\Om}(\na\cdot\vec{v})q\,dxdy
+\int_{\partial \Om}(\vec{v}\cdot \vec{n})q\,dx
\ee
one can see that the vectors $\vec{v}$ and $\na q$ are $L^2$-orthogonal if they satisfy the boundary condition
$
(\vec{v}\cdot\vec{n})q=0$ on $\partial \Om$.
This observation has been used in treating equations of incompressible fluids in a fixed domain to remove the pressure as an unknown by projecting the equation onto a subspace of vector fields of divergence-free that satisfy the same boundary conditions as the velocity (see e.g. \cite{fujita-kato1964, temam1975}). In the present case, since $\vec{v}=\mathbf{0}$ on $S_B$ and $\na\cdot\vec{v}=0$ in $\Om$ it follows from \eqref{ee98} that a vector in the gradient form $\na \rho$ is $L^2$-orthogonal to $\vec{v}$ if and only if $\rho=0$ on $\G$. With this in mind, we introduce the projection $P$ on $L^2(\Om)$ orthogonal to
\be\label{ee99}
W:=\{\na \rho :\ \rho\in H^1(\Om),\ \  \rho  =0\ \ \text{on}\ \ \G
\}.
\ee
For this projection $P$, the following property has been proved in \cite{beale1984}.
\begin{lemma}(\cite[Lemma 2.1]{beale1984})\label{pl8}
Let $P$ be the projection on $L^2(\Om)$ orthogonal to the subspace $W$ defined in \eqref{ee99}.
Then $P$ is a bounded operator on $L^2(\Om)$ and on $H^k(\Om)$ with $k\geq 1$. In particular, for $k\geq 1$
\ben
PH^k(\Om)=\{\vec{v}\in H^k(\Om):\ \na\cdot \vec{v}=0\ \text{in}\ \Om,\ \ \vec{v}\cdot\vec{n}=0\  \text{on}\  S_B\}.
\enn
Moreover, if $\ti{\xi}(x_1,x_2,y)\in H^1(\Om)$ and $\xi(x_1,x_2)\in H^{\frac{1}{2}}(\R^2)$ satisfy $\ti{\xi}(x_1,x_2,0)=\xi(x_1,x_2)$. Then the following holds true
\ben
 P(\na\tilde{\xi})=\na \mathcal{H}(\xi).
\enn
\end{lemma}

We shall use the following version of Korn's inequality, which has been justified in \cite{beale1981}.
\begin{lemma}(\text{\cite[Lemma 2.7]{beale1981}})\label{pl5}
Let $\vec{v}\in{_0}H^1(\Om)$. Then there exists a constant $C>0$ such that
\ben
\|\vec{v}\|_{H^1}^2\leq C[\vec{v},\vec{v}].
\enn
\end{lemma}

\begin{lemma}\label{pl7}
Let $\bar{\eta}(x_1,x_2,y)=\mathcal{H}(\eta)$ be the harmonic extension of $\eta(x_1,x_2)$.
Then for any integer $m\geq 1$, the following holds true
\be\label{eee13}
\|\bar{\eta}\|_{H^m}\leq C\|\eta\|_{H^{m-\frac{1}{2}}(\R^2)},
\ee
with $C>0$.
\end{lemma}
\begin{proof}
For fixed $m\geq 1$,
let $\psi(x_1,x_2,y)$ be an extension of $\eta(x_1,x_2)$ on $\R^3_{-}=\{(x_1,x_2,y)\in \R^3:\  y<0\}$ (see e.g. \cite[Chapter 1, Theorem 8.3]{lions&magenes}) satisfying that
$\psi(x_1,x_2,0)=\eta(x_1,x_2)$, $\psi(x_1,x_2,y)=0$ for $y\leq -\frac{1}{2}$
and
\be\label{eee12}
\|\psi\|_{H^m(\R^3_{-})}\leq C \|\eta\|_{H^{m-\frac{1}{2}}(\R^2)}.
\ee
Define
$\var(x_1,x_2,y)=\bar{\eta}(x_1,x_2,y)-\psi(x_1,x_2,y)$ for $(x_1,x_2,y)\in \Om$. Then we deduce from \eqref{eee9} that $\var$ solves
\be\label{eee10}
\left\{
\begin{array}{lll}
\De \var=-\De\psi\quad \ \ \text{in}\,\,\,\Omega,
\\
\var=0\quad \ \ \quad\quad \ \text{on}\,\,\, \G,\\
\partial_3\var=0\quad\ \ \quad \ \,\text{on}\,\,\,S_B.
\end{array}
\right.
\ee
Applying the standard elliptic theory (see e.g. \cite[Lemma 2.8]{beale1981}) to system \eqref{eee10} and using \eqref{eee12} one gets
\be\label{eee11}
\begin{split}
\|\var\|_{H^m}\leq C\|\psi\|_{H^m}
\leq  C \|\eta\|_{H^{m-\frac{1}{2}}(\R^2)},
\end{split}
\ee
which, along with \eqref{eee12} leads to \eqref{eee13}.
The proof is completed.

\end{proof}

When solving \eqref{e1}, we shall employ the following elliptic system
\be\label{pe1}
\left\{
\begin{array}{lll}
-\Delta\vec{\omega}+\nabla q=\tilde{F} \quad \text{in}\ \  \Omega,\\
\nabla\cdot \vec{\omega}=0\quad \text{in}\ \  \Omega,\\
\partial_3 \omega_1+\partial_1 \omega_3=\tilde{g}_1,\quad
\partial_3 \omega_2+\partial_2 \omega_3=\tilde{g}_2\quad \text{on}\ \  \Gamma,\\
q_1-2\partial_3^2 \omega_3=\ti{g}_3\quad \text{on}\ \ \Gamma,\\
\vec{\omega}=\mathbf{0} \quad \text{on}\ \ S_B.
\end{array}
\right.
\ee
For system \eqref{pe1} we have

\begin{lemma}\label{pl6} Let $m\geq 0$.
 Suppose $\tilde{g}=(\tilde{g}_1,\tilde{g}_2,\tilde{g}_3)\in H^{m-\frac{3}{2}}(\G)$ and $\tilde{F}=(\tilde{F}_1,\tilde{F}_2,\tilde{F}_3)\in H^{m-2}(\Om)$. Then \eqref{pe1} admits a unique solution $(\vec{\omega},q)$ with $\vec{\omega}=(\omega_1,\omega_2,\omega_3)$ such that
\ben
\|\vec{\omega}\|_{H^m}+\|q\|_{H^{m-1}}\leq C(\|\tilde{F}\|_{H^{m-2}}+\|\tilde{g}\|_{H^{m-\frac{3}{2}}(\G)}),
\enn
for some constant $C>0$.
\end{lemma}
The proof of Lemma \ref{pl6} is quite similar to that of \cite[Lemma 3.3]{beale1981}, thus we omit the details and refer the reader to \cite[Lemma 3.3]{beale1981}.

\subsection{Solvability of system \eqref{e2}}\label{ss4}
\begin{prop}\label{cp1}
Let $(w_0,h_0)\in H^2(\Om),$
$f_4$, $f_5$, $g_4$ and $a$ satisfy \eqref{ee01}-\eqref{eee0} and the compatibility conditions \eqref{cb1}.
Then there exists a unique solution $(w,h)$ to system \eqref{e2} fulfilling
\be\label{ce19}
\begin{split}
|||w|||^2+|||h|||^2\leq& C_3\big(\|w_0\|_{H^2(\Om)}^2+\|h_0\|_{H^2(\Om)}^2\big)
+C_3\big(
\|f_{4}\|_{L^2_tH^{1}}^2
+\|f_{4t}\|_{L^2_t({_0}H^1)^{'}}^2
\big)\\
&+C_3\big(
\|f_{5}\|_{L^2_tH^{1}}^2
+\|f_{5t}\|_{L^2_t({^0}H^1)^{'}}^2
\big)
+C_3\big(\|g_4\|_{L^2_tH^{\frac{3}{2}}(\Gamma)} ^2
+\|g_{4t}\|_{L^2_tH^{-\frac{1}{2}}(\Gamma)}^2
\big)
\end{split}
\ee
for all $t> 0$, where the constant $C_3>0$ is independent of $t$.
\end{prop}
The proof of Proposition \ref{cp1} is based on the following two lemmas where the \emph{a priori} estimates on solutions $(w,h)$ are derived.
\begin{lemma}\label{cl1} Let the assumptions in Proposition \ref{cp1} hold. Then the solution $(w,h)$ of system \eqref{e2} satisfies
\be\label{ce1}
\begin{split}
|||w|||^2\leq& C_1\|w_0\|_{H^2(\Om)}^2+ C_1 |||a|||^2 |||h|||^2
+C_1\big(\|g_4\|_{L^2_tH^{\frac{3}{2}}(\Gamma)} ^2
+\|g_{4t}\|_{L^2_tH^{-\frac{1}{2}}(\Gamma)}^2
\big)\\
&+C_1\big(
\|f_{4}\|_{L^2_tH^{1}}^2
+\|f_{4t}\|_{L^2_t({_0}H^1)^{'}}^2
\big)
\end{split}
\ee
for all $t>0$, where the constant $C_1>0$ is independent of $t$.
\end{lemma}
\begin{proof}
We first estimate $\|w_t\|_{L^\infty_tL^2}^2+\|\nabla w_t\|_{L^2_tL^2}^2$. Differentiating the first equation of \eqref{e2} with respect to $t$ one has
\be\label{ce4}
w_{tt}-\Delta w_t-\nabla\cdot(a\nabla h)_t=f_{4t}.
\ee
Multiplying \eqref{ce4} by $2w_t$ in $L^2$ and using integration by parts, we get
\be\label{ce5}
\begin{split}
\frac{d}{dt}\|w_t\|_{L^2}^2+2\|\nabla w_t\|_{L^2}^2
=&2\int_\Gamma(\partial_3 w+a\partial_3 h)_tw_t\,dx
-2\int_\Omega (a\nabla h)_t\cdot\nabla w_t\,dxdy
+2\int_\Omega f_{4t}w_t\,dxdy\\
=&2\int_\Gamma g_{4t}w_t\,dx
-2\int_\Omega (a\nabla h)_t\cdot\nabla w_t\,dxdy
+2\int_\Omega f_{4t}w_t\,dxdy\\
:=&I_1+I_2+I_3.
\end{split}
\ee
By the trace theorem we have
\be\label{ce6}
I_1\leq\va\|w_t\|_{H^{\frac{1}{2}}(\Gamma)}^2
+C(\va)\|g_{4t}\|_{H^{-\frac{1}{2}}(\Gamma)}^2
\leq C\va\|w_t\|_{H^{1}}^2
+C(\va)\|g_{4t}\|_{H^{-\frac{1}{2}}(\Gamma)}^2,
\ee
where $C(\va)$ is a constant independent of $t$, but depending on $\va$.
Then using the Poincar\'{e} inequality (\cite[Theorem 6.30]{adams-fournier}) thanks to the fact $w_t=0$ on $S_B$, we can choose $\va$ small enough such that $C\va\|w_t\|_{H^{1}}^2\leq \frac{1}{3}\|\nabla w_t\|_{L^2}^2$, which inserted into \eqref{ce6} gives rise to
\be\label{ce7}
I_1\leq \frac{1}{3}\|\nabla w_t\|_{L^2}^2+C\|g_{4t}\|_{H^{-\frac{1}{2}}(\Gamma)}^2.
\ee
It follows from the Sobolev embedding inequality that
\be\label{ce8}
\begin{split}
I_2
\leq& \frac{1}{3} \|\nabla w_t\|_{L^2}^2
+C\left(\|a\|_{L^\infty}^2\|\nabla h_t\|_{L^2}^2
+\|a_t\|_{L^4}^2\|\nabla h\|_{L^4}^2\right)\\
\leq &\frac{1}{3} \|\nabla w_t\|_{L^2}^2
+C\left(\|a\|_{H^2}^2\|\nabla h_t\|_{L^2}^2
+\|a_t\|_{H^1}^2\|h\|_{H^2}^2\right).
\end{split}
\ee
Noting $w_t=0$ on $S_B$, one employs the Poincar\'{e} inequality to derive
\be\label{ce9}
I_3\leq \va \|w_t\|_{H^1}^2+C(\va)\|f_{4t}\|_{({_0}H^1)^{'}}^2
\leq \frac{1}{3}\|\nabla w_t\|_{L^2}^2+C\|f_{4t}\|_{({_0}H^1)^{'}}^2,
\ee
where $\va$ has been chosen small such that $\va \|w_t\|_{H^1}^2\leq \frac{1}{3}\|\nabla w_t\|_{L^2}^2$. Substituting \eqref{ce7}-\eqref{ce9} into \eqref{ce5} and integrating the resulting inequality over $(0,t)$ we arrive at
 \be\label{ce21}
 \begin{split}
 \|w_t\|_{L^\infty_tL^2}^2+\|\nabla w_t\|_{L^2_tL^2}^2
\leq \|w_t(0)\|_{L^2}^2+ C\big(\|g_{4t}\|_{L^2_tH^{-\frac{1}{2}}(\Gamma)}^2
 +|||a|||^2||| h|||^2
+\|f_{4t}\|_{L^2_t({_0}H^1)^{'}}^2
\big).
\end{split}
\ee
The estimate of $\|w_t(0)\|_{L^2}^2$ follows from the first equation of \eqref{e2} and the compactness theorem (see e.g. \cite[Chapter 1, Theorem 3.1]{lions&magenes}):
\ben
\begin{split}
\|w_t(0)\|_{L^2}^2\leq &   C\big(\|w_0\|_{H^2}^2+\|a(0)\|_{H^2}^2 \|h(0)\|_{H^2}^2
+\|f_4(0)\|_{L^2}^2\big)\\
\leq & C\|w_0\|_{H^2}^2+C|||a|||^2\,|||h|||^2
+C\big(\|f_4\|_{L^2_tH^1}^2+\|f_{4t}\|_{L^2_t(_{0}H^1)^{'}}^2
\big)
\end{split}
\enn
which, substituted into \eqref{ce21} gives rise to
\be\label{ce10}
 \begin{split}
 \|w_t&\|_{L^\infty_tL^2}^2+\|w_t\|_{L^2_tH^1}^2\\
\leq &C\|w_0\|_{H^2}^2+ C\big(
 |||a|||^2\,|||h|||^2
+\|g_{4t}\|_{L^2_tH^{-\frac{1}{2}}(\Gamma)}^2
+\|f_4\|_{L^2_tH^1}^2
+\|f_{4t}\|_{L^2_t({_0}H^1)^{'}}^2
\big),
\end{split}
\ee
where we have used the fact $\|w_t\|_{L^2_tH^1}\leq C \|\na w_t\|_{L^2_tL^2}$, thanks to the Poincar\'{e} inequality (cf. \cite[Theorem 6.30]{adams-fournier}).
We proceed to estimating $\|w\|_{L^2_tH^3}^2.$
 For fixed $t>0$, from \eqref{e2} we know that $w$ solves the following elliptic system
 \be\label{ce11}
 \left\{
\begin{array}{lll}
-\Delta w=-w_t+\nabla\cdot(a\nabla h)+f_4,\quad \ \ (x_1,x_2,y)\in \Omega,\\
\partial_3 w=-a\partial_3 h+g_4\quad \text{on}\ \  \Gamma,\\
w=0 \quad \text{on}\ \  S_B.
\end{array}
\right.
\ee
Then it follows from the standard elliptic theory (see e.g. \cite[Lemma 2.8]{beale1981}) that
\ben
\|w\|_{H^k}\leq C\big(\|-w_t+\nabla\cdot(a\nabla h)+f_4\|_{H^{k-2}}+\|a\partial_3 h\|_{H^{k-\frac{3}{2}}(\Gamma)} +\|g_4\|_{H^{k-\frac{3}{2}}(\Gamma)}\big)
\enn
for fixed $t>0$, with $k\geq 2$. Taking $k=3$, one immediately gets
\be\label{ce12}
\begin{aligned}
\|w\|_{L^2_tH^3}^2
\leq& C\big(
\|w_t\|_{L^2_tH^{1}}^2+\|\nabla\cdot(a\nabla h)\|_{L^2_tH^{1}}^2
+\|f_4\|_{L^2_tH^{1}}^2
+\|a\partial_3 h\|_{L^2_tH^{\frac{3}{2}}(\Gamma)} ^2
+\|g_4\|_{L^2_tH^{\frac{3}{2}}(\Gamma)} ^2\big )\\
\leq & C\big(
\|w_t\|_{L^2_tH^1}^2
+\|a\|_{L^\infty_tH^2}^2\|h\|_{L^2_tH^3}^2
+\|f_{4}\|_{L^2_tH^{1}}^2
+\|g_4\|_{L^2_tH^{\frac{3}{2}}(\Gamma)} ^2
\big),
\end{aligned}
\ee
where we have used the following fact:
\ben
\begin{split}
\|\nabla\cdot(a\nabla h)\|_{L^2_tH^1}
+\|a\partial_3 h\|_{L^2_tH^{\frac{3}{2}}(\Gamma)} ^2
\leq  C\|a\nabla h\|_{L^2_tH^2}
\leq C\|a\|_{L^\infty_tH^2}\|h\|_{L^2_tH^3},
\end{split}
\enn
thanks to the trace theorem and Sobolev embedding inequality. On the other hand, it follows from the compactness theorem (cf. \cite[Chapter 1, Theorem 3.1]{lions&magenes}) that
\be\label{ce22}
\|w\|_{L^\infty_tH^2}\leq C(\|w_t\|_{L^2_tH^{1}}^2+\|w\|_{L^2_tH^{3}}^2).
\ee
Collecting \eqref{ce12}, \eqref{ce22} and using \eqref{ce10} we obtain \eqref{ce1}.
The proof is completed.

\end{proof}

\begin{lemma}\label{cl2} Suppose that the assumptions in Proposition \ref{cp1} hold. Then the solution $(w,h)$ of \eqref{e2} fulfills
\be\label{ce14}
|||h|||^2\leq C_2(\|h_0\|_{H^2(\Om)}^2+\|w_0\|_{H^2(\Om)}^2)+C_2|||w|||^2
+
C_2\big(
\|f_{5}\|_{L^2_tH^{1}}^2
+\|f_{5t}\|_{L^2_t({^0}H^1)^{'}}^2
\big)
\ee
for all $t>0$, where the constant $C_2>0$ is independent of $t$.
\end{lemma}

\begin{proof}
 Differentiating the second equation of \eqref{e2} with respect to $t$, then multiplying the resulting equation with $2h_t$ in $L^2$ and using integration by parts we have
 \be\label{ce15}
 \begin{split}
\frac{d}{dt}\|h_t\|_{L^2}^2+2\|\nabla h_t\|_{L^2}^2
=&2\int_{\Omega}f_{5t}h_t\,dxdy
+2\int_\Omega w_t h_t\,dxdy\\
\leq & \va \|h_t\|^2_{H^1}+C(\va)\big(\|f_{5t}\|_{({^0}H^1)^{'}}^2
+\|w_t\|_{L^2}^2\big),
\end{split}
 \ee
 where the positive constant $C(\va)$ depends on $\va>0$.
 Since $h_t=0$ on $\Gamma\times(0,\infty)$,
by the Poincar\'{e} inequality (cf. \cite[Themorem 6.30]{adams-fournier}) one can choose $\va$ small such that
  \be\label{ce16}
  \va \|h_t\|^2_{H^1}\leq \|\nabla h_t\|_{L^2}^2.
  \ee
  Inserting \eqref{ce16} into \eqref{ce15} and integrating the resulting inequality over $(0,t)$ to derive
 \be\label{ce18}
 \|h_t\|_{L^\infty_tL^2}^2+\|h_t\|_{L^2_tH^1}^2
 \leq \|h_t(0)\|_{L^2}^2+C(|||w|||^2+
\|f_{5t}\|_{L^2_t(^{0}H^1)^{'}}^2).
 \ee
We next use the second equation of \eqref{e2} and the compactness theorem (cf. \cite[Chapter 1, Theorem 3.1]{lions&magenes}) to estimate $\|h_t(0)\|_{L^2}^2$ as follows:
\ben
\begin{split}
\|h_t(0)\|_{L^2}^2\leq C(\|h_0\|_{H^2}^2+\|w_0\|_{L^2}^2+\|f_5(0)\|_{L^2}^2)
\leq C
(\|h_0\|_{H^2}^2+\|w_0\|_{L^2}^2+\|f_5\|_{L^2_tH^1}^2+
\|f_{5t}\|_{L^2_t(^{0}H^1)^{'}}^2),
\end{split}
\enn
which, substituted in to \eqref{ce18} gives rise to
\be\label{ce23}
\|h_t\|_{L^\infty_tL^2}^2+\|h_t\|_{L^2_tH^1}^2
\leq C
(\|h_0\|_{H^2}^2+\|w_0\|_{H^2}^2
+|||w|||^2+\|f_5\|_{L^2_tH^1}^2+
\|f_{5t}\|_{L^2_t(^{0}H^1)^{'}}^2).
\ee
 From \eqref{e2} we know that for fixed $t>0$, $h$ solves the following elliptic problem
 \be\label{ce17}
 \left\{
\begin{array}{lll}
-\Delta h=f_5-h_t+w, \quad  (x_1,x_2,y) \in\Omega,\\
 h=0\quad\quad\text{on}\,\, \Gamma,\\
\partial_3 h=0 \ \quad \text{on}\,\, S_B.
\end{array}
\right.
 \ee
 Then applying the standard elliptic theory (cf. \cite[Lemma 2.8]{beale1981}) to \eqref{ce17}  one deduces that
 \be\label{ce24}
 \|h\|_{L^2_tH^3}^2
 \leq C (\|f_5\|_{L^2_tH^1}^2
 +\|h_t\|_{L^2_tH^1}^2
 +\|w\|_{L^2_tH^1}^2
 ).
 \ee
 Moreover, it follows from the compactness theorem (cf. \cite[Chapter 1, Theorem 3.1]{lions&magenes}) that
 \ben
 \|h\|_{L^\infty_tH^2}^2\leq C(\|h\|_{L^2_tH^3}^2+\|h_t\|_{L^2_tH^1}^2),
   \enn
   which, along with \eqref{ce23} and \eqref{ce24}
   gives \eqref{ce14}.
 The proof is completed.

\end{proof}

We next prove Proposition \ref{cp1} by using Lemma \ref{cl1} and Lemma \ref{cl2}.\\
\textbf{\emph{Proof of Proposition \ref{cp1}}.} The local well-posedness of the initial-boundary value problem \eqref{e2} follows from the standard parabolic theory (see e.g. \cite[pp 341-342, Theorem 9.1]{LSU88}). We omit the details for brevity and proceed to the derivation of \eqref{ce19}.
Multiplying \eqref{ce1} by $(C_2+1)$ then adding the resulting inequality to \eqref{ce14} one gets
\ben
\begin{split}
|||w|||^2+|||h|||^2
\leq &(C_1+1)(C_2+1)\|w_0\|_{H^2(\Om)}^2+C_2\|h_0\|_{H^2(\Om)}^2+
 C_1(C_2+1) |||a|||^2 |||h|||^2\\
&+C_1(C_2+1)\big(\|g_4\|_{L^2_tH^{\frac{3}{2}}(\Gamma)} ^2
+\|g_{4t}\|_{L^2_tH^{-\frac{1}{2}}(\Gamma)}^2
\big)\\
&+C_1(C_2+1)\big(
\|f_{4}\|_{L^2_tH^{1}}^2
+\|f_{4t}\|_{L^2_t({_0}H^1)^{'}}^2
\big)\\
&+C_2\big(
\|f_{5}\|_{L^2_tH^{1}}^2
+\|f_{5t}\|_{L^2_t({^0}H^1)^{'}}^2
\big)
\end{split}
\enn
which, in conjunction with \eqref{eee0} gives \eqref{ce19}. Moreover, from \eqref{ce19} we know that the solution $(w,h)$ is unique and global.
 The proof is finished.

\endProof

\subsection{Solvability of system \eqref{e1}}\label{ss6}

\begin{prop}\label{np1} Suppose $\vec{v}_0\in H^2(\Om)$, $\eta_0\in H^{3}(\R^2)$, $\na \phi$, $\vec{f}=(f_1,f_2,f_3)$ and $\vec{g}=(g_1,g_2,g_3)$ satisfy \eqref{ee01} and the compatibility conditions \eqref{cb2}. Assume further that
\be\label{ab1}
w\in L^\infty(0,\infty;H^2)\cap L^2(0,\infty;H^3),\quad
w_t\in L^\infty(0,\infty;L^2)\cap L^2(0,\infty;H^1).
\ee
Then system \eqref{e1} admits a unique global solution $(\vec{v}, q,\eta)$ fulfilling
\be\label{ne19}
\begin{split}
|||\vec{v}&|||^2
+\|\nabla\vec{v}_{t}\|_{L^2_t H^{-\frac{1}{2}}(\G)}^2
+\|\na q\|_{L^\infty_t L^2}^2
+\| \na q\|_{L^2_t H^1}^2+\|\nabla q_t\|_{L^2_t({_0}H^1)^{'}}^2\\
&+\|\eta\|_{L^\infty_tH^3(\R^2)}^2
+\|\na_0\eta\|_{L^2_tH^{\frac{5}{2}}(\R^2)}
+\|\na^2\mathcal{H}(\eta)\|_{L^2_tH^2}
\\
\leq  & C_4 \big(\|\vec{v}_0\|_{H^2(\Om)}^2
+\|\eta_0\|_{H^3(\R^2)}^2\big)
+C_4|||w|||^2\big(\|\nabla \phi\|_{L^\infty_tL^\infty}^2
+\|\nabla^2 \phi\|_{L^\infty_tL^\infty}^2
+\|\nabla \phi_t\|_{L^2_tL^2}^2\big)
\\
&+C_4\big(
\|\vec{f}\|_{L^2_tH^1}^2
+\|\vec{f}_{t}\|_{L^2_t (_{0}H^1)^{'}}^2
+\|\vec{g}\|_{L^2_t H^{\frac{3}{2}}(\G)}^2
+\|\vec{g}_t\|_{L^2_tH^{-\frac{1}{2}}(\G)}^2
\big)
\end{split}
\ee
for all $t>0$,
where $C_4$ is a constant independent of $t$.
\end{prop}

We shall first prove the existence of weak solutions of \eqref{e1} in the following Lemma \ref{nl1}, and then improve the regularity of such weak solutions in Lemma \ref{nl2}- Lemma \ref{pl9} under the regularity assumptions on initial data and external forces in Proposition \ref{np1}. The proof of Proposition \ref{np1} will be given at the end of this subsection.
We introduce a space where weak solutions will be defined.
\ben
\begin{split}
V=
\{
\vec{v}\in L^2(0,\infty;H^1)\,|\ \ \nabla\cdot \vec{v}=0,
\quad \int_0^t \nabla_0 \vec{v}(x_1,x_2,y,s)\,ds\in L^\infty(0,\infty;L^2(\Gamma)),
\quad \vec{v}=\mathbf{0} \ \ \text{on}\ \  S_B\times(0,\infty)
\}.
\end{split}
\enn
For test functions, we introduce the following separable space:
\ben
\begin{split}
\mathcal{V}=\{
\vec{\varphi}\in H^1(\Om)\,| \ \  \nabla\cdot \vec{\var}=0,
\quad \nabla_0 \vec{\var}\in L^2(\Gamma),
\quad \vec{\var}=\mathbf{0} \ \ \text{on}\ \  S_B
\}.
\end{split}
\enn
Then the existence of weak solutions to system \eqref{e1} is as follows.
\begin{lemma}\label{nl1}
 Let $\eta_0\in H^1(\R^2)$ and $\vec{v}_0\in H^1(\Om)$ satisfy $\nabla\cdot\vec{v}_0=0$. Assume $\vec{g}\in L^2(0,\infty;H^{-\frac{1}{2}}(\G))$, $\vec{f}\in L^2(0,\infty;(_0H^1)^{'})$ and $(w\nabla \phi)\in L^2(0,\infty;L^2) $. Then there exists a weak solution $\vec{v}\in V$ of \eqref{e1} with $\vec{v}_t\in L^2(0,\infty;( _0H^1)^{'})$ satisfying
\be\label{a5}
\begin{split}
\langle\vec{v}_t&,\vec{\var}\rangle +[\vec{v},\vec{\var}]+\gamma\int_{\Gamma}\Big(\eta_0+\int_0^t v_3\, ds\Big)\var_3\,dx+\sigma\int_{\Gamma}\Big(\na_0\eta_0+\int_0^t\nabla_0 v_3\, ds\Big)\cdot\nabla_0\var_3\,dx\\
=&-(w\nabla \phi,\vec{\var})+
\langle\vec{f},\vec{\var}\rangle
+\langle \vec{g},\vec{\var}\rangle_{\Gamma},\qquad \forall \ \vec{\var}\in \mathcal{V}
\end{split}
\ee
and
\be\label{a6}
\vec{v}(x_1,x_2,y,0)=\vec{v}_0(x_1,x_2,y)
\ee
for almost every $t>0$, where $(\cdot,\cdot)$ denotes the $L^2$ inner product, $\langle\cdot,\cdot\rangle$ denotes the duality pairing between $_0H^1$ and $(_0H^1)^{'}$ and $\langle \cdot,\cdot\rangle_{\Gamma}$ represents the duality pairing between $H^{\frac{1}{2}}(\G)$ and $H^{-\frac{1}{2}}(\G)$ .
\end{lemma}
The weak formula \eqref{a5} is formally derived  by multiplying the first equation of \eqref{e1} with $\vec{\var}\in \mathcal{V}$ in $L^2$ and using the following two facts:
\ben
\int_\Omega (-\Delta \vec{v}+\nabla q)\cdot \vec{\var}\,dxdy
=\int_{\G}q \var_3\,dx
-\int_{\G} (\partial_i v_3+\partial_3 v_i)\var_{i}\,dx
+[\vec{v},\vec{\var}]
\enn
and
\be\label{ne42}
\eta(x_1,x_2,t)=\eta_0(x_1,x_2)+\int_0^t v_3(x_1,x_2,0,s)\,ds
 \ee
 thanks to $\eta_t=v_3$ on $\G\times (0,\infty)$.

One can prove Lemma \ref{nl1} by first using the \emph{Galerkin's method} (see e.g. \cite[page 377]{Evans}) to construct approximating solutions and then passing to limits.  Indeed, since $\mathcal{V}$ is separable (see \cite[Lemma 4.1]{coutand-shkoller2003}), there exists a basis $\{\vec{\var}^{k}\}_{k\in \N_{+}}$ of $\mathcal{V}$. Define the approximating solutions $\vec{v}^{\,m}:[0,\infty)\rightarrow \mathcal{V}$ as follows
\ben
\vec{v}^{\,m}(x_1,x_2,y,t)
=\sum_{k=1}^{\,m}\lambda_m^k(t)\vec{\var}^k(x_1,x_2,y).
\enn
By replacing the $\vec{v}$ and $\vec{\var}$
in \eqref{a5} with $\vec{v}^{\,m}$ and $\vec{\var}^k$ respectively, and using the orthogonality of the basis $\{\vec{\var}^{k}\}_{k\in \N_{+}}$ in both $\mathcal{V}$ and $L^2$ one derives a system of second-order ordinary differential equations on the unknown $\lambda_m^k(t)$. Solving these equations, we obtain the approximating solutions $\vec{v}^{\,m}\in L^\infty(0,\infty;L^2)\cap L^2(0,\infty;H^1)$. Then passing $m\rightarrow \infty$, it follows that the limit function $\vec{v}\in V$ fulfills the weak formula \eqref{a5} and initial condition \eqref{a6}. The procedure is standard, we thus omit the proof of Lemma \ref{nl1} and refer the reader to \cite{bae2011,coutand-shkoller2003} for details. We proceed to improving the regularity of the weak solutions.

\begin{lemma}\label{nl2} Let the assumptions in Proposition \ref{np1} hold true. Then the weak solutions $\vec{v}$ derived in Lemma \ref{nl1} and the $\eta$ defined in \eqref{ne42} satisfy
\be\label{ne13}
\begin{split}
\sum_{ k=0}^{2}&\left(\|\na_0^k\vec{v}\|_{L^\infty_tL^2}^2+
\|\na_0^k\vec{v}\|_{L^2_tH^1}^2
\right)
+\|\eta\|_{L^\infty_tH^3(\R^2)}^2\\
\leq& C(\|\vec{v}_0\|^2_{H^2(\Om)}+\|\eta_0\|_{H^3(\R^2)}^2)
+C(\|\vec{g}\|_{L^2_tH^{\frac{3}{2}}(\G)}^2+
\|\vec{f}\|_{L^2_tH^1}^2)\\
&+C\|w\|_{L^2_t H^2}^2(\|\nabla\phi\|^2_{L^\infty_tL^\infty}
+\|\nabla^2\phi\|^2_{L^\infty_tL^\infty}
)
\end{split}
\ee
for all $t>0$, with the constant $C$ independent of $t$.
\end{lemma}

\begin{proof}
Replacing $\vec{\var}$ in \eqref{a5} with $\vec{v}$ and using \eqref{ne42} one derives
\be\label{ne1}
\begin{split}
\frac{1}{2}&\frac{d}{dt}\|\vec{v}\|_{L^2}^2+[\vec{v},\vec{v}]
+\frac{1}{2}\frac{d}{dt}\Big(\gamma\|\eta\|^2_{L^2(\R^2)}
+\sigma\|\nabla_0\eta\|^2_{L^2(\R^2)}\Big)\\
&=\int_{\G}\vec{g}\cdot \vec{v}\,dx
+\int_{\Omega} \vec{f}\cdot\vec{v}\,dxdy
-\int_{\Omega}w\nabla\phi\cdot \vec{v}\,dxdy.
\end{split}
\ee
  Terms on the right-hand side of \eqref{ne1} can be estimated by the trace theorem and Sobolev embedding inequality as follows:
\be\label{ne3}
\begin{split}
\int_{\G}\vec{g}&\cdot \vec{v}\,dx
+\int_{\Omega} \vec{f}\cdot\vec{v}\,dxdy
-\int_{\Omega}w\nabla\phi\cdot \vec{v}\,dxdy\\
\leq& \va (\|\vec{v}\|_{L^2(\G)}^2
+ \|\vec{v}\|_{L^2}^2)
+C(\va)
(\|\vec{g}\|_{L^2(\G)}^2+
\|\vec{f}\|_{L^2}^2
+\|w\|_{L^2}^2\|\nabla\phi\|^2_{L^\infty}
)\\
\leq & C\va \|\vec{v}\|^2_{H^1}+C(\va)(\|\vec{g}\|_{L^2(\G)}^2+
\|\vec{f}\|_{L^2}^2
+\|w\|_{L^2}^2\|\nabla\phi\|^2_{L^\infty}
)\\
\leq &\frac{1}{2}[\vec{v},\vec{v}]
+C(\|\vec{g}\|_{L^2(\G)}^2+
\|\vec{f}\|_{L^2}^2
+\|w\|_{L^2}^2\|\nabla\phi\|^2_{L^\infty}
),
\end{split}
\ee
where the constant $C(\va)>0$ depends on $\va>0$ and in the last inequality $\va$ has been chosen small such that $C\va \|\vec{v}\|^2_{H^1}\leq \frac{1}{2}[\vec{v},\vec{v}]$ thanks to Lemma \ref{pl5}. Substituting \eqref{ne3} into \eqref{ne1} and integrating the resulting inequality with respect to $t$ one deduces that
\be\label{ne4}
\begin{split}
\|\vec{v}&\|_{L^\infty_tL^2}^2
+\int_0^t[\vec{v},\vec{v}]ds
+\|\eta\|_{L^\infty_tH^1(\R^2)}^2\\
&\leq C(\|\vec{v}_0\|^2_{L^2}+\|\eta_0\|_{H^1(\R^2)}^2)
+C(\|\vec{g}\|_{L^2_tL^2(\G)}^2+
\|\vec{f}\|_{L^2_tL^2}^2
+\|w\|_{L^2_t L^2}^2\|\nabla\phi\|^2_{L^\infty_tL^\infty}
).
\end{split}
\ee
We next estimate $\|\na_0\vec{v}\|_{L^\infty_tL^2}^2
+\|\na_0\vec{v}\|_{L^2_tH^1}^2$. Replacing $\vec{\var}$ in \eqref{a5} with $-\De_0\vec{v}$ and using \eqref{ne42} along with integration by parts to get
\be\label{ne5}
\begin{split}
\frac{1}{2}\frac{d}{dt}\|&\na_0\vec{v}\|_{L^2}^2
+[\na_0\vec{v},\na_0\vec{v}]
+\frac{1}{2}\frac{d}{dt}\Big(\gamma
\|\na_0\eta\|_{L^2(\R^2)}^2
+\sigma\|\De_0\eta\|_{L^2(\R^2)}^2
\Big)\\
=&-\int_{\G}\vec{g}\cdot\De_0\vec{v}\,dx
-\int_{\Omega} \vec{f}\cdot\De_0\vec{v}\,dxdy
+\int_{\Omega}w\nabla\phi\cdot \De_0\vec{v}\,dxdy.
\end{split}
\ee
 It follows from the trace theorem and Sobolev embedding inequality that
\be\label{ne7}
\begin{split}
-\int_{\G}\vec{g}&\cdot\De_0 \vec{v}\,dx
-\int_{\Omega} \vec{f}\cdot\De_0\vec{v}\,dxdy
+\int_{\Omega}w\nabla\phi\cdot \De_0\vec{v}\,dxdy\\
\leq &
\va (\|\De_0\vec{v}\|_{H^{-\frac{1}{2}}(\G)}^2
+ \|\De_0\vec{v}\|_{L^2}^2)
+C(\va)
(\|\vec{g}\|_{H^{\frac{1}{2}}(\G)}^2+
\|\vec{f}\|_{L^2}^2
+\|w\|_{L^2}^2\|\nabla\phi\|^2_{L^\infty}
)\\
\leq & C\va \|\na_0\vec{v}\|^2_{H^1}
+C(\va)(\|\vec{g}\|_{H^{\frac{1}{2}}(\G)}^2+
\|\vec{f}\|_{L^2}^2
+\|w\|_{L^2}^2\|\nabla\phi\|^2_{L^\infty})\\
\leq &\frac{1}{2}[\na_0\vec{v},\na_0\vec{v}]
+C(\|\vec{g}\|_{H^{\frac{1}{2}}(\G)}^2+
\|\vec{f}\|_{L^2}^2
+\|w\|_{L^2}^2\|\nabla\phi\|^2_{L^\infty}
),
\end{split}
\ee
 where in the last inequality we have employed Lemma \ref{pl5} and chosen $\va$ small such that $C\va \|\na_0\vec{v}\|^2_{H^1}\leq
\frac{1}{2}[\na_0\vec{v},\na_0\vec{v}]$.
Inserting \eqref{ne7} into \eqref{ne5} and then integrating the resulting inequality over $(0,t)$, one deduces that
\be\label{ne8}
\begin{split}
\|&\na_0\vec{v}\|_{L^\infty_tL^2}^2
+\int_0^t[\na_0\vec{v},\na_0\vec{v}]ds
+\|\na_0\eta\|_{L^\infty_tH^1(\R^2)}^2\\
&\leq C(\|\vec{v}_0\|^2_{H^1}+\|\eta_0\|_{H^2(\R^2)}^2)
+C(\|\vec{g}\|_{L^2_tH^{\frac{1}{2}}(\G)}^2+
\|\vec{f}\|_{L^2_tL^2}^2
+\|w\|_{L^2_t L^2}^2\|\nabla\phi\|^2_{L^\infty_tL^\infty}
).
\end{split}
\ee
Applying $\De_0$ to the first equation of \eqref{e1}, then multiplying the resulting equality by $\De_0 \vec{v}$ in $L^2$ and using integration by parts we obtain
\be\label{ne9}
\begin{split}
\frac{1}{2}\frac{d}{dt}&\|\De_0\vec{v}\|_{L^2}^2+[\De_0\vec{v},\De_0\vec{v}]
+\int_{\G}(\gamma\De_0\eta-\sigma\De_0^2 \eta)\De_0v_3 \,dx\\
&=\int_{\G}(\De_0 \vec{g})\cdot(\De_0 \vec{v})\,dx
+\int_{\Omega} \De_0\vec{f}\cdot\De_0\vec{v}\,dxdy
-\int_{\Omega}\De_0(w\nabla\phi)\cdot \De_0\vec{v}\,dxdy.
\end{split}
\ee
It follows from $\eta_t=v_3$ on $\G\times(0,\infty)$ and integration by parts that
\be\label{ne11}
\int_{\G}(\gamma\De_0\eta-\sigma\De_0^2 \eta)\De_0v_3 \,dx
=\frac{1}{2}\frac{d}{dt}\Big(
\gamma\|\De_0\eta\|_{L^2(\R^2)}^2
+\sigma\|\na_0\De_0\eta\|_{L^2(\R^2)}^2
\Big).
\ee
By a similar argument used in \eqref{ne7} one gets
\be\label{ne10}
\begin{split}
\int_{\G}&(\De_0 \vec{g})\cdot(\De_0 \vec{v})\,dx
+\int_{\Omega} \De_0\vec{f}\cdot\De_0\vec{v}\,dxdy
-\int_{\Omega}\De_0(w\nabla\phi)\cdot \De_0\vec{v}\,dxdy\\
&\leq
\va (\|\De_0\vec{v}\|_{H^{\frac{1}{2}}(\G)}^2
+ \|\na_0\De_0\vec{v}\|_{L^2}^2)
+C(\va)
(\|\De_0 \vec{g}\|_{H^{-\frac{1}{2}}(\G)}^2+
\|\na_0\vec{f}\|_{L^2}^2
+\|\na_0(w\nabla\phi)\|^2_{L^2}
)\\
&\leq \frac{1}{2}[\De_0\vec{v},\De_0\vec{v}]
+C(\|\vec{g}\|_{H^{\frac{3}{2}}(\G)}^2+
\|\na_0\vec{f}\|_{L^2}^2)
+C\|w\|_{H^1}^2(\|\nabla\phi\|^2_{L^\infty}
+\|\nabla^2\phi\|^2_{L^\infty}),
\end{split}
\ee
where $\va$ has been chosen small such that
$\va (\|\De_0\vec{v}\|_{H^{\frac{1}{2}}(\G)}^2
+ \|\na_0\De_0\vec{v}\|_{L^2}^2)\leq \frac{1}{2}[\De_0\vec{v},\De_0\vec{v}]$ thanks to the trace theorem and Lemma \ref{pl5}. Substituting \eqref{ne11}-\eqref{ne10} into \eqref{ne9} then integrating the resulting inequality with respect to $t$ we arrive at
\be\label{ne12}
\begin{split}
\|&\De_0\vec{v}\|_{L^\infty_tL^2}^2
+\int_0^t[\De_0\vec{v},\De_0\vec{v}]ds
+\|\De_0\eta\|_{L^\infty_tH^1(\R^2)}^2\\
&\leq C(\|\vec{v}_0\|^2_{H^2}+\|\eta_0\|_{H^3(\R^2)}^2)
+C(\|\vec{g}\|_{L^2_tH^{\frac{3}{2}}(\G)}^2+
\|\vec{f}\|_{L^2_tH^1}^2)\\
&+C\|w\|_{L^2_t H^1}^2(\|\nabla\phi\|^2_{L^\infty_tL^\infty}
+\|\nabla^2\phi\|^2_{L^\infty_tL^\infty}
).
\end{split}
\ee
Collecting \eqref{ne4}, \eqref{ne8} and \eqref{ne12} and using Lemma \ref{pl5} we derive \eqref{ne13}. The proof is completed.

\end{proof}

\begin{lemma}\label{nl3} Let the assumptions in Proposition \ref{np1} hold true. Then the weak solutions $\vec{v}$ of \eqref{e1} satisfy
\be\label{ne14}
\begin{split}
\|\vec{v}_{t}\|_{L^\infty_tL^2}^2
+\|\vec{v}_{t}\|_{L^2_tH^1}^2
\leq&
C(\|\vec{v}_0\|_{H^2(\Om)}^2+\|\eta_0\|_{H^3(\R^2)}^2)\\
&+C|||w|||^2(
\|\nabla \phi\|_{L^\infty_tL^\infty}^2
+\|\nabla^2 \phi\|_{L^\infty_tL^\infty}^2
+\|\nabla \phi_t\|_{L^2_tL^2}^2
)\\
&+C(\|\vec{f}\|_{L^2_tH^1}^2+\|\vec{f}_t\|_{L^2_t(_{0}H^1)^{'}}^2
+\|\vec{g}\|_{L^2_tH^{\frac{3}{2}}(\G)}^2
+\|\vec{g}_{t}\|_{L^2_tH^{-\frac{1}{2}}(\G)}^2
)
\end{split}
\ee
for all $t>0$, where the constant $C$ is independent of $t$.
\end{lemma}
\begin{proof}
Differentiating \eqref{a5} with respect to $t$ and then replacing
$\vec{\var}$ with $\vec{v}_t$ leads to
\be\label{ne15}
\begin{split}
\frac{1}{2}\frac{d}{dt}&\|\vec{v}_t\|^2_{L^2}+[\vec{v}_t,\vec{v}_t]
+\frac{1}{2}\frac{d}{dt}
\left(
\gamma\|v_3\|_{L^2(\Gamma)}^2
+\sigma\|\nabla_0 v_3\|_{L^2(\Gamma)}^2
\right)\\
&=\int_{\G}\vec{g}_{t}\cdot \vec{v}_{t}\,dx
+\int_{\Omega} \vec{f}_t\cdot\vec{v}_t\,dxdy
-\int_{\Omega}(w\nabla\phi)_t\cdot \vec{v}_t\,dxdy.
\end{split}
\ee
By the trace theorem one has
\be\label{ne17}
\begin{split}
\int_{\G}\vec{g}_{t}\cdot\vec{v}_{t}\,dx
+\int_{\Omega} \vec{f}_t\cdot\vec{v}_t\,dxdy
\leq &  C\|\vec{v}_{t}\|_{H^{\frac{1}{2}}(\G)}
\|\vec{g}_{t}\|_{H^{-\frac{1}{2}}(\G)}
+C\|\vec{v}_{t}\|_{H^1}\|\vec{f}_{t}\|_{(_{0}H^1)^{'}}\\
\leq & \va  \|\vec{v}_{t}\|_{H^1}^2
+C(\va)
(
\|\vec{g}_{t}\|_{H^{-\frac{1}{2}}(\G)}^2
+\|\vec{f}_{t}\|_{(_{0}H^1)^{'}}^2
)\\
\leq & \frac{1}{4}[\vec{v}_t,\vec{v}_t]
+C(
\|\vec{g}_{t}\|_{H^{-\frac{1}{2}}(\G)}^2
+\|\vec{f}_{t}\|_{(_{0}H^1)^{'}}^2
),
\end{split}
\ee
where $\va$ has been chosen small such that $\va  \|\vec{v}_{t}\|_{H^1}^2\leq \frac{1}{4}[\vec{v}_t,\vec{v}_t]$ thanks to Lemma \ref{pl5}.
It follows from the Sobolev embedding inequality that
\be\label{ne18}
\begin{split}
-\int_{\Omega}(w\nabla\phi)_t\cdot \vec{v}_t\,dxdy
\leq& C\|w_t\|_{L^2}\|\nabla \phi\|_{L^\infty}\|\vec{v}_t\|_{L^2}
+C\|w\|_{L^\infty} \|\nabla \phi_t\|_{L^2}\|\vec{v}_t\|_{L^2}\\
\leq & \va \|\vec{v}_t\|_{H^1}^2
+C(\va)
\left(
\|\nabla \phi\|^2_{L^\infty}\|w_t\|_{L^2}^2
+\|\nabla \phi_t\|_{L^2}^2
\|w\|_{H^2}^2
\right)\\
\leq & \frac{1}{4}[\vec{v}_t,\vec{v}_t]
+C
\left(
\|\nabla \phi\|_{L^\infty}^2\|w_t\|_{L^2}^2
+\|\nabla \phi_t\|_{L^2}^2
\|w\|_{H^2}^2
\right),
\end{split}
\ee
where $\va$ has been chosen small such that $\va  \|\vec{v}_{t}\|_{H^1}^2\leq \frac{1}{4}[\vec{v}_t,\vec{v}_t]$ by Lemma \ref{pl5}. Inserting \eqref{ne17}-\eqref{ne18} into \eqref{ne15} and integrating the resulting inequality with over $(0,t)$ and using Lemma \ref{pl5}, we derive
\be\label{ne39}
\begin{split}
\|\vec{v}_{t}\|_{L^\infty_tL^2}^2
+\|\vec{v}_{t}\|_{L^2_tH^1}^2
\leq& \|\vec{v}_t(0)\|_{L^2}^2+C\|\vec{v}_0\|_{H^2}^2
+C(
\|\vec{g}_{t}\|_{L^2_tH^{-\frac{1}{2}}(\G)}^2
+\|\vec{f}_{t}\|_{L^2_t (_{0}H^1)^{'}}^2
)\\
&+C|||w|||^2(
\|\nabla \phi\|_{L^\infty_tL^\infty}^2
+\|\nabla \phi_t\|_{L^2_tL^2}^2
).
\end{split}
\ee
We next estimate the term $\|\vec{v}_t(0)\|_{L^2}^2$ on the right-hand side of \eqref{ne39}. Given the regularity assumptions in \eqref{ee01} and \eqref{ab1}, compactness (see e.g. \cite[Chapter 1, Theorem 3.1]{lions&magenes}) implies that
\ben
\vec{f}\in C(0,\infty;L^2),\quad \vec{g}\in C(0,\infty;H^{\frac{1}{2}}(\G)),\quad w\na\phi\in C(0,\infty;L^2).
\enn
In particular,
\be\label{ae20}
\|\vec{f}(0)\|_{L^2}\leq C(\|\vec{f}\|_{L^2_tH^1}+\|\vec{f}_t\|_{L^2_t(_{0}H^1)^{'}}),\quad
\|\vec{g}(0)\|_{H^{\frac{1}{2}}(\G)}
\leq C(\|\vec{g}\|_{L^2_tH^{\frac{3}{2}}(\G)}
+\|\vec{g}_{t}\|_{L^2_tH^{-\frac{1}{2}}(\G)}
)
\ee
and
\be\label{ae21}
\begin{split}
\|w(0)\na\phi(0)\|_{L^2}
\leq &C(\|w\na\phi\|_{L^2_tH^1}+\|(w\na\phi)_t\|_{L^2_tL^2})\\
\leq &C|||w|||^2(
\|\nabla \phi\|_{L^\infty_tL^\infty}^2
+\|\nabla^2 \phi\|_{L^\infty_tL^\infty}^2
+\|\nabla \phi_t\|_{L^2_tL^2}^2
).
\end{split}
\ee
 The fourth boundary condition in \eqref{e1} implies that
\be\label{ne41}
q(0)=2\partial_3 v_{03}+\gamma\eta_0-\sigma\De_0\eta_0-g_{3}(0)\quad \text{on}\ \ \G.
\ee
Applying projection $P$ to $\na q(0)$ and denoting $\na q_1(0)=P\na q(0)$, it follows from Lemma \ref{pl8} and \eqref{ne41} that
\be\label{ae17}
\left\{
\begin{array}{lll}
\Delta q_1(0)=0\quad\ \ \text{in}\ \ \Om,\\
q_1(0)=2\partial_3 v_{03}+\gamma\eta_0-\sigma\De_0\eta_0-g_{3}(0)\quad \ \text{on}\ \ \G,\\
\na q_{1}(0)\cdot\vec{n}=0\quad \ \text{on}\ \ S_B.
\end{array}
\right.
\ee
Applying the standard elliptic theory to \eqref{ae17} we get
\be\label{ae18}
\begin{split}
\|P\na q(0)\|_{L^2}=\|\na q_{1}(0)\|_{L^2}
\leq &C\|2\partial_3 v_{03}+\gamma\eta_0-\sigma\De_0\eta_0-g_{3}(0)\|_{H^{\frac{1}{2}}(\G)}\\
\leq &C\big(\|\vec{v}_0\|_{H^2}+\|\eta_0\|_{H^3(\R^2)}
+\|g_{3}(0)\|_{H^{\frac{1}{2}}(\G)}\big).
\end{split}
\ee
On the other hand, applying projection $P$ to the first equation of \eqref{e1} one deduces that
\be\label{ae19}
\begin{split}
\|\vec{v}_t(0)\|_{L^2}\leq& \|P(\De \vec{v}_0-w(0)\na\phi(0)+\vec{f}(0))\|_{L^2}
+\|P\na q(0)\|_{L^2}\\
\leq &C(
\|\vec{v}_0\|_{H^2}+\|w(0)\na\phi(0)\|_{L^2}+\|\vec{f}(0)\|_{L^2}
)+\|P\na q(0)\|_{L^2}.
\end{split}
\ee
Substituting \eqref{ae18} into \eqref{ae19} one arrives at
\ben
\|\vec{v}_t(0)\|_{L^2}
\leq C(
\|\vec{v}_0\|_{H^2}
+\|\eta_0\|_{H^3(\R^2)}
+\|w(0)\na\phi(0)\|_{L^2}
+\|\vec{f}(0)\|_{L^2}
+\|g_{3}(0)\|_{H^{\frac{1}{2}}(\G)}),
\enn
which, in conjunction with \eqref{ae20}-\eqref{ae21} leads to
\be\label{ae22}
\begin{split}
\|\vec{v}_t(0)\|_{L^2}
\leq& C(\|\vec{v}_0\|_{H^2}+\|\eta_0\|_{H^3(\R^2)})
+C|||w|||^2(
\|\nabla \phi\|_{L^\infty_tL^\infty}^2
+\|\nabla^2 \phi\|_{L^\infty_tL^\infty}^2
+\|\nabla \phi_t\|_{L^2_tL^2}^2
)\\
&+C(\|\vec{f}\|_{L^2_tH^1}+\|\vec{f}_t\|_{L^2_t(_{0}H^1)^{'}}
+\|\vec{g}\|_{L^2_tH^{\frac{3}{2}}(\G)}
+\|\vec{g}_{t}\|_{L^2_tH^{-\frac{1}{2}}(\G)}
).
\end{split}
\ee
Plugging \eqref{ae22} into \eqref{ne39} we obtain the desired estimate. The proof is finished.

\end{proof}

\begin{lemma}\label{nl5}
  Suppose the assumptions in Proposition \ref{np1} hold true. Let $\vec{v}(x_1,x_2,y,t)$ be the weak solutions derived in Lemma \ref{nl1}. Then there exists
  an associate pressure $q(x_1,x_2,y,t)$ defined in $\Om\times(0,\infty)$, such that $(\vec{v},q)$ solves \eqref{e1}. Moreover,
 \be\label{eeea4}
 \begin{split}
 \|\vec{v}\|_{L^2_tH^3}^2
 +\| q\|_{L^2_tH^2}^2
\leq  & C \big(\|\vec{v}_0\|_{H^2(\Om)}^2
+\|\eta_0\|_{H^3(\R^2)}^2\big)\\
&+C|||w|||^2\big(\|\nabla \phi\|_{L^\infty_tL^\infty}^2
+\|\nabla^2 \phi\|_{L^\infty_tL^\infty}^2
+\|\nabla \phi_t\|_{L^2_tL^2}^2\big)
\\
&+C\big(
\|\vec{f}\|_{L^2_tH^1}^2
+\|\vec{f}_{t}\|_{L^2_t (_{0}H^1)^{'}}^2
+\|\vec{g}\|_{L^2_t H^{\frac{3}{2}}(\G)}^2
+\|\vec{g}_{t}\|_{L^2_tH^{-\frac{1}{2}}(\G)}^2
\big)
\end{split}
\ee
for all $t>0$, where the constant $C$ is independent of $t$.
\end{lemma}

\begin{proof}
Let
\be\label{ea5}
\tilde{F}=\vec{f}-(w\nabla \phi)-\vec{v}_t,\qquad
\tilde{g}_1=g_1,\qquad \tilde{g}_2=g_2,\qquad
\tilde{g_3}=\gamma\eta-\sigma\Delta_0\eta-g_3.
\ee
It follows from Lemma \ref{nl2}, Lemma \ref{nl3} and the assumptions in Proposition \ref{np1} that $\tilde{F}\in L^2(\Om)$ and that $\tilde{g}=(\tilde{g}_1,\tilde{g}_2,\tilde{g}_3)\in H^{\frac{1}{2}}(\G)$ for almost every $t>0$. Then from Lemma \ref{pl6} we know that for fixed $t>0$, system \eqref{pe1} with $\tilde{F}$ and $\tilde{g}$ defined in \eqref{ea5} admits a unique solution in $H^2(\Om)\times H^1(\Om)$ depending on the choice of $t>0$, which we denote by $(\vec{\omega},q)(x_1,x_2,y,t)$. Taking the $L^2$ inner product of the first equation in \eqref{pe1} with $\vec{\var}\in\mathcal{V}$ and using integration by parts one easily deduces that
\ben
\begin{split}
[\vec{\omega},\vec{\var}]
=& -(w\nabla \phi,\vec{\var})+
\langle\vec{f},\vec{\var}\rangle
-\langle\vec{v}_t,\vec{\var}\rangle
-\gamma\int_{\Gamma}\Big(\eta_0+\int_0^t v_3\, ds\Big)\var_3\,dx\\
&-\sigma\int_{\Gamma}
\Big(\na_0\eta_0+\int_0^t\nabla_0v_3\, ds\Big)\cdot\nabla_0\var_3\,dx
+\langle \vec{g},\vec{\var}\rangle_{\Gamma},\qquad \forall \ \vec{\var}\in \mathcal{V}
\end{split}
\enn
for almost every $t>0$, which along with \eqref{a5} gives rise to
\be\label{ea6}
[\vec{v}-\vec{\omega},\vec{\var}]=0,\qquad \forall \ \vec{\var}\in \mathcal{V}
\ee
for almost every $t>0$.
From Lemma \ref{nl2} we know that $\na_0\vec{v}\in H^{1}(\Om)$ for almost every $t>0$, then it follows from the trace theorem that $\na_0\vec{v}\in L^2(\G)$, which indicates $\vec{v}\in \mathcal{V}$. Taking $\vec{\var}=\vec{v}-\vec{\omega}$, from \eqref{ea6} and Lemma \ref{pl5} one deduces for fixed $t>0$ that
\ben
\vec{v}=\vec{\omega}\quad \text{for}\ \ \text{almost}\ \ \text{every}\ \  (x_1,x_2,y,t)\in \Om\times (0,\infty),
\enn
which, along with \eqref{ea5} and \eqref{pe1} implies that  $(\vec{v},q)$ solves \eqref{e1}.

To derive the desired estimate \eqref{eeea4},
we next prove the following inequality
\be\label{ea2}
\|v_3\|_{L^2_tH^{\frac{5}{2}}(\G)}^2
\leq C\sum_{k=0}^2 \|\nabla^k_0 \vec{v}\|_{L^2_tH^1}^2
\ee
for all $t>0$,
 by using the condition $\na\cdot\vec{v}=0$.
 It follows from the trace theorem \cite[pp 39, Theorem 8.3]{lions&magenes} that any $\rho\in H^{\frac{1}{2}}(\G)$ admits an extension $\tilde{\rho}\in H^1(\Om)$ satisfying $\tilde{\rho}=\rho$ on $\G$ and
\be\label{ea3}
\|\tilde{\rho}\|_{H^1}\leq C\|\rho\|_{H^{\frac{1}{2}}(\G)},
\ee
with some constant $C>0$.
 Let $\xi(y)$ be a smooth function defined on $y\in [-1,0]$ fulfilling
  \ben
  \xi(y)=1 \quad \text{for} \
  \ -\frac{1}{3}\leq y\leq 0,\qquad
  \xi(y)=0 \quad \text{for} \
  \ -1\leq y\leq -\frac{2}{3}.
  \enn
For $i,j\in\mathbb{N}$ with $1\leq i+j\leq 3$ and fixed $t>0$, it follows from $\na\cdot\vec{v}=0$ and \eqref{ea3}
 that
\ben
\begin{split}
\int_{\G}(\partial^{i}_{x_1}\!\partial^{j}_{x_2} v_3)\rho dx
=\int_{\Om}(\partial^{i}_{x_1}\!\partial^{j}_{x_2} \vec{v})\cdot \na(\xi\tilde{\rho} )dxdy
\leq& C\|\partial^{i}_{x_1}\!\partial^{j}_{x_2}\vec{v}\|_{L^2}
\|\tilde{\rho}\|_{H^{1}}\\
\leq &C\|\partial^{i}_{x_1}\!\partial^{j}_{x_2}\vec{v}\|_{L^2}
\|\rho\|_{H^{\frac{1}{2}}(\G)}.
\end{split}
\enn
Thus
\be\label{ea4}
\sum_{i+j=1}^3\|\partial^{i}_{x_1}\!\partial^{j}_{x_2} v_3\|_{H^{-\frac{1}{2}}(\G)}^2\leq C\sum_{i+j=1}^3\|\partial^{i}_{x_1}\!\partial^{j}_{x_2} \vec{v}\|_{L^2}^2.
\ee
On the other hand, it follows from the trace theorem that
\ben
\|v_3\|_{L^2(\G)}^2\leq C\| \vec{v}\|_{H^1}^2
\enn
for fixed $t>0$, which along \eqref{ea4} and integration over $(0,t)$ gives rise to \eqref{ea2}.

With \eqref{ea2} in hand,
applying \cite[pp 78, Theorem 10.5]{agmon-douglis-nirenberg1964} to \eqref{e1} and using Lemma \ref{nl2}-Lemma \ref{nl3} one easily deduces that
\ben
\begin{split}
\|\vec{v}\|_{L^2_tH^3}^2+\|q\|_{L^2_tH^2}^2
\leq& C\|\vec{f}-(w\nabla \phi)-\vec{v}_t\|_{L^2_tH^1}^2
+C\|g_1\|_{L^2_tH^{\frac{3}{2}}(\G)}^2
+C\|g_2\|_{L^2_tH^{\frac{3}{2}}(\G)}^2
+C\|v_3\|_{L^2_tH^{\frac{5}{2}}(\G)}^2\\
\leq &C(\|\vec{v}_0\|_{H^2(\Om)}^2+\|\eta_0\|_{H^3(\R^2)}^2)
+C|||w|||^2(
\|\nabla \phi\|_{L^\infty_tL^\infty}^2
+\|\nabla^2 \phi\|_{L^\infty_tL^\infty}^2
+\|\nabla \phi_t\|_{L^2_tL^2}^2
)\\
&+C(\|\vec{f}\|_{L^2_tH^1}^2+\|\vec{f}_t\|_{L^2_t(_{0}H^1)^{'}}^2)
+C(\|\vec{g}\|_{L^2_tH^{\frac{3}{2}}(\G)}^2
+\|\vec{g}_{t}\|_{L^2_tH^{-\frac{1}{2}}(\G)}^2
).
\end{split}
\enn
The proof is finished.

\end{proof}

\begin{lemma}\label{pl9}
Let $(\vec{v},q)$ be the solution derived in Lemma \ref{nl5} with $\eta$ defined in \eqref{ne42}.  Let $\mathcal{H}(\eta)$ be the harmonic extension of $\eta$. Then there exists a constant $C$ independent of $t$ such that
\ben
\begin{split}
\|\nabla&\vec{v}_{t}\|_{L^2_t H^{-\frac{1}{2}}(\G)}^2
+\|\nabla q_t\|_{L^2_t({_0}H^1)^{'}}^2
+\|\na_0\eta\|_{L^2_tH^{\frac{5}{2}}(\R^2)}^2
+\|\na^2\mathcal{H}(\eta)\|_{L^2_tH^2}^2\\
\leq& C \big(\|\vec{v}_0\|_{H^2(\Om)}^2
+\|\eta_0\|_{H^3(\R^2)}^2\big)
+C|||w|||^2\big(\|\nabla \phi\|_{L^\infty_tL^\infty}^2
+\|\nabla^2 \phi\|_{L^\infty_tL^\infty}^2
+\|\nabla \phi_t\|_{L^2_tL^2}^2\big)
\\
&+C\big(
\|\vec{f}\|_{L^2_tH^1}^2
+\|\vec{f}_{t}\|_{L^2_t (_{0}H^1)^{'}}
+\|\vec{g}\|_{L^2_t H^{\frac{3}{2}}(\G)}^2
+\|\vec{g}_{t}\|_{L^2_tH^{-\frac{1}{2}}(\G)}^2
\big)
\end{split}
\enn
for all $t>0$.
\end{lemma}
\begin{proof}
We first estimate $\|\na\vec{v}_t\|_{L^2_t H^{-\frac{1}{2}}(\G)}
$.
Let $\vec{\psi}\in H^{\frac{1}{2}}(\G)$. Then for $i=1,2$ it follows from the trace theorem that
\ben
\int_{\G}\partial_i \vec{v}_t\cdot \vec{\psi} \,dx
=-\int_{\G} \vec{v}_t\cdot\partial_i \vec{\psi} \,dx
\leq \|\vec{v}_t\|_{H^{\frac{1}{2}}(\G)}
\|\partial_i\vec{\psi}\|_{H^{-\frac{1}{2}}(\G)}
\leq C\|\vec{v}_t\|_{H^{1}}
\|\vec{\psi}\|_{H^{\frac{1}{2}}(\G)}.
\enn
Thus
\be\label{ne28}
\|\partial_i\vec{v}_t\|_{L^2_tH^{-\frac{1}{2}}(\G)}
\leq C\|\vec{v}_t\|_{L^2_tH^{1}}\quad\text{for}\ \ i=1,2.
\ee
Noting $\partial_3 v_{3t}=-\partial_1 v_{1t}-\partial_2 v_{2t}$ due to $\na\cdot\vec{v}=0$, one gets from \eqref{ne28} that
\be\label{ne29}
\|\partial_3 v_{3t}\|_{L^2_tH^{-\frac{1}{2}}(\G)}
\leq
\|\partial_1 v_{1t}\|_{L^2_tH^{-\frac{1}{2}}(\G)}
+\|\partial_2 v_{2t}\|_{L^2_tH^{-\frac{1}{2}}(\G)}
\leq C\|\vec{v}_t\|_{L^2_tH^{1}}.
\ee
On the other hand, from the boundary conditions in \eqref{e1} we know
$\partial_3 v_{i t}=-\partial_i v_{3 t}+g_{it}$ on $\G\times(0,\infty)$ for $i=1,2$. Thus it follows from \eqref{ne28} that
\be\label{ne30}
\begin{split}
\|\partial_3 v_{i t}\|_{L^2_tH^{-\frac{1}{2}}(\G)}
\leq &
\|\partial_i v_{3t}\|_{L^2_tH^{-\frac{1}{2}}(\G)}
+\|g_{it}\|_{L^2_tH^{-\frac{1}{2}}(\G)}\\
\leq  &C\|\vec{v}_t\|_{L^2_tH^{1}}
+\|\vec{g}_{t}\|_{L^2_tH^{-\frac{1}{2}}(\G)}\qquad \quad \text{for}\ \ i=1,2.
\end{split}
\ee
Collecting \eqref{ne28}-\eqref{ne30} and using Lemma \ref{nl3} we arrive at
\be\label{ne45}
\begin{split}
\|\nabla\vec{v}_{t}\|_{L^2_t H^{-\frac{1}{2}}(\G)}^2
\leq&
C(\|\vec{v}_0\|_{H^2(\Om)}^2+\|\eta_0\|_{H^3(\R^2)}^2)\\
&+C|||w|||^2(
\|\nabla \phi\|_{L^\infty_tL^\infty}^2
+\|\nabla^2 \phi\|_{L^\infty_tL^\infty}^2
+\|\nabla \phi_t\|_{L^2_tL^2}^2
)\\
&+C(\|\vec{f}\|_{L^2_tH^1}^2+\|\vec{f}_t\|_{L^2_t(_{0}H^1)^{'}}^2
+\|\vec{g}\|_{L^2_tH^{\frac{3}{2}}(\G)}^2
+\|\vec{g}_{t}\|_{L^2_tH^{-\frac{1}{2}}(\G)}^2
).
\end{split}
\ee

We proceed to estimating $\|\nabla q_t\|_{L^2_t ({_0}H^1)^{'}}$.
Differentiating \eqref{a5} with respect to $t$ to have
\ben
\begin{split}
\langle &\vec{v}_{tt},\vec{\var}\rangle +[\vec{v}_t,\vec{\var}]+\gamma\int_{\Gamma} v_3\var_3\,dx-\sigma\int_{\Gamma}\Delta_0 v_3\var_3\,dx\\
&=
-(w_t\nabla \phi,\vec{\var})
-(w\nabla \phi_t,\vec{\var})
+\langle\vec{f}_t,\vec{\var}\rangle
+\langle \vec{g}_{t},\vec{\var}\rangle_\G,
\qquad
\forall \vec{\var}\in \mathcal{V}
\end{split}
\enn
for $a.e.$ $t>0$, which along with the Sobolev embedding inequality and trace theorem gives
\be\label{ne33}
\begin{split}
\|\vec{v}_{tt}\|_{L^2_t(_0H^1)^{'}}
\leq &C(
\|\vec{v}_{t}\|_{L^2_tH^1}
+\|\vec{v}\|_{L^2_tH^3}
)
+C|||w|||
(\|\nabla \phi\|_{L^\infty_tL^\infty}
+
\|\na\phi_t\|_{L^2_t L^2}
)\\
&+C(\|\vec{f}_{t}\|_{ L^2_t({_0}H^1)^{'}}
+\|\vec{g}_{t}\|_{L^2_tH^{-\frac{1}{2}}(\G)}
).
\end{split}
\ee
For any $\vec{\varrho}\in\, _0H^1$, one gets from integration by parts that
\ben
\begin{split}
\langle\De \vec{v}_t,\vec{\varrho}\rangle
=&\int_\G \partial_3\vec{v}_t\cdot \vec{\varrho}\,dx
-\int_{\Omega} \nabla \vec{v}_t \cdot \nabla
\vec{\varrho}\,dxdy\\
\leq &\|\nabla \vec{v}_t\|_{H^{-\frac{1}{2}}(\G)}
\|\vec{\varrho}\|_{H^{\frac{1}{2}}(\G)}
+\|\vec{v}_{t}\|_{H^1}\|\vec{\varrho}\|_{H^1}\\
\leq& C(
\|\na\vec{v}_t\|_{H^{-\frac{1}{2}}(\G)}
+\|\vec{v}_{t}\|_{H^1}
)\|\vec{\varrho}\|_{H^1}
\end{split}
\enn
which, indicates
\be\label{ne35}
\|\De \vec{v}_t\|_{L^2_t ({_0}H^1)^{'}}
\leq C(
\|\na\vec{v}_t\|_{L^2_tH^{-\frac{1}{2}}(\G)}
+\|\vec{v}_{t}\|_{L^2_tH^1}
).
\ee
Differentiating the first equation of \eqref{e1} with respect to $t$ and using \eqref{ne33}-\eqref{ne35} to derive
\ben
\begin{split}
\|\nabla q_t\|_{L^2_t ({_0}H^1)^{'}}
\leq &\|\vec{v}_{tt}\|_{L^2_t ({_0}H^1)^{'}}
+\|\De \vec{v}_t\|_{L^2_t ({_0}H^1)^{'}}
+|||w|||
(\|\nabla \phi\|_{L^\infty_tL^\infty}
+\|\na\phi_t\|_{L^2_t L^2})
+\|\vec{f}_t\|_{L^2_t ({_0}H^1)^{'}}\\
\leq &C(
\|\vec{v}_{t}\|_{L^2_tH^1}
+\|\vec{v}\|_{L^2_tH^3}
+\|\na\vec{v}_t\|_{L^2_tH^{-\frac{1}{2}}(\G)}
)
+C|||w|||
(\|\nabla \phi\|_{L^\infty_tL^\infty}
+
\|\na\phi_t\|_{L^2_t L^2}
)\\
&+C(\|\vec{f}_{t}\|_{ L^2_t({_0}H^1)^{'}}
+\|\vec{g}_{t}\|_{L^2_tH^{-\frac{1}{2}}(\G)}
)
\end{split}
\enn
which, along with \eqref{ne45}  and Lemma \ref{nl3}-Lemma \ref{nl5} gives rise to
\be\label{ne46}
\begin{split}
\|\nabla q_t\|_{L^2_t ({_0}H^1)^{'}}^2
\leq &
C(\|\vec{v}_0\|_{H^2(\Om)}^2+\|\eta_0\|_{H^3(\R^2)}^2)\\
&+C|||w|||^2(
\|\nabla \phi\|_{L^\infty_tL^\infty}^2
+\|\nabla^2 \phi\|_{L^\infty_tL^\infty}^2
+\|\nabla \phi_t\|_{L^2_tL^2}^2
)\\
&+C(\|\vec{f}\|_{L^2_tH^1}^2+\|\vec{f}_t\|_{L^2_t(_{0}H^1)^{'}}^2
+\|\vec{g}\|_{L^2_tH^{\frac{3}{2}}(\G)}^2
+\|\vec{g}_{t}\|_{L^2_tH^{-\frac{1}{2}}(\G)}^2
).
\end{split}
\ee

It remains to estimate $\|\na_0\eta\|_{L^2_tH^{\frac{5}{2}}(\R^2)}$ and
$\|\na^2\mathcal{H}(\eta)\|_{L^2_tH^2}$. Noting that the harmonic extension $\mathcal{H}$ is a linear operator, applying projection $P$ to $\na q$ and using Lemma \ref{pl8}  and the fourth boundary condition in \eqref{e1}, we deduce that
\ben
P\na q=\na \mathcal{H}(2\partial_3 v_3+\gamma\eta-\sigma\De_0 \eta-g_3)
=2\na\mathcal{H}(\partial_3 v_3)
+\na \mathcal{H}(\gamma\eta-\sigma\De_0 \eta)
-\na\mathcal{H}(g_3)
\enn
which, along with Lemma \ref{pl7} and the trace theorem entails that
\be\label{ae7}
\begin{split}
\|\na \mathcal{H}(\gamma\eta-\sigma\De_0 \eta)\|_{L^2_tH^1}
\leq &\|P\na q \|_{L^2_tH^1}
+2\|\na\mathcal{H}(\partial_3 v_3)\|_{L^2_tH^1}+\|\na\mathcal{H}(g_3)\|_{L^2_tH^1}\\
\leq &C \big(\|\na q \|_{L^2_tH^1}
+ \|\na\vec{v}\|_{L^2_tH^{\frac{3}{2}}(\G)}
+\|g_3\|_{L^2_tH^{\frac{3}{2}}(\G)}
\big)\\
\leq &C\big( \|\na q \|_{L^2_tH^1}
+ \|\vec{v}\|_{L^2_tH^3}
+\|g_3\|_{L^2_tH^{\frac{3}{2}}(\G)}
\big).
\end{split}
\ee
Noting $\partial_i \mathcal{H}(\gamma\eta-\sigma\De_0 \eta)=\mathcal{H}(\partial_i[\gamma\eta-\sigma\De_0 \eta])$ for $i=1,2$, we deduce from the trace theorem and \eqref{ae7} that
\be\label{ae8}
\begin{split}
\|\na_0 (\gamma\eta-\sigma\De_0 \eta)\|_{L^2_tH^{\frac{1}{2}}(\R^2)}
\leq&
\|\na_0 \mathcal{H}(\gamma\eta-\sigma\De_0 \eta)\|_{L^2_tH^1}\\
\leq &C\big( \|\na q \|_{L^2_tH^1}
+ \|\vec{v}\|_{L^2_tH^3}
+\|g_3\|_{L^2_tH^{\frac{3}{2}}(\G)}
\big).
\end{split}
\ee
On the other hand, $\na_0\eta$ solves the following elliptic system
\be\label{ae9}
\left\{
\begin{array}{lll}
-\sigma\De_0 (\na_0\eta)+\gamma(\na_0\eta)=\na_0(\gamma\eta-\sigma\De_0\eta),\qquad (x_1,x_2)\in \R^2,\\
\lim\limits_{|\vec{x}|\rightarrow\infty}(\na_0\eta)=0
\end{array}
\right.
\ee
where $|\vec{x}|=\sqrt{x_1^2+x_2^2}$.
Then applying the standard elliptic theory (see e.g. \cite[pp 342, Theorem 5]{Evans}) to \eqref{ae9} and using \eqref{ae8} we deduce that
\be\label{ae10}
\|\na_0\eta\|_{L^2_tH^{\frac{5}{2}}(\R^2)}
\leq C\|\na_0(\gamma\eta-\sigma\De_0\eta)\|_{L^2_tH^{\frac{1}{2}}(\R^2)}
\leq C\big( \|\na q \|_{L^2_tH^1}
+\|\vec{v}\|_{L^2_tH^3}
+\|\vec{g}\|_{L^2_tH^{\frac{3}{2}}(\G)}
\big),
\ee
which, in conjunction with the fact $\partial_i \mathcal{H}(\eta)=\mathcal{H}(\partial_i\eta)$ for $i=1,2$ and the trace theorem leads to
\be\label{ae11}
\|\na_0 \mathcal{H}(\eta)\|_{L^2_tH^3}
\leq
\|\na_0\eta\|_{L^2_tH^{\frac{5}{2}}(\R^2)}
\leq C \big(\|\na q \|_{L^2_tH^1}
+ \|\vec{v}\|_{L^2_tH^3}
+\|\vec{g}\|_{L^2_tH^{\frac{3}{2}}(\G)}
\big).
\ee
Noting $\De \mathcal{H}(\eta)=0$ in $\Om$, it follows from \eqref{ae11} that
\ben
\|\partial_3^2 \mathcal{H}(\eta)\|_{L^2_tH^2}
=
\|\partial_1^2\mathcal{H}(\eta)
+\partial_2^2\mathcal{H}(\eta)\|_{L^2_tH^2}
\leq C\big( \|\na q \|_{L^2_tH^1}
+ \|\vec{v}\|_{L^2_tH^3}
+\|\vec{g}\|_{L^2_tH^{\frac{3}{2}}(\G)}
\big),
\enn
which, along with \eqref{ae11} leads to
\be\label{ae12}
\|\na^2\mathcal{H}(\eta)\|_{L^2_tH^2}
\leq C\big( \|\na q \|_{L^2_tH^1}
+\|\vec{v}\|_{L^2_tH^3}
+\|\vec{g}\|_{L^2_tH^{\frac{3}{2}}(\G)}
\big).
\ee
Collecting \eqref{ne45}, \eqref{ne46}, \eqref{ae10} and \eqref{ae12} and using Lemma \ref{nl5} we derive the desired estimates.
The proof is completed.

\end{proof}

We are now in the position to prove Proposition \ref{np1}.
~\\
\emph{\textbf{Proof of Proposition \ref{np1}.}}
First, from Lemma \ref{nl5} we know that $(\vec{v},q,\eta)$ with $\eta$ defined in \eqref{ne42} solves the initial-boundary value problem \eqref{e1}.
 Furthermore, one can easily get the regularity \eqref{ne19} by using Lemma \ref{nl2}-Lemma \ref{pl9} and the following fact:
\ben
\|\vec{v}\|_{L^\infty_tH^2}^2\leq C
(\|\vec{v}\|_{L^2_tH^3}^2+\|\vec{v}_t\|_{L^2_tH^1}^2)
,\quad
 \|\na q\|_{L^\infty_tL^2}^2
 \leq C
 (\|\na q\|_{L^2_tH^1}^2
+\|\nabla q_t\|_{L^2_t({_0}H^1)^{'}}^2)
\enn
for some constant $C>0$ independent of $t$, thanks to the compactness theorem (see e.g. \cite[Chapter 1, Theorem 3.1 ]{lions&magenes}). Uniqueness follows from \eqref{ne19}. The proof is completed.

\endProof

\subsection{Proof of Proposition \ref{p1}}\label{ss3}
\begin{proof}
Multiplying \eqref{ce19} with $ C_4(\|\nabla \phi\|_{L^\infty_tL^\infty}^2
+\|\nabla^2 \phi\|_{L^\infty_tL^\infty}^2+\|\na \phi_t\|_{L^2_t L^2})+1$ and adding the resulting inequality to \eqref{ne19}, one derives \eqref{ee78}. The proof is finished.

\end{proof}
\section{Approximation solutions for the nonlinear problem}
In this section, we shall first construct approximation solutions for the nonlinear problem \eqref{e001}-\eqref{e003} based on the results obtained on its linearized version \eqref{e2}-\eqref{e1} in Section 4 then proceed to gain uniform bounds for such approximation solutions by estimating the nonlinear terms on the right-hand side of each equation in the following approximating system  \eqref{ee1}-\eqref{eee7}.

Before constructing the approximation solutions, we exhibit some identities fulfilled by the initial data $w_0$, $h_0$ and $\vec{v}_0$, which are defined from $m_0$, $\tilde{c}_0$ and $\vec{u}_0$ through the transformation $\theta(0)$. Indeed, by similar arguments used in deriving \eqref{aae23}, \eqref{aae21} and \eqref{aae13}-\eqref{aae16} one can deduce from \eqref{cb} that the $w_0$, $h_0$ and $\vec{v}_0$ defined in \eqref{ne56}
satisfy the following identities:
\be\label{ne55}
\left\{
\begin{array}{lll}
\na \cdot \vec{v}_0=0\quad \text{in}\ \ \Om,\\
\partial_3 w_0=G_4(w_0,h_0,\bar{\eta}_0)-w_0\partial_3h_0,\quad h_0=0\quad
 \text{on}\ \ \G,\\
\partial_3 v_{01}+\partial_1 v_{03}=G_1(\vec{v}_0,\bar{\eta}_0),
\quad
\partial_3 v_{02}+\partial_2 v_{03}=G_2(\vec{v}_0,\bar{\eta}_0)
\quad \text{on}\ \ \G,\\
w_0=0,\quad\partial_3 h_0=0,\quad \vec{v}_0=\mathbf{0}\quad \text{on}\ \ S_B,
\end{array}
\right.
\ee
where $G_4(w_0,h_0,\bar{\eta}_0)$ is derived from $G_4(w,h,\bar{\eta})$ by
replacing the $\al$, $\beta$, $J$, $\eta$, $h$ and $w$ in \eqref{aae21} with $\al_0$, $\beta_0$, $J_0$, $\eta_0$, $h_0$ and $w_0$, respectively. Similarly, $G_1(\vec{v}_0,\bar{\eta}_0)$ and $G_2(\vec{v}_0,\bar{\eta}_0)$ are obtained from  the $G_1(\vec{v},\bar{\eta})$ defined below \eqref{aae13} and the $G_2(\vec{v},\bar{\eta})$ defined below \eqref{aae16} by replacing the $J$, $\eta$ and $\vec{v}$ there with $J_0$, $\eta_0$ and $\vec{v}_0$, respectively.

We are now in the position to construct the first and second approximation solutions. Assuming $w_0, h_0, \vec{v}_{0} \in H^2(\Om)$ and $\eta_0\in H^3(\R^2)$, one can easily deduce from the trace theorem that
\be\label{eeea5}
\begin{split}
 \|&G_1(\vec{v}_0,\bar{\eta}_0)\|_{H^{\frac{1}{2}}(\G)}
+\|G_2(\vec{v}_0,\bar{\eta}_0)\|_{H^{\frac{1}{2}}(\G)}
+\|G_4(w_0,h_0,\bar{\eta}_0)-
w_0\partial_3h_0\|_{H^{\frac{1}{2}}(\G)}\\
&\leq C(\|w_0\|_{H^2}+\|h_0\|_{H^2}
+\|\vec{v}_0\|_{H^2}+\|\eta_0\|_{H^3(\R^2)})
\end{split}
\ee
for some constant $C>0$.
Then from the trace theorem (see e.g. \cite[Chapter 1, Theorem 4.2]{lions&magenes}) and \eqref{eeea5}, we know that there exists $\tilde{G}_1, \tilde{G}_2, \tilde{G}_4 \in L^2(0,\infty;H^{\frac{3}{2}}(\G))$ with $\tilde{G}_{1t}, \tilde{G}_{2t}, \tilde{G}_{4t} \in L^2(0,\infty;H^{-\frac{1}{2}}(\G))$ satisfying that
\be\label{eeea7}
\left\{
\begin{array}{lll}
\tilde{G}_1(x_1,x_2,0)=G_1(\vec{v}_0,\bar{\eta}_0)\quad \text{on}\ \ \G ,\\
\tilde{G}_2(x_1,x_2,0)=G_1(\vec{v}_0,\bar{\eta}_0)\quad \text{on}\ \ \G,\\
\tilde{G}_4(x_1,x_2,0)=G_4(w_0,h_0,\bar{\eta}_0)-
w_0\partial_3h_0\quad \text{on}\ \ \G
\end{array}
\right.
\ee
and that
\be\label{eeea6}
\begin{split}
\|\ti{G}_{1}&\|_{L^2_tH^{\frac{3}{2}}(\G)}
+\|\ti{G}_{2}\|_{L^2_tH^{\frac{3}{2}}(\G)}
+\|\ti{G}_{4}\|_{L^2_tH^{\frac{3}{2}}(\G)}
+\|\ti{G}_{1t}\|_{L^2_tH^{-\frac{1}{2}}(\G)}
+\|\ti{G}_{2t}\|_{L^2_tH^{-\frac{1}{2}}(\G)}
+\|\ti{G}_{4t}\|_{L^2_tH^{-\frac{1}{2}}(\G)}\\
\leq &C\|G_1(\vec{v}_0,\bar{\eta}_0)\|_{H^{\frac{1}{2}}(\G)}
+C\|G_2(\vec{v}_0,\bar{\eta}_0)\|_{H^{\frac{1}{2}}(\G)}
+C\|G_4(w_0,h_0,\bar{\eta}_0)-
w_0\partial_3h_0\|_{H^{\frac{1}{2}}(\G)}\\
\leq& C(\|w_0\|_{H^2}+\|h_0\|_{H^2}
+\|\vec{v}_0\|_{H^2}+\|\eta_0\|_{H^3(\R^2)}).
\end{split}
\ee
Let $(w^1, h^1, \vec{v}^{\,1},q^1,\eta^1)=(w^2, h^2, \vec{v}^{\,2},q^2, \eta^2 )$ be the solutions of the following system:
\be\label{eee4}
\left\{
\begin{array}{lll}
w_t-\Delta w=0\quad \text{in}\ \ \Omega\times(0,\infty),\\
h_t-\Delta h-w=0,\\
\vec{v}_t-\Delta\vec{v}+\nabla q +w\nabla \phi=0,\\
\na\cdot\vec{v}=0,\\
(w,h,\vec{v})(x_1,x_2,y,0)
=(w_0,h_0,\vec{v}_0)(x_1,x_2,y),\quad \eta(x_1,x_2,0)=\eta_0(x_1,x_2),
\end{array}
\right.
\ee
with the following boundary conditions
\be\label{eee5}
\left\{
\begin{array}{lll}
\partial_3 w=\tilde{G}_4,\quad
 h=0,\quad
\partial_3 v_1+\partial_1 v_3=\tilde{G}_1,
\quad
\partial_3 v_2+\partial_2 v_3=\tilde{G_2}\quad \text{on}\ \ \Gamma\times(0,\infty),\\
\eta_t=v_3,\quad
q-2\partial_3 v_3=\gamma\eta-\sigma\Delta_0\eta \quad\text{on}\ \
\Gamma\times(0,\infty),\\
w=0,\quad\partial_3 h=0,\quad \vec{v}=\mathbf{0}\quad \text{on}\ \ S_B\times(0,\infty).
\end{array}
\right.
\ee
Applying Proposition \ref{p1} to system \eqref{eee4}-\eqref{eee5} we obtain the unique solution $(w^1, h^1, \vec{v}^{\,1}, q^1,\eta^1)=(w^2, h^2, \vec{v}^{\,2}, q^2, \eta^2 )$ since the required compatibility conditions in Proposition \ref{p1} follow directly from \eqref{ne55} and \eqref{eeea7}.

  With the well-defined first and second approximation solutions in hand, we proceed to constructing $(w^{(j+1)}, h^{(j+1)}, \vec{v}^{\,(j+1)},q^{(j+1)},\eta^{(j+1)})$ with $j\geq 2$ by solving the following linear system:
\be\label{ee1}
\left\{
\begin{array}{lll}
w^{(j+1)}_t-\Delta w^{(j+1)}-\nabla\cdot(w^{j}\nabla h^{(j+1)})=F_4(w^{(j-1)},w^j,h^j,\vec{v}^{j},\bar{\eta}^j)\quad \text{in}\ \ \Omega\times(0,\infty),\\
h^{(j+1)}_t-\Delta h^{(j+1)}-w^{(j+1)}=F_5(h^j,\vec{v}^{j},\bar{\eta}^j),
\\
\vec{v}^{\,(j+1)}_t-\Delta\vec{v}^{\,(j+1)}+\nabla q^{(j+1)} +w^{(j+1)}\nabla \phi=\vec{F}(w^j,\vec{v}^{j},\na q^j,\bar{\eta}^j),\\
\na\cdot\vec{v}^{\,(j+1)}=0,\\
(w^{(j+1)},h^{(j+1)},\vec{v}^{\,(j+1)})(x_1,x_2,y,0)
=(w_0,h_0,\vec{v}_0)(x_1,x_2,y),\\
\eta^{(j+1)}(x_1,x_2,0)=\eta_0(x_1,x_2),
\end{array}
\right.
\ee
with the following boundary conditions on $\Gamma\times(0,\infty):$
\be\label{ee2}
\left\{
\begin{array}{lll}
\partial_3 w^{(j+1)}+w^{j}\partial_3 h^{(j+1)}=G_4(w^j,h^j,\bar{\eta}^j),\qquad
 h^{(j+1)}=0,\\
\partial_3 v^{(j+1)}_1+\partial_1 v^{(j+1)}_3=G_1(\vec{v}^{j},\bar{\eta}^j),
\qquad
\partial_3 v^{(j+1)}_2+\partial_2 v^{(j+1)}_3=G_2(\vec{v}^{j},\bar{\eta}^j),\\
\eta^{(j+1)}_t=v^{(j+1)}_3,\qquad
q^{(j+1)}-2\partial_3 v^{(j+1)}_3=\gamma\eta^{(j+1)}-\sigma\Delta_0\eta^{(j+1)}-G_3
(\vec{v}^{j},\bar{\eta}^j)
\end{array}
\right.
\ee
and the following boundary conditions on $S_B\times(0,\infty):$
\be\label{eee7}
w^{(j+1)}=0,\quad\partial_3 h^{(j+1)}=0,
\quad \vec{v}^{\,(j+1)}=\mathbf{0},
\ee
where $\vec{F}(w^j,\vec{v}^{j},\na q^j,\bar{\eta}^j)=(F_1,F_2,F_3)(w^j,\vec{v}^{j},\na q^j,\bar{\eta}^j)$, $\vec{G}(\vec{v}^{j},\bar{\eta}^j)
=(G_1,G_2,G_3)(\vec{v}^{j},\bar{\eta}^j)$, $F_5(h^j,\vec{v}^{j},\bar{\eta}^j)$ and $G_4(w^j,h^j,\bar{\eta}^j)$ are derived from the $\vec{F}$ defined below \eqref{aae15}, the $\vec{G}$ defined below \eqref{aae13}-\eqref{aae17}, the $F_5$ defined below \eqref{aae9} and the $G_4$ defined in \eqref{aae21} by replacing the $J$,  $\al$, $\beta$, $w$,
 $h$, $\vec{v}$, $q$ and $\eta$ there with $J^j$, $\al^j$, $\beta^j$,
  $w^j$, $h^j$, $\vec{v}^j$, $q^j$ and $\eta^j$, respectively. Here
 \ben
\al^j=(1+y)\partial_1\bar{\eta}^j,
\quad
\beta^j=(1+y)\partial_2\bar{\eta}^j,
\quad
J^j=1+\bar{\eta}^j+\partial_3 \bar{\eta}^j(1+y),
\enn
with $\bar{\eta}^j=\mathcal{H}(\eta^j).$
 $F_4(w^{(j-1)},w^j,h^j,\vec{v}^{j},\bar{\eta}^j)$ is derived from the $F_4$ defined below \eqref{aae9} by replacing the first line with
\be\label{ee94}
\begin{split}
&\left\{(J^{j})^{-2}[(\al^j)^2+(\beta^j)^2+1]-1
\right\}\partial_3\big(\partial_3 w^{j}+w^{(j-1)}\partial_3 h^j\big)\\
&\ -2\al^j(J^{j})^{-1}
\big(\partial_3\partial_1 w^j
+w^{(j-1)}\partial_3\partial_1 h^j
\big)
-2\beta^j(J^{j})^{-1}
\big(\partial_3\partial_2 w^j
+w^{(j-1)}\partial_3\partial_2 h^j
\big)
\end{split}
\ee
and replacing the $J$,  $\al$, $\beta$, $w$,
 $h$, $\vec{v}$ and $\eta$ in Line 2 to Line 6 with $J^j$, $\al^j$, $\beta^j$,
  $w^j$, $h^j$, $\vec{v}^j$ and $\eta^j$, respectively.

To obtain such $(w^{(j+1)}, h^{(j+1)}, \vec{v}^{\,(j+1)}, q^{(j+1)},\eta^{(j+1)})$ by solving \eqref{ee1}-\eqref{eee7}, from Proposition \ref{p1} we know that the following compatibility conditions are required to be fulfilled:
  \be\label{ne51}
\left\{
\begin{array}{lll}
\na\cdot \vec{v}_0=0\quad \text{in}\quad \Om,\\
\partial_3 w_0+w_0\partial_3 h_0=G_4(w^j(x_1,x_2,y,0),h^j(x_1,x_2,y,0),\bar{\eta}^j(x_1,x_2,y,0)),
\quad h_0=0\quad \text{on}\ \ \G,\\
 \partial_3v_{01}+\partial_{1}v_{03}
 =G_1(\vec{v}^{j}(x_1,x_2,y,0),\bar{\eta}^j
(x_1,x_2,y,0))\quad \text{on}\ \ \G,\\
\partial_3v_{02}+\partial_{2}v_{03}
=G_2(\vec{v}^{j}(x_1,x_2,y,0),\bar{\eta}^j
(x_1,x_2,y,0))
\quad  \text{on}\ \ \G,\\
w_0=0,\quad \partial_3h_0=0,\quad\vec{v}_0=\mathbf{0}\quad \text{on}\ \ S_B.
\end{array}
\right.
\ee
 The solvability of the system \eqref{ee1}-\eqref{eee7} and the uniform bounds for the approximation solutions $\{(w^{j},h^{j},\vec{v}^{j}, q^{j},\eta^{j})\}_{j\in\N_{+}}$ are as follows.
\begin{prop}\label{p4}Suppose that the initial data $w_0, h_0, \vec{v}_0\in H^2(\Om)$ and $\eta_0\in H^3(\R^2)$ satisfy \eqref{ae3}, \eqref{ae4} and the identities in \eqref{ne55}. Assume further $\na\phi, \na^2\phi\in L^{\infty}(0,\infty;L^\infty)$ and $\na\phi_t\in L^{2}(0,\infty;L^2)$. Let $(w^{1},h^{1},\vec{v}^{\,1}, q^1,\eta^1)=(w^{2},h^{2},\vec{v}^{\,2}, q^2,\eta^2)$ be the solution of system \eqref{eee4}-\eqref{eee5}.
 Then system \eqref{ee1}-\eqref{eee7} with $j\geq 2$ admits a unique solution $(w^{(j+1)},h^{(j+1)},\vec{v}^{\,(j+1)}, q^{(j+1)},\eta^{(j+1)})$ with the regularity defined by the left-hand side of \eqref{eee8}.
Moreover, there exists a constant $C$ independent of $j$ and $t$ such that
\be\label{eee8}
\|\{w^j,h^j,\vec{v}^j,q^j,\eta^j\}\|^2
\leq C(\|w_0\|_{H^2(\Om)}+\|h_0\|_{H^2(\Om)}
+\|\vec{v}_0\|_{H^2(\Om)}+\|\eta_0\|_{H^3(\R^2)})^2
\ee
for all $j\geq 1$ and $t>0$.
Furthermore, it holds true for all $j\geq 1$ that
\be\label{eee14}
\frac{1}{2}<J^j<\frac{3}{2}\qquad \text{in}\ \  \Om\times(0,\infty),
\ee
where $J^j=1+\bar{\eta}^j+\partial_3 \bar{\eta}^j(1+y)$.
\end{prop}

In the next subsection, we shall estimate the nonlinear terms $\vec{F}$, $\vec{G}$, $F_4$, $F_5$ and $G_4$ in system \eqref{ee1}-\eqref{eee7} and the proof of Proposition \ref{p4} will be given in subsection \ref{ss2}.

\subsection{Estimates on nonlinear terms}\label{ss1}
In estimating the nonlinear terms, we assume that
\be\label{eee15}
\frac{1}{2}<J^j<\frac{3}{2}\qquad \text{in}\ \  \Om\times(0,\infty)
\ee
for all $j\geq 1$,
which will be verified in the proof of Proposition \ref{p4} by using the smallness of $\eta^j$. Assumption \eqref{eee15} will be used repeatedly in the proof of the following Lemma \ref{l2}- Lemma \ref{l8} without further clarification.

We first estimate $\vec{F}(w^j,\vec{v}^{j},\na q^j,\bar{\eta}^j)$ and $\vec{F}_t(w^j,\vec{v}^{j},\na q^j,\bar{\eta}^j)$. It follows from the definition of $\vec{F}$ below \eqref{aae15} that
\be\label{ee0}
\begin{split}
\vec{F}(w^j,\vec{v}^{j},\na q^j,\bar{\eta}^j)=&
\left\{(J^j)^{-2}\left[(\al^j)^2+(\beta^j)^2+1\right]-1
\right\}\partial_3^2 \vec{v}^{j}
-2(J^j)^{-1}\al^j\partial_3\partial_1\vec{v}^j\\
&-2(J^j)^{-1}\beta^j\partial_3\partial_2\vec{v}^j
+\vec{Q}+\vec{L},
\end{split}
\ee
where $\vec{Q}=(Q_1,Q_2,Q_3)$ with
\ben
 \begin{split}
 Q_1:=&(1-J^j)\partial_1 q^j+\al^j\partial_3 q^j,
 \quad \quad\,\,
 Q_2:=(1-J^j)\partial_2 q^j+\beta^j\partial_3 q^j,\\
 Q_3:=&\left\{1-(J^j)^{-2}
 \left[(\al^j)^2+(\beta^j)^2+1\right]
 \right\}\partial_3 q^j+\al^j\partial_1 q^j+\beta^j\partial_2 q^j
 \end{split}
 \enn
 and the lower order terms (with respect to $\vec{v}^j)$ $\vec{L}$ are as follows:
 \ben
 \vec{L} \sim (\na \bar{\eta}^j)^2 (\na^2 \bar{\eta}^j)\na\vec{v}^j
 +(\na \bar{\eta}^j)^2(\na^3 \bar{\eta}^j) \vec{v}^j
+(\na\bar{\eta}^j_t)\vec{v}^j+(\na^2 \bar{\eta}^j)\vec{v}^j\bar{\eta}^j_t+w^j(\na\bar{\eta}^j)\na\phi.
 \enn

 \begin{lemma}\label{l2} Let \eqref{eee15} and the assumptions in Proposition \ref{p4} hold. Then there exists a constant $C$ independent of $j$ and $t$ such that
\ben
\|\vec{F}\|_{L^2_t H^1}^2
\leq C \|\{
w^{j},h^{j},\vec{v}^{j},q^{j}, \eta^{j}
\}\|^4
+C\|\{
w^{j},h^{j},\vec{v}^{j},q^{j}, \eta^{j}
\}\|^8
\enn
for all $j\geq 2$ and $t>0$.
\end{lemma}
\begin{proof}
By the definition of $\al^j$, $\beta^j$ and $J^j$ in \eqref{aae18} we know that
\be\label{eeea1}
\left\{(J^j)^{-2}
\left[
(\al^j)^2+(\beta^j)^2+1
\right]-1
\right\}\sim (\na \bar{\eta}^j)^2.
 \ee
 Thus it follows from the Sobolev embedding inequality and Lemma \ref{pl7} that
\be\label{ee17}
\begin{split}
&\left\|\left\{(J^j)^{-2}
\left[
(\al^j)^2+(\beta^j)^2+1
\right]-1
\right\}\partial_3^2 \vec{v}^{j}\right\|_{L^2_tH^1}^2\\
&\,\,\leq C \left(\|\na \bar{\eta}^j\|_{L^\infty_tL^\infty}^2
\|\na^2 \bar{\eta}^j\|_{L^\infty_tL^4}^2
\|\partial^2_3 \vec{v}^j\|_{L^2_tL^4}^2
+\|\na \bar{\eta}^j\|_{L^\infty_tL^\infty}^4
\|\na \partial^2_3 \vec{v}^j\|_{L^2_tL^2}^2
\right)\\
&\,\,\leq C \|\eta^j\|_{L^\infty_tH^3(\R^2)}^4
\| \vec{v}^j\|_{L^2_tH^3}^2.
\end{split}
\ee
Similarly, one gets
\be\label{ee18}
\left\|2\al^j(J^{j})^{-1}\partial_3\partial_1\vec{v}^j
\right\|_{L^2_tH^1}^2
+
\left\|2\beta^j(J^{j})^{-1}\partial_3\partial_2\vec{v}^j
\right\|_{L^2_tH^1}^2
\leq C \|\eta^j\|_{L^\infty_tH^3(\R^2)}^2
\| \vec{v}^j\|_{L^2_tH^3}^2.
\ee
Recalling the definition of $\vec{Q}$ below \eqref{ee0}, using the fact $(1-J^j),\,\al^j,\,\beta^j\sim \na \bar{\eta}^j$ and \eqref{eeea1} we have
\be\label{ee19}
\begin{split}
\|Q_1\|_{L^2_tH^1}^2+\|Q_2\|_{L^2_tH^1}^2
+\|Q_3\|_{L^2_tH^1}^2
\leq &C(\|\eta^j\|_{L^\infty_tH^3(\R^2)}^2
+\|\eta^j\|_{L^\infty_tH^3(\R^2)}^4)
\|\na q^j\|_{L^2_tH^1}^2.
\end{split}
\ee
We next estimate each term in $\vec{L}$. 
Sobolev embedding inequality and Lemma \ref{pl7} lead to
\be\label{ee22}
\begin{split}
\|(\na &\bar{\eta}^j)^2(\na^2\bar{\eta}^j)\na \vec{v}^j\|_{L^2_t H^1}^2\\
\leq &C\|\na \bar{\eta}^j\|_{L^\infty_tL^\infty}^2
\|\na^2 \bar{\eta}^j\|_{L^\infty_tL^6}^4
\|\na \vec{v}^j\|_{L^2_tL^6}^2
+C\|\na \bar{\eta}^j\|_{L^\infty_tL^\infty}^4
\|\na^3 \bar{\eta}^j\|_{L^\infty_tL^2}^2
\|\na \vec{v}^j\|_{L^2_tL^\infty}^2\\
&+C\|\na \bar{\eta}^j\|_{L^\infty_tL^\infty}^4
\|\na^2 \bar{\eta}^j\|_{L^\infty_tL^4}^2
\|\na^2 \vec{v}^j\|_{L^2_tL^4}^2\\
\leq &C \|\eta^j\|_{L^\infty_tH^3(\R^2)}^6
\|\vec{v}^j\|_{L^2_tH^3}^2
\end{split}
\ee
and
\be\label{ee23}
\begin{split}
\|(\na&\bar{\eta}^j)^2(\na^3\bar{\eta}^j) \vec{v}^j\|_{L^2_tH^1}^2\\
\leq &
\|\na \bar{\eta}^j\|_{L^\infty_tL^\infty}^2
\|\na^2 \bar{\eta}^j\|_{L^\infty_tL^4}^2
\|\na^3 \bar{\eta}^j\|_{L^2_tL^4}^2
\|\vec{v}^j\|_{L^\infty_tL^\infty}^2
+C\|\na \bar{\eta}^j\|_{L^\infty_tL^\infty}^4
\|\na^4 \bar{\eta}^j\|_{L^2_tL^2}^2
\|\vec{v}^j\|_{L^\infty_tL^\infty}^2\\
&+C\|\na \bar{\eta}^j\|_{L^\infty_tL^\infty}^4
\|\na^3 \bar{\eta}^j\|_{L^\infty_tL^2}^2
\|\na \vec{v}^j\|_{L^2_tL^\infty}^2\\
\leq &C\|\eta^j\|_{L^\infty_tH^3(\R^2)}^4\|\na^2 \bar{\eta}^j\|_{L^2_tH^2}^2
\|\vec{v}^j\|_{L^\infty_tH^2}^2
+C\| \eta^j\|_{L^\infty_tH^3(\R^2)}^6
\|\vec{v}^j\|_{L^2_tH^3}^2.
\end{split}
\ee
It follows from $\eta^j_t=v_3^j$ on $\G\times(0,\infty)$, Lemma \ref{pl7} and the trace theorem that
\be\label{ee83}
\|\bar{\eta}^j_t(t)\|_{H^m}\leq C\|\eta^j_t(t)\|_{H^{m-\frac{1}{2}}(\R^2)}= C\|v^j_3(t)\|_{H^{m-\frac{1}{2}}(\G)}
\leq C\|\vec{v}^j(t)\|_{H^m},\quad \forall t>0,\  m\geq 1.
\ee
By the Sobolev embedding inequality and \eqref{ee83} one deduces that
\be\label{ee86}
\|(\na\bar{\eta}^j_t)\vec{v}^j\|_{L^2_tH^1}^2
+\|(\na^2\bar{\eta}^j)\vec{v}^j\bar{\eta}^j_t\|_{L^2_tH^1}^2
\leq C(\| \eta^j\|_{L^\infty_tH^3(\R^2)}^2+1)\|\vec{v}^j\|_{L^\infty_tH^2}^2
\|\vec{v}^j\|_{L^2_tH^3}^2
\ee
and that
\be\label{ee87}
\|w^j(\na\bar{\eta}^j)\na\phi\|_{L^2_tH^1}^2
\leq
C\| \eta^j\|_{L^\infty_tH^3(\R^2)}^2
\|w^j\|_{L^2_tH^3}^2
(\|\na\phi\|_{L^\infty_tL^\infty}^2
+\|\na^2\phi\|_{L^\infty_tL^\infty}^2).
\ee
Combining \eqref{ee22}-\eqref{ee87} we arrive at
\ben
\|\vec{L}\|_{L^2_tH^1}^2
\leq C \|\{
w^{j},h^{j},\vec{v}^{j},q^{j}, \eta^{j}
\}\|^4
+C\|\{
w^{j},h^{j},\vec{v}^{j},q^{j}, \eta^{j}
\}\|^{8}
\enn
which, along with \eqref{ee17}-\eqref{ee19} and \eqref{ee0} gives the desired estimate. The proof is completed.

\end{proof}

 \begin{lemma}\label{l1} Let \eqref{eee15} and the assumptions in Proposition \ref{p4} hold true. Then
\ben
\begin{split}
\|\vec{F}_t\|_{L^2_t({_0}H^1)^{'}}^2
\leq C \|\{
w^{j},h^{j},\vec{v}^{j},q^{j}, \eta^{j}
\}\|^4+C\|\{
w^{j},h^{j},\vec{v}^{j},q^{j}, \eta^{j}
\}\|^8
\end{split}
\enn
for all $j\geq 2$ and $t>0$,
where the constant $C$ is independent of $j$ and $t$.
 \end{lemma}
\begin{proof}
 From the definition of $\al^j$, $\beta^j$ and $J^j$ in \eqref{aae18} we know that
\be\label{ee4}
\left\{(J^j)^{-2}\left[
(\al^j)^2+(\beta^j)^2+1\right]-1
\right\}_t\partial_3^2 \vec{v}^{j}\sim \na \bar{\eta}^j\na \bar{\eta}^j_t
\partial_3^2 \vec{v}^{j}
+(\na \bar{\eta}^j)^2\na \bar{\eta}^j_t
\partial_3^2 \vec{v}^{j}.
\ee
It follows from \eqref{ee83}, the Sobolev embedding inequality and Lemma \ref{pl7} that
\ben
\begin{split}
\|\na \bar{\eta}^j\na \bar{\eta}^j_t
\partial_3^2 \vec{v}^{j}+(\na \bar{\eta}^j)^2\na \bar{\eta}^j_t\partial_3^2 \vec{v}^{j}\|_{L^2_tL^2}^2
\leq& (\|\na \bar{\eta}^j\|_{L^\infty_tL^\infty}^2
+\|\na \bar{\eta}^j\|_{L^\infty_tL^\infty}^4)
\|\na \bar{\eta}^j_t\|_{L^\infty L^4}^2
\|\partial_3^2 \vec{v}^{j}\|_{L^2_tL^4}^2
\\
\leq &C(\|\bar{\eta}^j\|_{L^\infty_tH^3}^2
+\|\bar{\eta}^j\|_{L^\infty_tH^3}^4)
\|\bar{\eta}^j_t\|_{L^\infty_t H^2}^2
\| \vec{v}^{j}\|_{L^2_tH^3}^2\\
\leq &C(\|\eta^j\|_{L^\infty_tH^3(\R^2)}^2
+\|\eta^j\|_{L^\infty_tH^3(R^2)}^4)
\|\vec{v}^j\|_{L^\infty_t H^2}^2
\| \vec{v}^{j}\|_{L^2_tH^3}^2,
\end{split}
\enn
which, along with \eqref{ee4} entails that
\be\label{ee5}
\begin{split}
&\left\|\left\{(J^j)^{-2}\left[
(\al^j)^2+(\beta^j)^2+1
\right]-1
\right\}_t\partial_3^2 \vec{v}^{j}
\right\|_{L^2_t L^2}^2\\
&\  \leq
C(\|\eta^j\|_{L^\infty_tH^3(\R^2)}^2
+\|\eta^j\|_{L^\infty_tH^3(\R^2)}^4)
\|\vec{v}^j\|_{L^\infty_t H^2}^2
\| \vec{v}^{j}\|_{L^2_tH^3}^2.
\end{split}
\ee
For any $\vec{\var}\in\, {_0}H^1$, integration by parts along with \eqref{eeea1} yields
\be\label{ee6}
\begin{split}
\int_{\Om}& \left\{
(J^j)^{-2}
\left[
(\al^j)^2+(\beta^j)^2+1
\right]-1
\right\}(\partial_3^2 \vec{v}^{j}_t) \cdot\vec{\var}\,dxdy\\
=&
\int_{\G} \left\{
(J^j)^{-2}
\left[
(\al^j)^2+(\beta^j)^2+1
\right]-1
\right\}(\partial_3 \vec{v}^{j}_t) \cdot\vec{\var}\,dx\\
&-\int_{\Om} \partial_3\left\{
(J^j)^{-2}
\left[
(\al^j)^2+(\beta^j)^2+1
\right]-1
\right\}(\partial_3 \vec{v}^{j}_t) \cdot\vec{\var}\,dxdy\\
&-\int_{\Om} \left\{
(J^j)^{-2}
\left[
(\al^j)^2+(\beta^j)^2+1
\right]-1
\right\}(\partial_3 \vec{v}^{j}_t) \cdot\partial_3\vec{\var}\,dxdy\\
\leq &C \|\na \vec{v}^j_t\|_{H^{-\frac{1}{2}}(\G)}
\|(\na \bar{\eta}^j)^2 \var\|_{H^{\frac{1}{2}}(\G)}
+C\|\na \bar{\eta}^j\|_{L^\infty}\|\na^2 \bar{\eta}^j\|_{L^4}\|\partial_3 \vec{v}^j_t\|_{L^2}\|\var\|_{L^4}\\
&+C\|\na \bar{\eta}^j\|_{L^\infty}^2 \|\partial_3 \vec{v}_t^j\|_{L^2}\|\partial_3\var\|_{L^2}\\
\leq& C\|\na \vec{v}^j_t\|_{H^{-\frac{1}{2}}(\G)}
\| \eta^j\|_{H^{3}(\R^2)}^2
\|\var\|_{H^1}
+C
\| \eta^j\|_{H^{3}(\R^2)}^2
\|\na \vec{v}^j_t\|_{L^{2}}
\|\var\|_{H^1},
\end{split}
\ee
where we have used the Sobolev embedding inequality and the following fact
\ben
\begin{split}
\|(\na \bar{\eta}^j)^2 \var\|_{H^{\frac{1}{2}}(\G)}
\leq &C
\|(\na \bar{\eta}^j)^2 \var\|_{L^{2}(\G)}
+C\|\na_0^{\frac{1}{2}}(\na \bar{\eta}^j)^2 \var\|_{L^2(\G)}
+C\|(\na \bar{\eta}^j)^2 \na_0^{\frac{1}{2}}\var\|_{L^2(\G)}\\
\leq & C \|\na \bar{\eta}^j\|_{L^\infty(\G)}^2
\|\var\|_{L^{2}(\G)}
+C\|\na \bar{\eta}^j\|_{L^\infty(\G)}
\|\na_0^{\frac{1}{2}}(\na \bar{\eta}^j)\|_{L^4(\G)}
\|\var\|_{L^{4}(\G)}\\
&+C \|\na \bar{\eta}^j\|_{L^\infty(\G)}^2
\|\na_0^{\frac{1}{2}}\var\|_{L^2(\G)}\\
\leq& C\|\eta^j\|_{H^3(\R^2)}^2\|\var\|_{H^{1}}
\end{split}
\enn
 thanks to Lemma \ref{pl7} and the trace theorem.
 It follows from \eqref{ee6} that
 \ben
 \big\|\left\{(J^j)^{-2}
\left[(\al^j)^2+(\beta^j)^2+1\right]-1
\right\}\partial_3^2 \vec{v}^{j}_t
\big\|_{L^2_t({_0}H^1)^{'}}^2
\leq C\| \eta^j\|_{L^\infty_tH^{3}(\R^2)}^4
\big(\|\na \vec{v}^j_t\|_{L^2_tH^{-\frac{1}{2}}(\G)}^2
+\|\na \vec{v}^j_t\|_{L^2_tL^{2}}^2
\big)
 \enn
 which, along with \eqref{ee5} implies that
 \be\label{ee7}
 \begin{split}
&\left\|\left\{\left[(J^j)^{-2}
\left((\al^j)^2+(\beta^j)^2+1\right)-1
\right]\partial_3^2 \vec{v}^{j}\right\}_t\right\|_{L^2_t({_0}H^1)^{'}}^2\\
\,\,\,&\leq C\| \eta^j\|_{L^\infty_tH^{3}(\R^2)}^4
\big(\|\na \vec{v}^j_t\|_{L^2_tH^{-\frac{1}{2}}(\G)}^2
+\|\na \vec{v}^j_t\|_{L^2_tL^{2}}^2
\big)\\
&\ \ \ \ +C(\| \eta^j\|_{L^\infty_tH^{3}(\R^2)}^2
+\| \eta^j\|_{L^\infty_tH^{3}(\R^2)}^4)\|\vec{v}^j\|_{L^\infty_t H^2}^2
\| \vec{v}^{j}\|_{L^2_tH^3}^2.
\end{split}
 \ee
 A similar argument used in deriving \eqref{ee7} leads to
 \be\label{ee8}
 \begin{split}
 &\left\|
 \left[2(J^j)^{-1}\al^j\partial_3\partial_1\vec{v}^j
 \right]_t
 \right\|_{L^2_t({_0}H^1)^{'}}^2
 +
 \left\|
 \left[2(J^j)^{-1}\beta^j\partial_3\partial_2\vec{v}^j
 \right]_t
 \right\|_{L^2_t({_0}H^1)^{'}}^2\\
 \,\,\,&\leq
  C\| \eta^j\|_{L^\infty_tH^{3}}^2
\big(\|\na \vec{v}^j_t\|_{L^2_tH^{-\frac{1}{2}}(\G)}^2
+\|\na \vec{v}^j_t\|_{L^2_tL^{2}}^2
\big)
+C\big(1+
\|\eta^j\|_{L^\infty_t H^3(\R^2)}^2
\big)\|\vec{v}^j\|_{L^\infty_t H^2}^2
\| \vec{v}^{j}\|_{L^2_tH^3}^2
.
\end{split}
 \ee
 Noting $(1-J^j)\sim \na \bar{\eta}^j$, for any $\psi\in\,{_0}H^1$ we have
 \ben
 \begin{split}
 \int_\Om [(1-J^j)\partial_1 q^j]_t\psi\,dxdy
 \leq& \|\nabla \bar{\eta}^j_t\|_{L^4}\|\nabla q^j\|_{L^2}\|\psi\|_{L^4}
 +\|(\nabla \bar{\eta}^j)\psi\|_{H^1}\|\nabla q^j_t\|_{({_0}H^1)^{'}}\\
 \leq& C \|\vec{v}^j\|_{H^2} \|\na q^j\|_{L^2}\|\psi\|_{H^1}
 +\|\eta^j\|_{H^3(\R^2)}\|\psi\|_{H^1}\|\nabla q^j_t\|_{({_0}H^1)^{'}},
 \end{split}
 \enn
where we have used the Sobolev embedding inequality, \eqref{ee83} and Lemma \ref{pl7} in the last inequality.
 Thus
 \be\label{ee9}
 \|[(1-J^j)\partial_1 q^j]_t\|_{L^2_t({_0}H^1)^{'}}^2
 \leq C\|\vec{v}^j\|_{L^2_tH^2}^2 \|\na q^j\|_{L^\infty_tL^2}^2
 +C\|\eta^j\|_{L^\infty_tH^3(\R^2)}^2\|\nabla q^j_t\|_{L^2_t({_0}H^1)^{'}}^2.
 \ee
 By a similar argument used in deriving \eqref{ee9} one deduces that
 \ben
 \|[\al^j\partial_3 q^j]_t\|_{L^2_t({_0}H^1)^{'}}^2
 \leq C\|\vec{v}^j\|_{L^2_tH^2}^2 \|\na q^j\|_{L^\infty_tL^2}^2
 +C\|\eta^j\|_{L^\infty_tH^3(\R^2)}^2\|\nabla q^j_t\|_{L^2_t({_0}H^1)^{'}}^2
 \enn
 which, along with \eqref{ee9} gives rise to
 \be\label{ee10}
 \|Q_{1t}\|_{L^2_t({_0}H^1)^{'}}^2
 \leq C\|\vec{v}^j\|_{L^2_tH^2}^2 \|\na q^j\|_{L^\infty_tL^2}^2
 +C\|\eta^j\|_{L^\infty_tH^3(\R^2)}^2\|\nabla q^j_t\|_{L^2_t({_0}H^1)^{'}}^2.
 \ee
 Similarly,
 \be\label{ee11}
 \|Q_{2t}\|_{L^2_t({_0}H^1)^{'}}^2
 \leq C\|\vec{v}^j\|_{L^2_tH^2}^2 \|\na q^j\|_{L^\infty_tL^2}^2
 +C\|\eta^j\|_{L^\infty_tH^3(\R^2)}^2\|\nabla q^j_t\|_{L^2_t({_0}H^1)^{'}}^2
 \ee
 and
 \be\label{ee12}
\begin{split}
 \|Q_{3t}\|_{L^2_t({_0}H^1)^{'}}^2
 \leq& C(\|\eta^j\|_{L^\infty_tH^3(\R^2)}^2
+\|\eta^j\|_{L^\infty_tH^3(\R^2)}^4)
 (\|\vec{v}^j\|_{L^2_tH^2}^2 \|\na q^j\|_{L^\infty_tL^2}^2
 +\|\nabla q^j_t\|_{L^2_t({_0}H^1)^{'}}^2).
\end{split}
 \ee
We proceed to estimating each term in $(\vec{L})_t$.
 For any $\vec{\var}\in{_0}H^1$, it follows from Lemma \ref{pl7} and \eqref{ee83} that
 \ben
 \begin{split}
 \int_\Om [(\na & \bar{\eta}^j)^2(\na^2\bar{\eta}^j)\na \vec{v}^j]_t\cdot\vec{\var}\,dxdy\\
  \leq& \|\na \bar{\eta}^j\|_{L^\infty}
 \|\na \bar{\eta}^j_t\|_{L^4}\|\na^2\bar{\eta}^j\|_{L^4}
 \|\na \vec{v}^j\|_{L^4}\|\vec{\var}\|_{L^4}
 +\|\na \bar{\eta}^j\|_{L^\infty}^2
 \|\na^2 \bar{\eta}^j_t\|_{L^4}
 \|\na \vec{v}^j\|_{L^2}\|\vec{\var}\|_{L^4}\\
 &+\|\na \bar{\eta}^j\|_{L^\infty}^2
 \|\na^2 \bar{\eta}^j\|_{L^4}
 \|\na \vec{v}^j_t\|_{L^2}\|\vec{\var}\|_{L^4}
 \\
 \leq&C (\|\eta^j\|_{H^3(\R^2)}^2\|\vec{v}^j\|_{H^3}
 \|\vec{v}^j\|_{H^2}
 +\|\eta^j\|_{H^3(\R^2)}^3\|\vec{v}^j_t\|_{H^1})
 \|\vec{\var}\|_{H^1}.
 \end{split}
 \enn
Thus
 \be\label{ee14}
 \|[(\na  \bar{\eta}^j)^2\na^2\bar{\eta}^j\na \vec{v}^j]_t\|_{L^2_t({_0}H^1)^{'}}^2
 \leq
 C (\|\eta^j\|_{L^\infty_tH^3(\R^2)}^4\|\vec{v}^j\|_{L^2_tH^3}^2
 \|\vec{v}^j\|_{L^\infty_tH^2}^2
 +\|\eta^j\|_{L^\infty_tH^3(\R^2)}^6\|\vec{v}^j_t\|_{L^2_tH^1}^2).
 \ee
 For any $\vec{\var}\in{_0}H^1$ we get
 \ben
 \begin{split}
 \int_\Om [(\na&\bar{\eta}^j)^2(\na^3 \bar{\eta}^j) \vec{v}^j]_{t}\cdot\vec{\var}\,dxdy\\
 \leq &
2\|\na\bar{\eta}^j\|_{L^\infty}
\|\na\bar{\eta}^j_t\|_{L^4}
\|\na^3 \bar{\eta}^j\|_{L^2}\|\vec{v}^j\|_{L^\infty}
 \|\vec{\var}\|_{L^4}
+
\|\na\bar{\eta}^j\|_{L^\infty}^2\|\na^3 \bar{\eta}^j_t\|_{L^2}\|\vec{v}^j\|_{L^\infty}
 \|\vec{\var}\|_{L^2}\\
 &+\|\na\bar{\eta}^j\|_{L^\infty}^2\|\na^3 \bar{\eta}^j\|_{L^2}\|\vec{v}^j_t\|_{L^4}
 \|\vec{\var}\|_{L^4}\\
 \leq &C\|\eta^j\|_{H^3(\R^2)}^2
\|\vec{v}^j\|_{H^3}\|\vec{v}^j\|_{H^2}
 \|\vec{\var}\|_{H^1}
 +C\|\eta^j\|^3_{H^3(\R^2)}\|\vec{v}^j_t\|_{H^1}
 \|\vec{\var}\|_{H^1}.
 \end{split}
 \enn
 Hence
 \be\label{ee15}
 \|[(\na \bar{\eta}^j)^2(\na^3 \bar{\eta}^j) \vec{v}^j]_{t}\|_{L^2_t({_0}H^1)^{'}}^2
 \leq C(\|\eta^j\|_{L^\infty_tH^3(\R^2)}^4\|\vec{v}^j\|_{L^2_tH^3}^2
 \|\vec{v}^j\|_{L^\infty_tH^2}^2
 +\|\eta^j\|_{L^\infty_tH^3(\R^2)}^6\|\vec{v}^j_t\|_{L^2_tH^1}^2
).
\ee
Employing Lemma \ref{pl7} and \eqref{ee83} one can easily deduce that
\be\label{ee84}
\begin{split}
\|[\na\bar{\eta}^j_t\vec{v}^j]_t\|_{L^2_tL^2}^2
\leq& C \|\vec{v}^j_t\|_{L^2_tH^1}^2\|\vec{v}^j\|_{L^\infty_tH^2}^2,
\\
\|[(\na^2\bar{\eta}^j)\vec{v}^j\bar{\eta}^j_t]_t\|_{L^2_tL^2}^2
\leq& C
\|\vec{v}^j\|_{L^\infty_tH^2}^2
(\|\vec{v}^j\|_{L^\infty_tH^2}^2
\|\vec{v}^j\|_{L^2_tH^3}^2+\|\eta^j\|_{L^\infty_tH^3(\R^2)}^2
\|\vec{v}^j_t\|_{L^2_tH^1}^2)
\end{split}
\ee
and that
\be\label{ee85}
\begin{split}
\|[w^j(\na\bar{\eta}^j)\na\phi]_t\|_{L^2_tL^2}^2
\leq &C\|\eta^j\|_{L^\infty_tH^3(\R^2)}^2
(\|w^j_t\|_{L^2_tH^1}^2\|\na\phi\|_{L^\infty_tL^\infty}^2
+\|w^j\|_{L^\infty_tH^2}^2\|\na\phi_t\|_{L^2_tL^2}^2)\\
&+C\|w^j\|_{L^\infty_tH^2}^2\|\na\phi\|_{L^\infty_tL^\infty}^2
\|\vec{v}^j\|_{L^2_tH^3}^2.
\end{split}
\ee
Combining \eqref{ee14}-\eqref{ee85} we arrive at
\be\label{ee16}
\begin{split}
\|(\vec{L})_t\|_{L^2_t({_0}H^1)^{'}}^2
\leq &C \|\{
w^{j},h^{j},\vec{v}^{j},q^{j}, \eta^{j}
\}\|^4\\
&+C\|\{
w^{j},h^{j},\vec{v}^{j},q^{j}, \eta^{j}
\}\|^8.
\end{split}
\ee
Collecting \eqref{ee7}, \eqref{ee8}, \eqref{ee10}-\eqref{ee12}, \eqref{ee16} and using \eqref{ee0} we derive the desired estimates. The proof is finished.

\end{proof}

We next estimate $\vec{G}(\vec{v}^{j},\bar{\eta}^j)$ and $\vec{G}_t(\vec{v}^{j},\bar{\eta}^j).$ From \eqref{aae13}-\eqref{aae17} we know that
\be\label{ee25}
\vec{G}(\vec{v}^{j},\bar{\eta}^j)\sim (\na_0 \eta^j)^4 (\na \vec{v}^j)
+(\na_0 \eta^j)^4 (\na^2_0 \eta^j)\vec{v}^j
+(\na_0 \eta^j)^2 (\na^2_0 \eta^j).
\ee
\begin{lemma}\label{l3} Suppose that \eqref{eee15} and the assumptions in Proposition \ref{p4} hold true. Then there exists a constant $C$ independent of $j$ and $t$ such that
\ben
\begin{split}
\|\vec{G}\|_{L^2_tH^{\frac{3}{2}}(\G)}^2
+\|\vec{G}_t\|_{L^2_tH^{-\frac{1}{2}}(\G)}^2
\leq C \|\{
w^{j},h^{j},\vec{v}^{j},q^{j}, \eta^{j}
\}\|^4
+C\|\{
w^{j},h^{j},\vec{v}^{j},q^{j}, \eta^{j}
\}\|^{12}
\end{split}
\enn
for all $j\geq 2$ and $t>0$.
\end{lemma}
\begin{proof}
It follows from the Sobolev embedding inequality and trace theorem that
\be\label{ee28}
\begin{split}
\|(\na_0 \eta^j)^4\na \vec{v}^j\|_{L^2_tH^{\frac{3}{2}}(\G)}^2
\leq C \|(\na_0 \eta^j)^4\|_{L^\infty_t H^{\frac{3}{2}}(\R^2)}^2
\|\na\vec{v}^j\|_{L^2_tH^{\frac{3}{2}}(\G)}^2
\leq C\|\eta^j\|_{L^\infty_tH^3(\R^2)}^8
\|\vec{v}^j\|_{L^2_tH^{3}}^2.
\end{split}
\ee
Similarly,
\be\label{ee29}
\begin{split}
\|(\na_0 \eta^j)^4
(\na_0^2 \eta^j) \vec{v}^j\|_{L^2_tH^{\frac{3}{2}}(\G)}^2
\leq &C \|(\na_0 \eta^j)^4\|_{L^\infty_t H^{\frac{3}{2}}(\R^2)}^2
\|\na_0^2 \eta^j\|_{L^2_t H^{\frac{3}{2}}(\R^2)}^2
\|\vec{v}^j\|_{L^\infty_tH^{\frac{3}{2}}(\G)}^2\\
\leq & C\|\eta^j\|_{L^\infty_tH^3(\R^2)}^8
\|\vec{v}^j\|_{L^\infty_tH^{2}}^2
\|\na_0\eta^j\|_{L^2_tH^{\frac{5}{2}}(\R^2)}^2
\end{split}
\ee
and
\be\label{ee30}
\begin{split}
\|(\na_0 \eta^j)^2
(\na_0^2 \eta^j)\|_{L^2_tH^{\frac{3}{2}}(\R^2)}^2
\leq & C\|\eta^j\|_{L^\infty_tH^3(\R^2)}^4
\|\na_0\eta^j\|_{L^2_tH^{\frac{5}{2}}(\R^2)}^2.
\end{split}
\ee
Combining \eqref{ee28}-\eqref{ee30} and using \eqref{ee25} yields
\be\label{ee92}
\begin{split}
\|\vec{G}\|_{L^2_tH^{\frac{3}{2}}(\G)}^2
\leq C \|\{
w^{j},h^{j},\vec{v}^{j},q^{j}, \eta^{j}
\}\|^4
+C\|\{
w^{j},h^{j},\vec{v}^{j},q^{j}, \eta^{j}
\}\|^{12}.
\end{split}
\ee

We proceed to estimating $\|\vec{G}_t\|_{L^2_tH^{-\frac{1}{2}}(\G)}^2$. First,
the Sobolev embedding inequality, trace theorem and \eqref{ee83} lead to
\be\label{ee88}
\begin{split}
\| (\na_0 \eta^j)^3(\na_0 \eta^j_t) (\na \vec{v}^j)\|_{L^2_tL^2(\G)}^2
\leq &\| \na_0 \eta^j\|_{L^\infty_tL^\infty(\R^2)}^6
\|\na \eta^j_t\|_{L^\infty_tL^4(\R^2)}^2
\|\na \vec{v}^j\|_{L^2_tL^4(\G)}^2\\
\leq&  C\|\eta^j\|_{L^\infty_tH^3(\R^2)}^6
\|\vec{v}^j\|_{L^\infty_tH^2}^2
\| \vec{v}^j\|_{L^2_tH^2}^2.
\end{split}
\ee
For any $\psi\in H^{\frac{1}{2}}(\R^2)$ one has
\ben
\int_{\G} (\na_0 \eta^j)^4\na \vec{v}^j_t \psi dx
\leq \|\na \vec{v}^j_t\|_{H^{-\frac{1}{2}}(\G)}
\|(\na_0 \eta^j)^4\psi\|_{H^{\frac{1}{2}}(\R^2)}
\enn
with
\ben
\begin{split}
\|(\na_0 \eta^j)^4\psi\|_{H^{\frac{1}{2}}(\R^2)}
\leq& C \|\na_0\eta^j\|_{L^\infty(\R^2)}^3
\|\na_0^{\frac{3}{2}}\eta^j\|_{L^4(\R^2)}
\|\psi\|_{L^4(\R^2)}
+C\|\na_0\eta^j\|_{L^\infty(\R^2)}^4
\|\na_0^{\frac{1}{2}}\psi\|_{L^2(\R^2)}\\
\leq &C\|\eta^j\|_{H^3(\R^2)}^4
\|\psi\|_{H^{\frac{1}{2}}(\R^2)}.
\end{split}
\enn
Thus
\ben
\|(\na_0 \eta^j)^4\na \vec{v}^j_t\|_{L^2_tH^{-\frac{1}{2}}(\G)}^2
\leq C\|\eta^j\|_{L^\infty_tH^3(\R^2)}^8
\|\na \vec{v}^j_t\|_{L^2_tH^{-\frac{1}{2}}(\G)}^2,
\enn
which, along with \eqref{ee88} implies that
\be\label{ee90}
\begin{split}
\|[(\na_0 \eta^j)^4 (\na \vec{v}^j)]_t\|_{L^2_tH^{-\frac{1}{2}}(\G)}^2
\leq& C\|\eta^j\|_{L^\infty_tH^3(\R^2)}^6
\|\vec{v}^j\|_{L^\infty_tH^2}^2
\| \vec{v}^j\|_{L^2_tH^2}^2\\
&+C\|\eta^j\|_{L^\infty_tH^3(\R^2)}^8
\|\na \vec{v}^j_t\|_{L^2_tH^{-\frac{1}{2}}(\G)}^2.
\end{split}
\ee
A direct computation yields
\ben
\begin{split}
[(\na_0 \eta^j)^4 (\na^2_0 \eta^j)\vec{v}^j]_t
=&(\na_0 \eta^j)^3(\na_0 \eta^j_t) (\na^2_0 \eta^j)\vec{v}^j
+(\na_0 \eta^j)^4 (\na^2_0 \eta^j_t)\vec{v}^j
+(\na_0 \eta^j)^4 (\na^2_0 \eta^j)\vec{v}^j_t\\
:=&I_1+I_2+I_3.
\end{split}
\enn
By \eqref{ee83}, the Sobolev embedding inequality and trace theorem  one deduces that
\ben
\begin{split}
\|I_1\|_{L^2_tL^2(\G)}^2
\leq& C \|\na_0\eta^j\|_{L^\infty_tL^\infty(\R^2)}^6
\|\na_0 \eta^j_t\|_{L^2_t L^4(\R^2)}^2
\|\na_0^2\eta^j\|_{L^\infty_tL^4(\R^2)}^2
\| \vec{v}^j\|_{L^\infty_t L^\infty(\G)}^2\\
\leq &C\|\eta^j\|_{L^\infty_tH^3(\R^2)}^8
\| \vec{v}^j\|_{L^\infty_t H^2}^2
\| \vec{v}^j\|_{L^2_t H^2}^2
\end{split}
\enn
and that
\ben
\begin{split}
\|I_2\|_{L^2_tL^2(\G)}^2
\leq& C  \|\na_0\eta^j\|_{L^\infty_tL^\infty(\R^2)}^8
\|\na_0^2 \eta^j_t\|_{L^2_t L^2(\R^2)}^2
\| \vec{v}^j\|_{L^\infty_t L^\infty(\G)}^2\\
\leq &C\|\eta^j\|_{L^\infty_tH^3(\R^2)}^8
\| \vec{v}^j\|_{L^\infty_t H^2}^2
\| \vec{v}^j\|_{L^2_t H^3}^2.
\end{split}
\enn
Sobolev embedding inequality gives
\ben
\begin{split}
\|I_3\|_{L^2_tL^2(\G)}^2
\leq
\|\na_0\eta^j\|_{L^\infty_tL^\infty(\R^2)}^8
\|\na_0^2\eta^j\|_{L^\infty_tL^4(\R^2)}^2
\|\vec{v}^j_t\|_{L^2_tL^4(\G)}^2
\leq C\|\eta^j\|_{L^\infty_tH^3(\R^2)}^{10}
\|\vec{v}^j_t\|_{L^2_tH^1}^2
\end{split}
\enn
which, in conjunction with the estimates for $I_1$ and $I_2$ indicates that
\be\label{ee89}
\|[(\na_0 \eta^j)^4 (\na^2_0 \eta^j)\vec{v}^j]_t\|_{L^2_tL^2(\G)}^2
\leq C\|\eta^j\|_{L^\infty_tH^3(\R^2)}^8
\| \vec{v}^j\|_{L^\infty_t H^2}^2
\| \vec{v}^j\|_{L^2_t H^3}^2
+C\|\eta^j\|_{L^\infty_tH^3(\R^2)}^{10}
\|\vec{v}^j_t\|_{L^2_tH^1}^2.
\ee
By a similar argument used in deriving \eqref{ee89} one gets
\be\label{ee91}
\|[(\na_0 \eta^j)^2 (\na^2_0 \eta^j)]_t\|_{L^2_tL^2(\G)}^2
\leq  C\|\eta^j\|_{L^\infty_tH^3(\R^2)}^4
\| \vec{v}^j\|_{L^2_t H^3}^2.
\ee
Collecting \eqref{ee90}-\eqref{ee91} and using \eqref{ee25} we have
\ben
\|\vec{G}_t\|_{L^2_tH^{-\frac{1}{2}}(\G)}^2
\leq C \|\{
w^{j},h^{j},\vec{v}^{j},q^{j}, \eta^{j}
\}\|^4
+\|\{
w^{j},h^{j},\vec{v}^{j},q^{j}, \eta^{j}
\}\|^{12}
\enn
which, along with \eqref{ee92} gives the desired estimates and completes the proof.

\end{proof}

The estimates for $\|G_4(w^j,h^j,\bar{\eta}^j)\|_{L^2_tH^{\frac{3}{2}}(\G)}^2$ and
$\|G_{4t}(w^j,h^j,\bar{\eta}^j)\|_{L^2_tH^{-\frac{1}{2}}(\G)}^2$ are as follows.
\begin{lemma}\label{l5} Let \eqref{eee15} and the assumptions in Proposition \ref{p4} hold. Then there exists a constant $C$ independent of $j$ and $t$ such that
\ben
\begin{split}
\|G_{4}\|_{L^2_tH^{\frac{3}{2}}(\G)}^2
+\|G_{4t}\|_{L^2_tH^{-\frac{1}{2}}(\G)}^2
\leq C \|\{
w^{j},h^{j},\vec{v}^{j},q^{j}, \eta^{j}
\}\|^4
+C\|\{
w^{j},h^{j},\vec{v}^{j},q^{j}, \eta^{j}\}\|^{10}
\end{split}
\enn
for $j\geq 2$ and $t>0$.
\end{lemma}
\begin{proof}
From \eqref{aae21} we know that
\be\label{ee37}
\begin{split}
G_4(w^j,h^j,\bar{\eta}^j)
\sim
\na_0 \bar{\eta}^j \na_0\eta^j \left(\na_0 w^j +w^j \na_0 h^j\right).
\end{split}
\ee
Then the Sobolev embedding inequality, trace theorem and Lemma \ref{pl7} entail that
\be\label{ee38}
\begin{split}
\|G_4\|_{L^2_tH^{\frac{3}{2}}(\G)}^2
\leq& C\|\na_0 \bar{\eta}^j\|_{L^\infty_tH^{\frac{3}{2}}(\G)}^2
\|\na_0 \eta^j\|_{L^\infty_tH^{\frac{3}{2}}(\R^2)}^2
\|\na w^j\|_{L^2_tH^{\frac{3}{2}}(\G)}^2\\
&+C\|\na_0 \bar{\eta}^j\|_{L^\infty_tH^{\frac{3}{2}}(\G)}^2
\|\na_0 \eta^j\|_{L^\infty_tH^{\frac{3}{2}}(\R^2)}^2\|w^j\|_{L^\infty_tH^{\frac{3}{2}}(\G)}^2
\|\na h^j\|_{L^2_tH^{\frac{3}{2}}(\G)}^2
\\
\leq& C \| \eta^j\|_{L^\infty_tH^{3}(\R^2)}^4
(
\|w^j\|_{L^2_tH^3}^2
+\|w^j\|_{L^\infty_tH^2}^2
\|h^j\|_{L^2_tH^3}^2
).
\end{split}
\ee
We proceed to estimating $\|G_{4t}\|_{L^2_tH^{-\frac{1}{2}}(\G)}^2$. It follows from the Sobolev embedding inequality, Lemma \ref{pl7} and \eqref{ee83} that
\be\label{ee93}
\begin{split}
\|\na_0& \bar{\eta}^j_t \na_0\eta^j \na_0 w^j\|_{L^2_tL^2(\G)}^2
+\|\na_0 \bar{\eta}^j \na_0\eta^j_t \na_0 w^j\|_{L^2_tL^2(\G)}^2\\
\leq &(\|\na_0 \bar{\eta}^j_t\|_{L^2_tL^4(\G)}^2
\|\na_0\eta^j\|_{L^\infty_tL^\infty(\R^2)}^2
+\|\na_0\bar{\eta}^j\|_{L^\infty_tL^\infty(\G)}^2
\|\na_0 \eta^j_t\|_{L^2_tL^4(\G)}^2
)\|\na_0 w^j\|_{L^\infty_tL^4(\G)}^2\\
\leq & C\|\eta^j\|_{L^\infty_tH^3(\R^2)}^4
\|\vec{v}^j\|_{L^2_tH^3}^2
\|w^j\|_{L^\infty_tH^2}^2.
\end{split}
\ee
For any $\psi\in H^{\frac{1}{2}}(\R^2)$, using Lemma \ref{pl7} and the trace theorem one gets
\ben
\begin{split}
\int_\G \na_0 \bar{\eta}^j \na_0\eta^j \na_0 w^j_t \psi dx
\leq& \|\na_0\bar{\eta}^j\na_0\eta^j \psi\|_{H^{\frac{1}{2}}(\G)}
\|\na_0 w^j_t\|_{H^{-\frac{1}{2}}(\G)}\\
\leq &
\|\na_0\bar{\eta}^j\na_0\eta^j \psi\|_{H^{\frac{1}{2}}(\G)}\|w^j_t\|_{H^{\frac{1}{2}}(\G)}\\
\leq & C \|\eta^j\|_{H^3(\R^2)}^2 \| w^j_t\|_{H^1}
\|\psi\|_{H^{\frac{1}{2}}(\R^2)}.
 \end{split}
 \enn
Thus,
\ben
\|\na_0 \bar{\eta}^j \na_0\eta^j \na_0 w^j_t\|_{L^2_tH^{-\frac{1}{2}}(\G)}^2
\leq C \|\eta^j\|_{L^\infty_tH^3(\R^2)}^4 \| w^j_t\|_{L^2_tH^1}^2,
\enn
which, in conjunction with \eqref{ee93} gives rise to
\be\label{ee39}
\|(\na_0 \bar{\eta}^j \na_0\eta^j \na_0 w^j)_t\|_{L^2_t H^{-\frac{1}{2}}(\G)}^2
\leq C\|\eta^j\|_{L^\infty_tH^3(\R^2)}^4
\big(\|\vec{v}^j\|_{L^2_tH^3}^2
\|w^j\|_{L^\infty_tH^2}^2
+\| w^j_t\|_{L^2_tH^1}^2
\big).
\ee
By a similar argument used in deriving \eqref{ee39}, one can deduce that
\be\label{ee40}
\begin{split}
\|(\na_0 \bar{\eta}^j \na_0\eta^j w^j\na_0 h^j)_t\|_{L^2_t H^{-\frac{1}{2}}(\G)}^2
\leq& C\|\eta^j\|_{L^\infty_tH^3(\R^2)}^4
\|\vec{v}^j\|_{L^2_tH^2}^2
\|w^j\|_{L^\infty_tH^2}^2 \|h^j\|_{L^\infty_tH^2}^2\\
&+C\|\eta^j\|_{L^\infty_tH^3}^4
 (\| w^j_t\|_{L^2_tH^1}^2\|h^j\|_{L^\infty_tH^2}^2
+\|w^j\|_{L^\infty_tH^2}^2\| h^j_t\|_{L^2_tH^1}^2).
\end{split}
\ee
Combining \eqref{ee39} and \eqref{ee40} we conclude that
\ben
\begin{split}
\|G_{4t}\|_{L^2_t H^{-\frac{1}{2}}(\G)}^2
\leq C\|\{
w^{j},h^{j},\vec{v}^{j},q^{j}, \eta^{j}
\}\|^4
+C\|\{
w^{j},h^{j},\vec{v}^{j},q^{j}, \eta^{j}\}\|^{10}
\end{split}
\enn
which, along with \eqref{ee38} indicates the desired estimates. The proof is finished.

\end{proof}

We proceed to estimate $F_4(w^{(j-1)},w^j,h^j,\vec{v}^{j},\bar{\eta}^j)$ and $F_{4t}(w^{(j-1)},w^j,h^j,\vec{v}^{j},\bar{\eta}^j)$.
Recalling the definition of $F_4(w^{(j-1)},w^j,h^j,\vec{v}^{j},\bar{\eta}^j)$ in \eqref{ee94} we have
\be\label{ee95}
\begin{split}
F_4(w^{(j-1)},w^j,h^j,\vec{v}^{j},\bar{\eta}^j)=
&\left\{(J^{j})^{-2}[(\al^j)^2+(\beta^j)^2+1]-1
\right\}\partial_3\big(\partial_3 w^{j}+w^{(j-1)}\partial_3 h^j\big)\\
 &-2\al^j(J^{j})^{-1}
\big(\partial_3\partial_1 w^j
+w^{(j-1)}\partial_3\partial_1 h^j
\big)\\
&-2\beta^j(J^{j})^{-1}
\big(\partial_3\partial_2 w^j
+w^{(j-1)}\partial_3\partial_2 h^j
\big)+L_4
\end{split}
\ee
with
\ben
L_4\sim \na\bar{\eta}^j\na w^j \vec{v}^j
+(\na\bar{\eta}^j)^2(\na^2\bar{\eta}^j) \na w^j
+ (\na\bar{\eta}^j)^2\na w^j\na h^j
+(\na\bar{\eta}^j)^2 (\na^2\bar{\eta}^j)w^j \na h^j
+\na w^j\bar{\eta}^j_t.
\enn
\begin{lemma}\label{l6} Suppose that \eqref{eee15} and the assumptions in Proposition \ref{p4} hold true. Then there exists a constant $C$ independent of $j$ and $t$ such that
\ben
\begin{split}
\|F_{4}\|_{L^2_tH^1}^2
\leq& C \|\{
w^{j},h^{j},\vec{v}^{j},q^{j}, \eta^{j}
\}\|^4
+C\|\{
w^{j},h^{j},\vec{v}^{j},q^{j}, \eta^{j}\}\|^{12}\\
&+C \|\{
w^{(j-1)},h^{(j-1)},\vec{v}^{\,(j-1)},q^{(j-1)}, \eta^{(j-1)}
\}\|^4\\
&+C\|\{
w^{(j-1)},h^{(j-1)},\vec{v}^{\,(j-1)},q^{(j-1)}, \eta^{(j-1)}\}\|^{12}
\end{split}
\enn
for $j\geq 2$ and $t>0$.
\end{lemma}
\begin{proof}
 It follows from the Sobolev embedding inequality, \eqref{eeea1} and Lemma \ref{pl7} that
\be\label{ee33}
\begin{split}
&\|\{(J^{j})^{-2}[(\al^j)^2+(\beta^j)^2+1]-1
\}\partial_3\big(\partial_3 w^{j}+w^{(j-1)}\partial_3 h^j\big)\|_{L^2_tH^1}^2\\
&\ \leq C\|\na\bar{\eta}^j\|_{L^\infty_tH^2}^4
\|\na^2 w^j\|_{L^2_tH^1}^2
+C\|\na\bar{\eta}^j\|_{L^\infty_tH^2(\R^2)}^4
\|\na( w^{(j-1)}\na h^j)\|_{L^2_tH^1}^2\\
&\ \leq C \|\eta^j\|_{L^\infty_tH^3}^4
(\|w^j\|_{L^2_tH^3}^2
+\|w^{(j-1)}\|_{L^\infty_tH^2}^2 \| h^{j}\|_{L^2_tH^3}^2
).
\end{split}
\ee
A similar argument leads to
\be\label{ee34}
\begin{split}
&\|-2\al^j(J^{j})^{-1}
\big(\partial_3\partial_1 w^j
+w^{(j-1)}\partial_3\partial_1 h^j
\big)
-2\beta^j(J^{j})^{-1}
\big(\partial_3\partial_2 w^j
+w^{(j-1)}\partial_3\partial_2 h^j
\big)\|_{L^2_tH^1}^2\\
&\ \leq C \|\eta^j\|_{L^\infty_tH^3(\R^2)}^2
(\|w^j\|_{L^2_tH^3}^2
+\|w^{(j-1)}\|_{L^\infty_tH^2}^2 \| h^{j}\|_{L^2_tH^3}^2
).
\end{split}
\ee
We next estimate the second term in $L_4$ by the Sobolev embedding inequality and Lemma \ref{pl7} as follows:
\be\label{ee35}
\begin{split}
\|(\na\bar{\eta}^j)^2(\na^2\bar{\eta}^j) \na w^j\|_{L^2_tH^1}^2
 \leq& \|\na\bar{\eta}^j\|_{L^\infty_tL^\infty}^2
\|\na^2\bar{\eta}^j\|_{L^\infty_tL^4}^4
\|\na w^j\|_{L^2_tL^\infty}^2\\
&+
\|\na\bar{\eta}^j\|_{L^\infty_tL^\infty}^4
\|\na^3\bar{\eta}^j\|_{L^\infty_tL^2}^2
\|\na w^j\|_{L^2_tL^\infty}^2\\
& +\|\na\bar{\eta}^j\|_{L^\infty_tL^\infty}^4
\|\na^2\bar{\eta}^j\|_{L^\infty_tL^4}^2
\|\na^2 w^j\|_{L^2_tL^4}^2\\
\leq &C \|\eta^j\|_{L^\infty_tH^3(\R^2)}^6
\| w^j\|_{L^2_tH^3}^2.
\end{split}
\ee
By a similar argument used in deriving \eqref{ee35} one can estimate the other terms in $L_4$ to conclude that
\be\label{ee36}
\begin{split}
\|L_4\|_{L^2_tH^1}^2
\leq C \|\{
w^{j},h^{j},\vec{v}^{j},q^{j}, \eta^{j}
\}\|^4
+C\|\{
w^{j},h^{j},\vec{v}^{j},q^{j}, \eta^{j}\}\|^{10}.
\end{split}
\ee
Collecting \eqref{ee33}, \eqref{ee34}, \eqref{ee36} and using \eqref{ee95} we derive the desired estimate. The proof is completed.

\end{proof}

\begin{lemma}\label{l7} Let \eqref{eee15} and the assumptions in Proposition \ref{p4} hold. Then there exists a constant $C$ independent of $j$ and $t$ such that
\ben
\begin{split}
\|F_{4t}\|_{L^2_t ({_0}H^1)^{'}}^2
\leq &C \|\{
w^{j},h^{j},\vec{v}^{j},q^{j}, \eta^{j}
\}\|^4
+C\|\{
w^{j},h^{j},\vec{v}^{j},q^{j}, \eta^{j}\}\|^{20}\\
&+C \|\{
w^{(j-1)},h^{(j-1)},\vec{v}^{\,(j-1)},q^{(j-1)}, \eta^{(j-1)}
\}\|^4\\
&+C\|\{
w^{(j-1)},h^{(j-1)},\vec{v}^{\,(j-1)},q^{(j-1)}, \eta^{(j-1)}\}\|^{20}
\end{split}
\enn
for $j\geq 2$ and $t>0$.
\end{lemma}
\begin{proof}
 From he definition of $\al^j$, $\beta^j$ and $J^j$ in \eqref{aae18} we know that
 \ben
 \{(J^j)^{-2}[(\al^j)^2+(\beta^j)^2+1]-1\}_t\sim \na\bar{\eta}^j\na\bar{\eta}^j_t+(\na\bar{\eta}^j)^2
\na\bar{\eta}^j_t,
 \enn
  which, along with the Sobolev embedding inequality, \eqref{ee83} and Lemma \ref{pl7} indicates that
\be\label{ee42}
\begin{split}
&\big\|\{(J^{j})^{-2}
[(\al^j)^2+(\beta^j)^2+1]-1
\}_t\partial_3\big(\partial_3 w^{j}+w^{(j-1)}\partial_3 h^j\big)
\big\|_{L^2_tL^2}^2\\
&\ \ \leq  C(\|\na \bar{\eta}^j\|_{L^\infty_tL^\infty}^2
+\|\na \bar{\eta}^j\|_{L^\infty_tL^\infty}^4)
\|\na \bar{\eta}^j_t\|_{L^2_tL^\infty}^2
(\|\partial_3^2 w^j\|_{L^\infty_tL^2}^2
+\|\partial_3 (w^{(j-1)}\partial_3 h^j)\|_{L^\infty_tL^2}^2)
\\
&\ \ \leq C(\|\eta^j\|_{L^\infty_tH^3(\R^2)}^2
+\|\eta^j\|_{L^\infty_tH^3(\R^2)}^4)
\|\vec{v}^j\|_{L^2_tH^3}^2
(\|w^j\|_{L^\infty_tH^2}^2
+\|w^{(j-1)}\|_{L^\infty_tH^2}^2\|h^j\|_{L^\infty_tH^2}^2).
\end{split}
\ee
On the other hand, for any $\psi\in {_0}H^1$ one gets
\be\label{ee43}
\begin{split}
\int_{\Om}& \{(J^{j})^{-2}
[(\al^j)^2+(\beta^j)^2+1]-1
\}\partial_3\big(\partial_3 w^{j}+w^{(j-1)}\partial_3 h^j\big)_t \psi dxdy\\
&=\int_{\G}
\{(J^{j})^{-2}
[(\al^j)^2+(\beta^j)^2+1]-1
\}\big(\partial_3 w^{j}+w^{(j-1)}\partial_3 h^j\big)_t \psi dx\\
&\ \ -\int_\Om \{(J^{j})^{-2}
[(\al^j)^2+(\beta^j)^2+1]-1
\}\big(\partial_3 w^{j}+w^{(j-1)}\partial_3 h^j\big)_t \partial_3\psi dxdy\\
&\ \ -\int_\Om \partial_3\{(J^{j})^{-2}
[(\al^j)^2+(\beta^j)^2+1]-1
\}\big(\partial_3 w^{j}+w^{(j-1)}\partial_3 h^j\big)_t \psi dxdy\\
&:=I_4+I_5+I_6.
\end{split}
\ee
From the first boundary condition in \eqref{ee2} we know that $\partial_3 w^{j}+w^{(j-1)}\partial_3 h^{j}=G_4(w^{(j-1)},h^{(j-1)},\bar{\eta}^{(j-1)})$ on $\G\times(0,\infty)$. Thus it follows from \eqref{eeea1} and Lemma \ref{pl7} that
\ben
\begin{split}
I_4\leq &C \|G_{4t}(w^{(j-1)},h^{(j-1)},\bar{\eta}^{(j-1)})
\|_{H^{-\frac{1}{2}}(\G)}
\big\|\big\{(J^j)^{-2}
[(\al^j)^2+(\beta^j)^2+1]-1
\big\} \psi\big\|_{H^{\frac{1}{2}}(\G)}\\
\leq &C\|G_{4t}(w^{(j-1)},h^{(j-1)},\bar{\eta}^{(j-1)})
\|_{H^{-\frac{1}{2}}(\G)}
\|\na\bar{\eta}^j\|_{H^\frac{3}{2}(\G)}^2
\|\psi\|_{H^{\frac{1}{2}}(\G)}\\
\leq &C\|G_{4t}(w^{(j-1)},h^{(j-1)},\bar{\eta}^{(j-1)})
\|_{H^{-\frac{1}{2}}(\G)}
\|\eta^j\|_{H^3(\R^2)}^2
\|\psi\|_{H^1}.
\end{split}
\enn
The Sobolev embedding inequality, \eqref{eeea1} and Lemma \ref{pl7} lead to
\ben
\begin{split}
I_5
\leq & \|\na\bar{\eta}^j\|_{L^\infty}^2
(\|\na w^j_t\|_{L^2}+\| w^{(j-1)}_t\|_{L^4}\|\na h^j\|_{L^4}+\| w^{(j-1)}\|_{L^\infty}\|\na h^j_t\|_{L^2})\|\na \psi\|_{L^2}\\
\leq & C \|\eta^j\|_{H^3(\R^2)}^2
 (\|w^j_t\|_{H^1}+\| w^{(j-1)}_t\|_{H^1}\| h^j\|_{H^2}+\| w^{(j-1)}\|_{H^2}\|h^j_t\|_{H^1})
 \|\psi\|_{H^1}.
\end{split}
\enn
 Similarly, it follows from the Sobolev embedding inequality and Lemma \ref{pl7} that
\ben
\begin{split}
I_6\leq &
(\|\na\bar{\eta}^j\|_{L^\infty}+
\|\na\bar{\eta}^j\|_{L^\infty}^2)\|\na^2\bar{\eta}^j\|_{L^4}
(\|\na w^j_t\|_{L^2}+\| w^{(j-1)}_t\|_{L^4}\|\na h^j\|_{L^4}+\| w^{(j-1)}\|_{L^\infty}\|\na h^j_t\|_{L^2})\| \psi\|_{L^4}\\
\leq & C (\|\eta^j\|_{H^3(\R^2)}^2
+\|\eta^j\|_{H^3(\R^2)}^4)
 (\|w^j_t\|_{H^1}+\| w^{(j-1)}_t\|_{H^1}\| h^j\|_{H^2}+\| w^{(j-1)}\|_{H^2}\|h^j_t\|_{H^1})
 \|\psi\|_{H^1}.
\end{split}
\enn
Substituting the above estimate for $I_4$, $I_5$ and $I_6$ into \eqref{ee43} we arrive at
\ben
\begin{split}
&\big\|\big\{(J^j)^{-2}
[(\al^j)^2+(\beta^j)^2+1]-1
\big\}\partial_3\big(\partial_3 w^{j}+w^{(j-1)}\partial_3 h^j\big)_t
\big\|_{L^2_t({_0}H^1)^{'}}^2\\
&\ \ \leq C(\|\eta^j\|_{L^\infty_tH^3(\R^2)}^4
+\|\eta^j\|_{L^\infty_tH^3(\R^2)}^8)
 (\|w^j_t\|_{L^2_tH^1}+\| w^{(j-1)}_t\|_{L^2_tH^1}\| h^j\|_{L^\infty_tH^2}+\| w^{(j-1)}\|_{L^\infty_tH^2}\|h^j_t\|_{L^2_tH^1})^2\\
 &\ \ \ \ +C\|\eta^j\|_{L^\infty_tH^3(\R^2)}^4
\|G_{4t}(w^{(j-1)},h^{(j-1)},\bar{\eta}^{(j-1)})
\|_{L^2_tH^{-\frac{1}{2}}(\G)}^2,
\end{split}
\enn
which, along with \eqref{ee42} and Lemma \ref{l5} gives rise to
\be\label{ee44}
\begin{split}
&\big\|\big\{\big[(J^j)^{-2}
((\al^j)^2+(\beta^j)^2+1)-1
\big]\partial_3\big(\partial_3 w^{j}+w^{(j-1)}\partial_3 h^j\big)\big\}_t\big\|_{L^2_t({_0}H^1)^{'}}^2\\
&\ \ \leq
 C \|\{
w^{j},h^{j},\vec{v}^{j},q^{j}, \eta^{j}
\}\|^4
+C\|\{
w^{j},h^{j},\vec{v}^{j},q^{j}, \eta^{j}\}\|^{20}\\
&\ \ \ \ \ +C \|\{
w^{(j-1)},h^{(j-1)},\vec{v}^{\,(j-1)},q^{(j-1)}, \eta^{(j-1)}
\}\|^4\\
&\ \ \ \ \ +C\|\{
w^{(j-1)},h^{(j-1)},\vec{v}^{\,(j-1)},q^{(j-1)}, \eta^{(j-1)}\}\|^{20}.
\end{split}
\ee
By a similar argument used in deriving \eqref{ee44} one deduces that
\be\label{ee45}
\begin{split}
&\big\|\big[2\al^j(J^j)^{-1}
\big(\partial_3\partial_1 w^j
+w^{(j-1)}\partial_3\partial_1 h^j
\big)\big]_t+\big[2\beta^j(J^j)^{-1}
\big(\partial_3\partial_2 w^j
+w^{(j-1)}\partial_3\partial_2 h^j
\big)\big]_t\big\|_{L^2_t({_0}H^1)^{'}}^2\\
&\ \ \leq
 C \|\{
w^{j},h^{j},\vec{v}^{j},q^{j}, \eta^{j}
\}\|^4
+C\|\{
w^{j},h^{j},\vec{v}^{j},q^{j}, \eta^{j}\}\|^{20}\\
&\ \ \ \ \ +C \|\{
w^{(j-1)},h^{(j-1)},\vec{v}^{\,(j-1)},q^{(j-1)}, \eta^{(j-1)}
\}\|^4\\
&\ \ \ \ \ +C\|\{
w^{(j-1)},h^{(j-1)},\vec{v}^{\,(j-1)},q^{(j-1)}, \eta^{(j-1)}\}\|^{20}.
\end{split}
\ee

We proceed to estimating $\|(L_4)_t\|_{L^2_t(_{0}H^1)^{'}}^2$. For any $\psi\in {_0}H^1$ it follows from \eqref{ee83} and Lemma \ref{pl7} that
\ben
\begin{split}
\int_{\Om} [(\na &\bar{\eta}^j)^2 (\na^2\bar{\eta}^j) \na w^j]_t \psi dxdy\\
\leq & \|\na\bar{\eta}^j\|_{L^\infty}
\|\na\bar{\eta}^j_t\|_{L^2}
\|\na^2\bar{\eta}^j\|_{L^4}
\|\na w^j\|_{L^\infty}
\|\psi\|_{L^4}\\
&+\|\na\bar{\eta}^j\|_{L^\infty}^2
\|\na^2\bar{\eta}^j_t\|_{L^2}
\|\na w^j\|_{L^4}
\|\psi\|_{L^4}
+\|\na\bar{\eta}^j\|_{L^\infty}^2
\|\na^2\bar{\eta}^j\|_{L^4}
\|\na w^j_t\|_{L^2}
\|\psi\|_{L^4}\\
\leq& C \|\eta^j\|_{H^3(\R^2)}^2
\|\vec{v}^j\|_{H^2}
\|w^j\|_{H^3}
\|\psi\|_{H^1}
+C\|\eta^j\|_{H^3(\R^2)}^3
\|w^j_t\|_{H^1}
\|\psi\|_{H^1}.
\end{split}
\enn
Thus,
\be\label{ee46}
\begin{split}
\|[(\na \bar{\eta}^j)^2 (\na^2\bar{\eta}^j) \na w^j]_t\|_{L^2_t({_0}H^1)^{'}}^2
\leq &C\|\eta^j\|_{L^\infty_tH^3(\R^2)}^4
\|\vec{v}^j\|_{L^\infty_tH^2}^2
\|w^j\|_{L^2_tH^3}^2\\
&+C\|\eta^j\|_{L^\infty_tH^3(\R^2)}^6
\|w^j_t\|_{L^2_tH^1}^2.
\end{split}
\ee
By a similar argument used in deriving \eqref{ee46} one deduces that
\be\label{ee47}
\begin{split}
\|[\na \bar{\eta}^j  \na w^j\vec{v}^j]_t\|_{L^2_t({_0}H^1)^{'}}^2
\leq &C
\|\vec{v}^j\|_{L^\infty_tH^2}^4
\|w^j\|_{L^2_tH^3}^2
+C\|\eta^j\|_{L^\infty_tH^3(\R^2)}^2
\|w^j_t\|_{L^2_tH^1}^2\|\vec{v}^j\|_{L^\infty_tH^2}^2
\\
&+C\|\eta^j\|_{L^\infty_tH^3(\R^2)}^2
\|\vec{v}^j_t\|_{L^2_tH^1}^2
\|w^j\|_{L^\infty_tH^2}^2
\end{split}
\ee
and that
\be\label{ee48}
\begin{split}
\|[(\na& \bar{\eta}^j)^2  \na w^j \na h^j]_t\|_{L^2_t({_0}H^1)^{'}}^2
+\|[(\na \bar{\eta}^j)^2 (\na^2\bar{\eta}^j)  w^{j}\na h^j]_t\|_{L^2_t({_0}H^1)^{'}}^2\\
\leq &C(\|\eta^j\|_{L^\infty_tH^3(\R^2)}^2
+\|\eta^j\|_{L^\infty_tH^3(\R^2)}^4)
\|\vec{v}^j\|_{L^\infty_tH^2}^2
\|h^j\|_{L^\infty_tH^2}^2\|w^j\|_{L^2_tH^3}^2\\
&+C(\|\eta^j\|_{L^\infty_tH^3(\R^2)}^4
+\|\eta^j\|_{L^\infty_tH^3(\R^2)}^6)
(\|w^j_t\|_{L^2_tH^1}^2\|h^j\|_{L^\infty_tH^2}^2
+\|w^j\|_{L^\infty_tH^2}^2
\|h^j_t\|_{L^2_tH^1}^2
).
\end{split}
\ee
By \eqref{ee83} one gets
\be\label{ee49}
\begin{split}
\|[\na w^j\bar{\eta}^j_t]_t\|_{L^2_tL^2}^2
\leq \|w^j_t\|_{L^2_tH^1}^2\|\vec{v}^j\|_{L^\infty_tH^2}^2
+\|w^j\|_{L^\infty_tH^2}^2\|\vec{v}^j_t\|_{L^2_tH^1}^2.
\end{split}
\ee
Collecting \eqref{ee46}-\eqref{ee49} we conclude that
\be\label{ee50}
\begin{split}
\|(L_4)_t\|_{L^2_t({_0}H^1)^{'}}^2
\leq C\|\{
w^{j},h^{j},\vec{v}^{j},q^{j}, \eta^{j}
\}\|^4
+C\|\{
w^{j},h^{j},\vec{v}^{j},q^{j}, \eta^{j}\}\|^{10}.
\end{split}
\ee
Combining \eqref{ee44}, \eqref{ee45}, \eqref{ee50} and using \eqref{ee95} we derive the desired estimate. The proof is completed.

\end{proof}
 It remains to estimate $F_5(h^j,\vec{v}^{j},\bar{\eta}^j)$ and $F_{5t}(h^j,\vec{v}^{j},\bar{\eta}^j)$.
From \eqref{aae9} we know that
\be\label{ee51}
\begin{split}
F_5(h^j,\vec{v}^{j},\bar{\eta}^j)=&\big\{(J^j)^{-2}
[(\al^j)^2+(\beta^j)^2+1]-1
\big\}\partial_3^2 h^{j}\\
&-2\al^j(J^j)^{-1}
\partial_3\partial_1 h^j
-2\beta^j(J^j)^{-1}
\partial_3\partial_2 h^j
+L_5
\end{split}
\ee
with
\ben
L_5\sim \vec{v}^j \na\bar{\eta}^j \na h^j
+(\na\bar{\eta}^j)^2 (\na^2\bar{\eta}^j)\na h^j
+(\na\bar{\eta}^j)^2 (\na h^j)^2
+(\na h^j)\bar{\eta}^j_t.
\enn
\begin{lemma}\label{l8} Let \eqref{eee15} and the assumptions in Proposition \ref{p4} hold. Then there exists a constant $C$ independent of $j$ and $t$ such that
\ben
\begin{split}
\|F_5\|_{L^2_tH^1}^2
+\|F_{5t}\|_{L^2_t({^0}H^1)^{'}}^2
\leq
C \|\{
w^{j},h^{j},\vec{v}^{j},q^{j}, \eta^{j}
\}\|^4
+C\|\{
w^{j},h^{j},\vec{v}^{j},q^{j}, \eta^{j}\}\|^{8}
\end{split}
\enn
for $j\geq 2$ and $t>0$.
\end{lemma}

\begin{proof}
 First, by a similar argument used in deriving \eqref{ee33} and \eqref{ee34} one deduces that
\be\label{ee52}
\begin{split}
&\|\big\{(J^j)^{-2}
[(\al^j)^2+(\beta^j)^2+1]-1
\big\}\partial_3^2 h^{j}\|_{L^2_tH^1}^2
+\|2\al^j(J^j)^{-1}
\partial_3\partial_1 h^j
+2\beta^j(J^j)^{-1}
\partial_3\partial_2 h^j\|_{L^2_tH^1}^2\\
&\ \leq C\|{\eta}^j\|_{L^\infty_tH^3(\R^2)}^4
\|h^j\|_{L^2_tH^3}^2+\|{\eta}^j\|_{L^\infty_tH^3(\R^2)}^2
\|h^j\|_{L^2_tH^3}^2.
\end{split}
\ee
Replacing the $w^j$ in \eqref{ee35} with $h^j$ one gets
\be\label{ee53}
\begin{split}
\|(\na\bar{\eta}^j)^2(\na^2\bar{\eta}^j) \na h^j\|_{L^2_tH^1}^2
\leq
C\|\eta^j\|_{L^\infty_tH^3(\R^2)}^6
\|h^j\|_{L^2_tH^3}^2.
\end{split}
\ee
It follows from the Sobolev embedding inequality and Lemma \ref{pl7} that
\be\label{ee55}
\begin{split}
\|\vec{v}^j\na\bar{\eta}^j \na h^j\|_{L^2_tH^1}^2
+\|(\na\bar{\eta}^j)^2 (\na h^j)^2\|_{L^2_tH^1}^2
\leq& C\|\eta^j\|_{L^\infty_tH^3(\R^2)}^2
\|\vec{v}^j\|_{L^\infty_tH^2}^2
\|h^j\|_{L^2_tH^3}^2\\
&+C\|\eta^j\|_{L^\infty_tH^3(\R^2)}^4
\|h^j\|_{L^\infty_tH^2}^2
\|h^j\|_{L^2_tH^3}^2.
\end{split}
\ee
The Sobolev embedding inequality and
 \eqref{ee83} entail that
\be\label{ee96}
\|(\na h^j)\bar{\eta}^j_t\|_{L^2_tH^1}^2
\leq C\|h^j\|_{L^2_tH^3}^2\|\vec{v}^j\|_{L^\infty_tH^2}^2.
\ee
 Collecting \eqref{ee52}-\eqref{ee96} and using \eqref{ee51} we obtain
 \be\label{ee100}
\begin{split}
\|F_5\|_{L^2_tH^1}^2
\leq
C \|\{
w^{j},h^{j},\vec{v}^{j},q^{j}, \eta^{j}
\}\|^4
+C\|\{
w^{j},h^{j},\vec{v}^{j},q^{j}, \eta^{j}\}\|^{8}.
\end{split}
\ee

We proceed to estimating $\|F_{5t}\|_{L^2_t({^0}H^1)^{'}}^2$.
A similar arguments used in deriving \eqref{ee42} leads to
\be\label{eeea2}
\begin{split}
&\big\|\big\{(J^j)^{-2}
\big[(\al^j)^2+(\beta^j)^2+1\big]-1
\big\}_t\partial_3^2 h^{j}
\big\|_{L^2_tL^2}^2\\
&\ \ \leq C(\|{\eta}^j\|_{L^\infty_tH^3(\R^2)}^2
+\|{\eta}^j\|_{L^\infty_tH^3(\R^2)}^4)
\|\vec{v}^j\|_{L^2_tH^3}^2
\|h^j\|_{L^\infty_tH^2}^2.
\end{split}
\ee
For any $\psi\in {^0}H^1$, integration by parts along with the Sobolev embedding inequality and Lemma \ref{pl7} leads to
\ben
\begin{split}
\int_{\Om}&\big\{(J^j)^{-2}
\big((\al^j)^2+(\beta^j)^2+1\big)-1
\big\}\partial_3^2 h^{j}_t\psi dxdy\\
=
&-\int_{\Om}\big\{(J^j)^{-2}
\big((\al^j)^2+(\beta^j)^2+1\big)-1
\big\}\partial_3 h^{j}_t\partial_3\psi dxdy\\
&-\int_{\Om}\partial_3\big\{(J^j)^{-2}
\big((\al^j)^2+(\beta^j)^2+1\big)-1
\big\}\partial_3 h^{j}_t\psi dxdy\\
\leq& \|\na\bar{\eta}^j\|_{L^\infty}^2
\|\na h^j_t\|_{L^2}
\|\na \psi\|_{L^2}
+(\|\na\bar{\eta}^j\|_{L^\infty}
+\|\na\bar{\eta}^j\|_{L^\infty}^2
)
\|\na^2\bar{\eta}^j\|_{L^4}
\|\na h^j_t\|_{L^2}
\| \psi\|_{L^4}\\
\leq& C(\|{\eta}^j\|_{H^3(\R^2)}^2
+\|{\eta}^j\|_{H^3(\R^2)}^3
)
\|h^j_t\|_{H^1}
\|\psi\|_{H^1}.
\end{split}
\enn
Thus,
\ben
\big\|\big\{(J^j)^{-2}
\big((\al^j)^2+(\beta^j)^2+1\big)-1
\big\}\partial_3^2 h^{j}_t
\big\|_{L^2_t({^0}H^1)^{'}}^2
\leq C(\|{\eta}^j\|_{L^\infty_tH^3(\R^2)}^4
+\|{\eta}^j\|_{L^\infty_tH^3(\R^2)}^6)
\|h^j_t\|_{L^2_tH^1}^2,
\enn
which, along with \eqref{eeea2} gives
\be\label{ee56}
\begin{split}
\big\|\big\{&\big[(J^j)^{-2}
\big((\al^j)^2+(\beta^j)^2+1\big)-1
\big]\partial_3^2 h^{j}\big\}_t
\big\|_{L^2_t({^0}H^1)^{'}}^2\\
\leq &C(\|{\eta}^j\|_{L^\infty_tH^3(\R^2)}^2
+\|{\eta}^j\|_{L^\infty_tH^3(\R^2)}^4)
\|\vec{v}^j\|_{L^2_tH^3}^2
\|h^j\|_{L^\infty_tH^2}^2\\
&+C(\|{\eta}^j\|_{L^\infty_tH^3(\R^2)}^4
+\|{\eta}^j\|_{L^\infty_tH^3(\R^2)}^6)
\|h^j_t\|_{L^2_tH^1}^2.
\end{split}
\ee
By a similar argument used in deriving \eqref{ee56} one deduces that
\be\label{ee57}
\begin{split}
\big\|\big[&2\al^j(J^j)^{-1}
\partial_3\partial_1 h^j
\big]_t\big\|_{L^2_t({^0}H^1)^{'}}^2
+\big\|\big[2\beta^j(J^j)^{-1}
\partial_3\partial_2 h^j\big]_t
\big\|_{L^2_t({^0}H^1)^{'}}^2\\
\leq &C\big(1+\|{\eta}^j\|_{L^\infty_tH^3(\R^2)}^2\big)
\|\vec{v}^j\|_{L^2_tH^3}^2
\|h^j\|_{L^\infty_tH^2}^2
+C\big(\|{\eta}^j\|_{L^\infty_tH^3(\R^2)}^2
+\|{\eta}^j\|_{L^\infty_tH^3(\R^2)}^4\big)
\|h^j_t\|_{L^2_tH^1}^2.
\end{split}
\ee
We next estimate each term in $(L_5)_t$. Replacing the $w^j$ in \eqref{ee46}, \eqref{ee47} and \eqref{ee49} with $h^j$, one deduces that
\be\label{ee58}
\begin{split}
\|[(\na \bar{\eta}^j)^2 (\na^2\bar{\eta}^j) \na h^j]_t\|_{L^2_t({^0}H^1)^{'}}^2
\leq &C\|\eta^j\|_{L^\infty_tH^3(\R^2)}^4
\|\vec{v}^j\|_{L^\infty_tH^2}^2
\|h^j\|_{L^2_tH^3}^2
+C\|\eta^j\|_{L^\infty_tH^3(\R^2)}^6
\|h^j_t\|_{L^2_tH^1}^2
\end{split}
\ee
and that
\be\label{ee59}
\begin{split}
\|[\vec{v}^j\na \bar{\eta}^j  \na h^j]_t\|_{L^2_t({^0}H^1)^{'}}^2
\leq &C
\|\vec{v}^j\|_{L^\infty_tH^2}^4
\|h^j\|_{L^2_tH^3}^2
+C\|\eta^j\|_{L^\infty_tH^3(\R^2)}^2
\|h^j_t\|_{L^2_tH^1}^2\|\vec{v}^j\|_{L^\infty_tH^2}^2
\\
&+C\|\eta^j\|_{L^\infty_tH^3(\R^2)}^2
\|\vec{v}^j_t\|_{L^2_tH^1}^2
\|h^j\|_{L^\infty_tH^2}^2
\end{split}
\ee
and that
\be\label{ee97}
\|[\na h^j\bar{\eta}^j_t]_t\|_{L^2_tL^2}^2
\leq \|h^j_t\|_{L^2_tH^1}^2 \|\vec{v}^j\|_{L^\infty_tH^2}^2
+\|h^j\|_{L^\infty_tH^2}^2\|\vec{v}^j_t\|_{L^2_tH^1}^2.
\ee
For any $\psi\in {^0}H^1$, it follows from
 Lemma \ref{pl7} and \eqref{ee83} that
 \ben
 \begin{split}
 \int_{\Om}[(\na \bar{\eta}^j)^2  (\na h^j)^2]_{t}\psi\,dxdy
 \leq &C\|\na\bar{\eta}^j\|_{L^\infty}
 \|\na\bar{\eta}^j_t\|_{L^4}
 \|\na h^j\|_{L^4}^2
 \|\psi\|_{L^4}\\
 &+C\|\na\bar{\eta}^j\|_{L^\infty}^2
 \|\na h^j\|_{L^4}
 \|\na h^j_t\|_{L^2}
 \|\psi\|_{L^4}\\
 \leq &C\|\eta^j\|_{H^3(\R^2)}
\|\vec{v}^j\|_{H^2}
\|h^j\|_{H^2}^2
\|\psi\|_{H^1}\\
&+C\|\eta^j\|_{H^3(\R^2)}^2
\|h^j\|_{H^2}
\|h^j_t\|_{H^1}
\|\psi\|_{H^1},
 \end{split}
 \enn
 which, implies that
\be\label{ee60}
\begin{split}
\|[(\na \bar{\eta}^j)^2  (\na h^j)^2]_{t}\|_{L^2_t({^0}H^1)^{'}}^2
\leq& C
\|\eta^j\|_{L^\infty_tH^3(\R^2)}^2
\|\vec{v}^j\|_{L^\infty_tH^2}^2
\|h^j\|_{L^2_tH^3}^4\\
&+C\|\eta^j\|_{L^\infty_tH^3(\R^2)}^4
\|h^j\|_{L^\infty_tH^2}^2
\|h^j_t\|_{L^2_tH^1}^2.
\end{split}
\ee
Collecting \eqref{ee58}-\eqref{ee60} we arrive at
\ben
\begin{split}
\|(L_5)_{t}\|_{L^2_t({^0}H^1)^{'}}^2
\leq
C \|\{
w^{j},h^{j},\vec{v}^{j},q^{j}, \eta^{j}
\}\|^4
+C\|\{
w^{j},h^{j},\vec{v}^{j},q^{j}, \eta^{j}\}\|^{8},
\end{split}
\enn
which, in conjunction with \eqref{ee56}, \eqref{ee57} and \eqref{ee51} gives rise to
\be\label{ee61}
\|F_{5t}\|_{L^2_t({^0}H^1)^{'}}^2
\leq
C \|\{
w^{j},h^{j},\vec{v}^{j},q^{j}, \eta^{j}
\}\|^4
+C\|\{
w^{j},h^{j},\vec{v}^{j},q^{j}, \eta^{j}\}\|^{8}.
\ee
Combining \eqref{ee100} and \eqref{ee61},
 we derive the desired estimate. The proof is completed.

\end{proof}
\subsection{Proof of Proposition \ref{p4}}\label{ss2}
First, applying Proposition \ref{p1} with $a(x_1,x_2,y,t)\equiv 0$ to system \eqref{eee4}-\eqref{eee5} one deduces that there is a constant $C_7$ independent of $t$ such that
\be\label{ae6}
\begin{split}
\|\{w^{1},h^{1},\vec{v}^{\,1},q^{1},
\eta^{1}\}\|^2
\leq &C_7 (\|w_0\|_{H^2}+\|h_0\|_{H^2}+\|\vec{v}_0\|_{H^2}
+\|\eta_0\|_{H^3(\R^2)})^2,\quad \forall t>0,\\
\|\{w^{2},h^{2},\vec{v}^{\,2},q^{2},
\eta^{2}\}\|^2
\leq &C_7 (\|w_0\|_{H^2}+\|h_0\|_{H^2}+\|\vec{v}_0\|_{H^2}
+\|\eta_0\|_{H^3(\R^2)})^2,\quad \forall t>0.
\end{split}
\ee
Then we assert that for all $j\geq 2$, system \eqref{ee1}-\eqref{eee7} admits a unique solution and that the following holds for all $j\geq 1$:
\be\label{ae5}
\begin{split}
\|\{w^{j},h^{j},\vec{v}^{\,j},q^{j},
\eta^{j}\}\|^{2}\leq C_7 (\|w_0\|_{H^2}+\|h_0\|_{H^2}+\|\vec{v}_0\|_{H^2}
+\|\eta_0\|_{H^3(\R^2)})^2,\quad \forall t>0,
\end{split}
\ee
provided \eqref{ae3} and \eqref{ae4}.
We next prove \eqref{ae5} by the argument of induction.
Assume that
\be\label{ae26}
\begin{split}
\|\{w^{k},h^{k},\vec{v}^{\,k},q^{k},
\eta^{k}\}\|^{2}\leq C_7 (\|w_0\|_{H^2}+\|h_0\|_{H^2}+\|\vec{v}_0\|_{H^2}
+\|\eta_0\|_{H^3(\R^2)})^2,\quad \forall t>0
\end{split}
\ee
holds true for all
$1\leq k\leq j$ with $j\geq 2$.
From the Sobolev embedding inequality and Lemma \ref{pl7} we know that there exists constants $C,\,C_5>0$ independent of $k$ and $t$ such that
\be\label{ae27}
\|\bar{\eta}^k+\partial_3 \bar{\eta}^k(1+y)\|_{L^\infty_tL^\infty}
\leq C\|\bar{\eta}^k\|_{L^\infty_tH^3}
\leq C_5 \|\eta^k\|_{L^\infty_tH^3(\R^2)}.
\ee
Assume that the initial data satisfy
\be\label{ae3}
\begin{split}
C_5\sqrt{C_7}(\|w_0\|_{H^2}+\|h_0\|_{H^2}
+\|\vec{v}_0\|_{H^2}
+\|\eta_0\|_{H^3(\R^2)})<\frac{1}{2},\\
C_1(C_2+1)C_7 (\|w_0\|_{H^2}+\|h_0\|_{H^2}
+\|\vec{v}_0\|_{H^2}
+\|\eta_0\|_{H^3(\R^2)})^2<\frac{1}{2},
\end{split}
\ee
where the constant $C_1$ and $C_2$ are defined in Proposition \ref{p1}.
Then it follows from \eqref{ae27}, \eqref{ae26} and the first inequality in \eqref{ae3} that
\ben
\|\bar{\eta}^k+\partial_3 \bar{\eta}^k(1+y)\|_{L^\infty_tL^\infty}
\leq C_5 \|\eta^k\|_{L^\infty_tH^3(\R^2)}
\leq C_5 \|\{w^{k},h^{k},\vec{v}^{\,k},q^{k},
\eta^{k}\}\|
 <\frac{1}{2},
\enn
which, along with the definition $J^k=1+\bar{\eta}^k+\partial_3 \bar{\eta}^k(1+y)$ indicates that
\be\label{ae23}
\frac{1}{2}<J^k<\frac{3}{2}\qquad \text{in}\ \   \Om\times (0,\infty)
\ee
holds for all $1\leq k\leq j$ with $j\geq 2$. Moreover, it follows from \eqref{ae26} and the second inequality in \eqref{ae3} that
\be\label{ae2}
C_1(C_2+1)|||w^{k}|||^2\leq C_1(C_2+1)\|\{w^{k},h^{k},\vec{v}^{\,k},q^{k},
\eta^{k}\}\|^{2}<\frac{1}{2}
\ee
for all $1\leq k\leq j$ with $j\geq 2$.
With the regularity \eqref{ae26} on $(w^{k},h^{k},\vec{v}^{\,k},q^{k},
\eta^{k})$, one easily deduces from the trace theorem that
\ben
\left\{
\begin{array}{lll}
G_1(\vec{v}^{\,k}(x_1,x_2,y,0),
\bar{\eta}^k(x_1,x_2,y,0))=G_1(\vec{v}_0,\bar{\eta}_0)\quad \text{on}\ \ \G,\\
 G_2(\vec{v}^{\,k}(x_1,x_2,y,0),
\bar{\eta}^k(x_1,x_2,y,0))=G_2(\vec{v}_0,\bar{\eta}_0)\quad \text{on}\ \ \G,\\
 G_4(w^k(x_1,x_2,y,0),h^k(x_1,x_2,y,0),\bar{\eta}^k(x_1,x_2,y,0))
=G_4(w_0,h_0,\bar{\eta}_0)\quad \text{on}\ \ \G,\\
\end{array}
\right.
 \enn
 holds for all $1\leq k\leq j$ with $j\geq 2$, which along with
 \eqref{ne55} indicates that the compatibility conditions \eqref{ne51} holds.

 Since the compatibility conditions \eqref{ne51} are fulfilled, with \eqref{ae23} and \eqref{ae2} in hand, we can apply Proposition \ref{p1} to system \eqref{ee1}-\eqref{eee7}, and use Lemma \ref{l2}- Lemma \ref{l8} to conclude that system \eqref{ee1}-\eqref{eee7} admits a unique solution $(w^{(j+1)},h^{(j+1)},\vec{v}^{\,(j+1)},q^{(j+1)},
\eta^{(j+1)})$ satisfying
\be\label{ae24}
\begin{split}
\|\{w&^{(j+1)},h^{(j+1)},\vec{v}^{\,(j+1)},q^{(j+1)},
\eta^{(j+1)}\}\|^2\\
\leq& C_6
\|\{w^{j},h^{j},\vec{v}^{j},q^{j},
\eta^{j}\}\|^4+C_6\|\{w^{j},h^{j},\vec{v}^{j},q^{j},
\eta^{j}\}\|^{20}\\
&+ C_6
\|\{w^{(j-1)},h^{(j-1)},\vec{v}^{\,(j-1)},q^{(j-1)},
\eta^{(j-1)}\}\|^4\\
&+C_6\|\{w^{(j-1)},h^{(j-1)},\vec{v}^{\,(j-1)},q^{(j-1)},
\eta^{(j-1)}\}\|^{20}
\end{split}
\ee
for all $t>0$, where the constant $C_6>0$ independent of $j$ and $t$. Assume further
\be\label{ae4}
\begin{split}
C_6&C_7 (\|w_0\|_{H^2}+\|h_0\|_{H^2}+\|\vec{v}_0\|_{H^2}
+\|\eta_0\|_{H^3(\R^2)})^2\\
&+C_6C_7^{9} (\|w_0\|_{H^2}+\|h_0\|_{H^2}+\|\vec{v}_0\|_{H^2}
+\|\eta_0\|_{H^3(\R^2)})^{18}
<\frac{1}{2}.
\end{split}
\ee
Then it follows from \eqref{ae24}, \eqref{ae26} and \eqref{ae4} that
\ben
\begin{split}
\|\{w^{(j+1)},h^{(j+1)},\vec{v}^{\,(j+1)},q^{(j+1)},
\eta^{(j+1)}\}\|^2
<&\frac{1}{2}
\|\{w^{j},h^{j},\vec{v}^{\,j},q^{j},
\eta^{j}\}\|^2\\
&+\frac{1}{2}
\|\{w^{(j-1)},h^{(j-1)},\vec{v}^{\,(j-1)},q^{(j-1)},
\eta^{(j-1)}\}\|^2\\
\leq &C_7 (\|w_0\|_{H^2}+\|h_0\|_{H^2}+\|\vec{v}_0\|_{H^2}
+\|\eta_0\|_{H^3(\R^2)})^2
\end{split}
\enn
which, implies that \eqref{ae26} holds true for $k=j+1$. We thus proved that \eqref{ae5} holds for all $j\geq 1$ by the argument of induction. \eqref{eee8} follows directly from \eqref{ae5} and we proceed to proving \eqref{eee14}. Combining \eqref{ae5} and the first inequality in \eqref{ae3} one deduces that \eqref{ae23} is true for all $k\geq 1$ and thus derives \eqref{eee14}.
 The proof is completed.

\endProof
\section{Proof of Theorem \ref{t2} and Theorem \ref{t1}}
\textbf{\emph{Proof of Theorem \ref{t2}}.}
 For $j\geq 3$, define
 \ben
 \begin{split}
 \de &w^{(j+1)}=w^{(j+1)}-w^j,
  \quad \de h^{(j+1)}=h^{(j+1)}-h^j,
   \quad \de \vec{v}^{\,(j+1)}=\vec{v}^{\,(j+1)}-\vec{v}^j,\\
    \de &q^{(j+1)}=q^{(j+1)}-q^j,
    \quad
 \de \eta^{(j+1)}=\eta^{(j+1)}-\eta^j.
 \end{split}
  \enn
  Then it follows from \eqref{ee1} that
\be\label{eee1}
\left\{
\begin{array}{lll}
\de w^{(j+1)}_t-\Delta \de w^{(j+1)}-\nabla\cdot(w^{j}\nabla \de h^{(j+1)})=\nabla\cdot(\de w^{j}\nabla  h^{j})+\de F_4
\quad\text{in}\ \  \Omega\times(0,\infty),\\
\de h^{(j+1)}_t-\Delta  \de h^{(j+1)}-\de w^{(j+1)}=\de F_5,\\
\de \vec{v}^{\,(j+1)}_t-\Delta \de\vec{v}^{\,(j+1)}+\nabla \de q^{(j+1)} +\de w^{(j+1)}\nabla \phi=\de \vec{F},\\
\na\cdot \de \vec{v}^{\,(j+1)}=0,\\
(\de w^{(j+1)},
\de h^{(j+1)},
\de \vec{v}^{\,(j+1)})(x_1,x_2,y,0)=(0,0,\mathbf{0}),\quad \de \eta^{(j+1)}(x_1,x_2,0)=0,
\end{array}
\right.
\ee
where
\ben
\begin{split}
\de F_4&:=
F_4(w^{(j-1)},w^j,h^j,\vec{v}^{j},\bar{\eta}^j)
-F_4(w^{(j-2)},w^{(j-1)},h^{(j-1)},
\vec{v}^{\,(j-1)},\bar{\eta}^{(j-1)}),\\
\de F_5&:=F_5(h^j,\vec{v}^{j},\bar{\eta}^j)
-F_5(h^{(j-1)},\vec{v}^{\,(j-1)},\bar{\eta}^{(j-1)}),\\
\de \vec{F}&:=\vec{F}(w^j,\vec{v}^{j},\na q^j,\bar{\eta}^j)
-\vec{F}(w^{(j-1)},\vec{v}^{\,(j-1)},\na q^{(j-1)},\bar{\eta}^{(j-1)}).
\end{split}
\enn
The boundary conditions on $\Gamma\times(0,\infty)$ follows from \eqref{ee2}:
\be\label{eee2}
\left\{
\begin{array}{lll}
\partial_3 \de w^{(j+1)}+w^{j}\partial_3 \de h^{(j+1)}=-\de w^{j}\partial_3 h^{j}+\de G_4,\qquad
 \de h^{(j+1)}=0,\\
\partial_3 \de v^{(j+1)}_1+\partial_1 \de v^{(j+1)}_3=\de G_1,
\qquad
\partial_3\de v^{(j+1)}_2+\partial_2 \de v^{(j+1)}_3=\de G_2,\\
\de \eta^{(j+1)}_t=\de v^{(j+1)}_3,\qquad
\de q^{(j+1)}-2\partial_3 \de v^{(j+1)}_3=\gamma\,\de\eta^{(j+1)}-\sigma\,\Delta_0\de\eta^{(j+1)}
-\de G_3,
\end{array}
\right.
\ee
where
\ben
\begin{split}
\de G_4&:=G_4(w^j,h^j,\bar{\eta}^j)
-G_4(w^{(j-1)},h^{(j-1)},\bar{\eta}^{(j-1)}),\qquad
\de G_1:=G_1(\vec{v}^{j},\bar{\eta}^j)
-G_1(\vec{v}^{\,(j-1)},\bar{\eta}^{(j-1)}),\\
\de G_2&:=G_2(\vec{v}^{j},\bar{\eta}^j)
-G_2(\vec{v}^{\,(j-1)},\bar{\eta}^{(j-1)}),\qquad
\de G_3:=G_3
(\vec{v}^{j},\bar{\eta}^j)-G_3
(\vec{v}^{\,(j-1)},\bar{\eta}^{(j-1)}).
\end{split}
\enn
The boundary conditions on $S_B\times(0,\infty)$ follows from \eqref{eee7}:
\be\label{eee3}
\de w^{(j+1)}=0,\quad\partial_3 \de h^{(j+1)}=0,
\quad \de\vec{v}^{\,(j+1)}=\mathbf{0}.
\ee
By a similar argument used in deriving \eqref{ae24}, that is applying Proposition \ref{p1} to system \eqref{eee1}-\eqref{eee3} and following the procedure in  Subsection \ref{ss1} to estimate the nonlinear terms on the right-hand side of each equation in \eqref{eee1}-\eqref{eee2}, one can deduce that
\be\label{ae1}
\begin{split}
\|\{\de& w^{(j+1)},\de h^{(j+1)},\de \vec{v}^{\,(j+1)},\de q^{(j+1)},
\de \eta^{(j+1)}\}\|^2\\
\leq& C_8(j)\|\{\de w^{j},\de h^{j},\de \vec{v}^{j},\de q^{j},
\de \eta^{j}\}\|^2\\
&+C_9(j)\|\{\de w^{(j-1)},\de h^{(j-1)},\de \vec{v}^{\,(j-1)},\de q^{(j-1)},
\de \eta^{(j-1)}\}\|^2
\end{split}
\ee
where
\ben
\begin{split}
C_8(j):=&
C\|\{w^j,h^j,\vec{v}^{\,j},q^j,\eta^j\}\|^2
+C\|\{w^j,h^j,\vec{v}^{\,j},q^j,\eta^j\}\|^{18}\\
&+C\|\{w^{(j-1)},h^{(j-1)},
\vec{v}^{\,(j-1)},q^{(j-1)},\eta^{(j-1)}\}\|^2
+C\|\{w^{(j-1)},h^{(j-1)},
\vec{v}^{\,(j-1)},q^{(j-1)},\eta^{(j-1)}\}\|^{18}
\end{split}
\enn
and
\ben
\begin{split}
C_9(j):=&
C\|\{w^{(j-1)},h^{(j-1)},
\vec{v}^{\,(j-1)},q^{(j-1)},\eta^{(j-1)}\}\|^2
+C\|\{w^{(j-1)},h^{(j-1)},
\vec{v}^{\,(j-1)},q^{(j-1)},\eta^{(j-1)}\}\|^{18}\\
&+C\|\{w^{(j-2)},h^{(j-2)},
\vec{v}^{\,(j-2)},q^{(j-2)},\eta^{(j-2)}\}\|^2
+C\|\{w^{(j-2)},h^{(j-2)},
\vec{v}^{\,(j-2)},q^{(j-2)},\eta^{(j-2)}\}\|^{18}
\end{split}
\enn
with some constant $C$ independent of $j$ and $t$.
 From Proposition \ref{p4} and \eqref{ae1} we conclude that there exists a constant $C_{10}$ independent of $j$ and $t$ such that
\be\label{ae13}
\begin{split}
\|\{\de& w^{(j+1)},\de h^{(j+1)},\de \vec{v}^{\,(j+1)},\de q^{(j+1)},
\de \eta^{(j+1)}\}\|^2\\
\leq& C_{10}(\|w_0\|_{H^2}+\|h_0\|_{H^2}
+\|\vec{v}_0\|_{H^2}+\|\eta_0\|_{H^3(\R^2)})^2
\times\|\{\de w^{j},\de h^{j},\de \vec{v}^{j},\de q^{j},
\de \eta^{j}\}\|^2\\
&+
C_{10}(\|w_0\|_{H^2}+\|h_0\|_{H^2}
+\|\vec{v}_0\|_{H^2}+\|\eta_0\|_{H^3(\R^2)})^{18}
\times\|\{\de w^{j},\de h^{j},\de \vec{v}^{j},\de q^{j},
\de \eta^{j}\}\|^2\\
&+C_{10}(\|w_0\|_{H^2}+\|h_0\|_{H^2}
+\|\vec{v}_0\|_{H^2}+\|\eta_0\|_{H^3(\R^2)})^{2}\\
&\ \ \ \times\|\{\de w^{(j-1)},\de h^{(j-1)},\de \vec{v}^{\,(j-1)},\de q^{(j-1)},
\de \eta^{(j-1)}\}\|^2\\
&+C_{10}(\|w_0\|_{H^2}+\|h_0\|_{H^2}
+\|\vec{v}_0\|_{H^2}+\|\eta_0\|_{H^3(\R^2)})^{18}\\
&\ \ \ \times\|\{\de w^{(j-1)},\de h^{(j-1)},\de \vec{v}^{\,(j-1)},\de q^{(j-1)},
\de \eta^{(j-1)}\}\|^2.
\end{split}
\ee
Assume
\be\label{ae14}
\begin{split}
C_{10}&(\|w_0\|_{H^2}+\|h_0\|_{H^2}
+\|\vec{v}_0\|_{H^2}+\|\eta_0\|_{H^3(\R^2)})^2\\
&+C_{10}(\|w_0\|_{H^2}+\|h_0\|_{H^2}
+\|\vec{v}_0\|_{H^2}+\|\eta_0\|_{H^3(\R^2)})^{18}
<\frac{1}{4}.
\end{split}
\ee
Then it follows from \eqref{ae13} and \eqref{ae14} that
\ben
\begin{split}
\|\{\de &w^{(j+1)},\de h^{(j+1)},\de \vec{v}^{\,(j+1)},\de q^{(j+1)},
\de \eta^{(j+1)}\}\|^2
+\frac{1}{2}\|\{\de w^{j},\de h^{j},\de \vec{v}^{j},\de q^{j},
\de \eta^{j}\}\|^2\\
< & \frac{3}{4}
\big(\|\{\de w^{j},\de h^{j},\de \vec{v}^{j},\de q^{j},
\de \eta^{j}\}\|^2
+\frac{1}{2}\|\{\de w^{(j-1)},\de h^{(j-1)},\de \vec{v}^{\,(j-1)},\de q^{(j-1)},
\de \eta^{(j-1)}\}\|^2
\big),
\end{split}
\enn
which, along with Proposition \ref{p4} indicates that
\be\label{ae15}
\begin{split}
\|\{\de& w^{(j+1)},\de h^{(j+1)},\de \vec{v}^{\,(j+1)},\de q^{(j+1)},
\de \eta^{(j+1)}\}\|^2\\
<&\Big(\frac{3}{4}\Big)^{j-2}
\big(
\|\{\de w^{3},\de h^{3},\de \vec{v}^{\,3},\de q^{3},
\de \eta^{3}\}\|^2
+\frac{1}{2}\|\{\de w^{2},\de h^{2},\de \vec{v}^{\,2},\de q^{2},
\de \eta^{2}\}\|^2
\big) \\
\leq & C\Big(\frac{3}{4}\Big)^{j-2}
(\|w_0\|_{H^2}+\|h_0\|_{H^2}
+\|\vec{v}_0\|_{H^2}+\|\eta_0\|_{H^3(\R^2)})^2
\end{split}
\ee
for $j\geq 3$ and $t>0$, where the constant $C>0$ is independent of $j$ and $t$.
From \eqref{ae15} we know that $\{(w^{j}, h^{j},\vec{v}^{j}, q^{j},
 \eta^{j})\}_{j\in \N_{+}}$ is a Cauchy sequence, thus there exists a unique limit $
(w, h,\vec{v},q,
 \eta)$
satisfying
 \ben
 \lim\limits_{j\to \infty}\|\{w^{j}-w, h^{j}-h,\vec{v}^{j}-\vec{v},q^{j}-q,
 \eta^{j}-\eta\}\|^2=0
 \enn
and
\be\label{ae16}
\|\{w,h,\vec{v},q,\eta\}\|^2
\leq C(\|w_0\|_{H^2}+\|h_0\|_{H^2}
+\|\vec{v}_0\|_{H^2}+\|\eta_0\|_{H^3(\R^2)})^2.
\ee
 Passing $j\to \infty$ in \eqref{ee1}-\eqref{eee7} we deduce that $(w, h,\vec{v},q, \eta)$ solves \eqref{e001}-\eqref{e003}. Moreover, from \eqref{eee14} we deduce that $J$, the Jacobian determinant of $d\theta$ satisfies
\be\label{ae25}
\frac{1}{2}<J<\frac{3}{2}\quad \text{in}\ \ \Om\times(0,\infty),
\ee
which, along with the transformation $\theta$ given in \eqref{aae1} indicates that the $(m,\ti{c},\vec{u},p)$ defined in \eqref{aae3}-\eqref{aae8}, along with $\eta$ solves the initial-boundary value problem \eqref{e01}-\eqref{e03}. Uniqueness and estimates \eqref{e07} follow directly from \eqref{ae16} and \eqref{ae25}. The proof is finished.

 \endProof

We next prove Theorem \ref{t1} by using Theorem \ref{t2} and reversing transformation \eqref{tr}.
\newline
\textbf{\emph{Proof of Theorem \ref{t1}.}} First, it follows from the Sobolev embedding inequality and Theorem \ref{t2} that
\be\label{e18}
\begin{split}
\sup_{t>0}\|\tilde{c}(t)\|_{L^\infty(\Om_t)}\leq& C \sup_{t>0}\|\tilde{c}(t)\|_{H^2(\Om_t)}\\
\leq &C (\|m_0\|_{H^2(\Om_0)}+\|\ti{c}_0\|_{H^2(\Om_0)}
+\|\vec{u}_0\|_{H^2(\Om_0)}
+\|\eta_0\|_{H^3(\R^2)}).
\end{split}
\ee
 From transformation \eqref{tr} we deduce that
\be\label{eeea3}
c(x_1,x_2,y,t)-\hat{c}=\hat{c}\big(\exp\{-\tilde{c}(x_1,x_2,y,t)\}-1\big)
=\hat{c}\sum_{k=1}^{\infty}\frac{[-\tilde{c}(x_1,x_2,y,t)]^k}{k!},
\ee
which, along with Theorem \ref{t2} and \eqref{e18} indicates that
\be\label{e19}
\begin{split}
\sup_{t>0}\|c(t)-\hat{c}\|_{L^2(\Om_t)}
\leq& \hat{c}\sup_{t>0}\|\tilde{c}(t)\|_{L^2(\Om_t)}
\sum_{k=1}^\infty \frac{\Big(\sup\limits_{t>0}\|\tilde{c}(t)\|_{L^\infty(\Om_t)}\Big)^{k-1}}{k!}\\
\leq &\hat{c}\sup_{t>0}\|\tilde{c}(t)\|_{L^2(\Om_t)}
\exp\Big\{\sup_{t>0}\|\tilde{c}(t)\|_{L^\infty(\Om_t)}\Big\}\\
\leq &C \big(\|m_0\|_{H^2(\Om_0)}+\|\ti{c}_0\|_{H^2(\Om_0)}
+\|\vec{u}_0\|_{H^2(\Om_0)}
+\|\eta_0\|_{H^3(\R^2)}\big)\\
&\ \times \exp\{C \big(\|m_0\|_{H^2(\Om_0)}+\|\ti{c}_0\|_{H^2(\Om_0)}
+\|\vec{u}_0\|_{H^2(\Om_0)}
+\|\eta_0\|_{H^3(\R^2)}\big)\}
.
\end{split}
\ee
Applying $\na$ to \eqref{eeea3} and using Theorem \ref{t2} and \eqref{e18} one gets
\be\label{e20}
\begin{split}
\sup_{t>0}\|\na c(t)\|_{L^2(\Om_t)}
\leq& \hat{c}\sup_{t>0}\|\na\tilde{c}(t)\|_{L^2(\Om_t)}
\exp\Big\{\sup_{t>0}\|\tilde{c}(t)\|_{L^\infty(\Om_t)}\Big\}\\
\leq &C \big(\|m_0\|_{H^2(\Om_0)}+\|\ti{c}_0\|_{H^2(\Om_0)}
+\|\vec{u}_0\|_{H^2(\Om_0)}
+\|\eta_0\|_{H^3(\R^2)}\big)\\
&\ \times \exp\{C \big(\|m_0\|_{H^2(\Om_0)}+\|\ti{c}_0\|_{H^2(\Om_0)}
+\|\vec{u}_0\|_{H^2(\Om_0)}
+\|\eta_0\|_{H^3(\R^2)}\big)\}.
\end{split}
\ee
By a similar argument in deriving \eqref{e20} we deduce that
\ben
\begin{split}
\sup_{t>0}\|\na^2 c(t)\|_{L^2(\Om_t)}
\leq& \hat{c}\big(\sup_{t>0}\|\na\tilde{c}(t)\|_{L^4(\Om_t)}^2
 +\sup_{t>0}\|\na^2\tilde{c}(t)\|_{L^2(\Om_t)}\big)
\exp\Big\{\sup_{t>0}\|\tilde{c}(t)\|_{L^\infty(\Om_t)}\Big\}\\
\leq &C \big[\sum_{k=1}^2\big(\|m_0\|_{H^2(\Om_0)}+\|\ti{c}_0\|_{H^2(\Om_0)}
+\|\vec{u}_0\|_{H^2(\Om_0)}
+\|\eta_0\|_{H^3(\R^2)}\big)^k\big]\\
&\ \ \times \exp\{C \big(\|m_0\|_{H^2(\Om_0)}+\|\ti{c}_0\|_{H^2(\Om_0)}
+\|\vec{u}_0\|_{H^2(\Om_0)}
+\|\eta_0\|_{H^3(\R^2)}\big)\},
\end{split}
\enn
which, in conjunction with \eqref{e19}-\eqref{e20} leads to
\be\label{e21}
\begin{split}
\sup_{t>0}\| c(t)-\hat{c}\|_{H^2(\Om_t)}
\leq &C \big[\sum_{k=1}^2\big(\|m_0\|_{H^2(\Om_0)}+\|\ti{c}_0\|_{H^2(\Om_0)}
+\|\vec{u}_0\|_{H^2(\Om_0)}
+\|\eta_0\|_{H^3(\R^2)}\big)^k\big]\\
&\ \ \times \exp\{C \big(\|m_0\|_{H^2(\Om_0)}+\|\ti{c}_0\|_{H^2(\Om_0)}
+\|\vec{u}_0\|_{H^2(\Om_0)}
+\|\eta_0\|_{H^3(\R^2)}\big)\}.
\end{split}
\ee
By a similar arguments used in obtaining \eqref{e21} one can easily deduce that
\ben
\begin{split}
\int_0^\infty \| c(t)-\hat{c}\|_{H^3(\Om_t)}^2\,dt
\leq & C \big[\sum_{k=1}^3\big(\|m_0\|_{H^2(\Om_0)}+\|\ti{c}_0\|_{H^2(\Om_0)}
+\|\vec{u}_0\|_{H^2(\Om_0)}
+\|\eta_0\|_{H^3(\R^2)}\big)^{2k}
\big]\\
&\ \ \times \exp\{C \big(\|m_0\|_{H^2(\Om_0)}+\|\ti{c}_0\|_{H^2(\Om_0)}
+\|\vec{u}_0\|_{H^2(\Om_0)}
+\|\eta_0\|_{H^3(\R^2)}\big)\},
\end{split}
\enn
 which, along with \eqref{e21} and Theorem \ref{t2} gives the desired estimates \eqref{e08}. The nonnegativity of $m$ follows from the maximum principle and $c>0$ follows from $c(x_1,x_2,y,t)=\hat{c}\exp\{-\tilde{c}(x_1,x_2,y,t)\}$. The proof is finished.

\endProof
\section{Appendix}
This section is devoted to the derivation of system \eqref{e001}-\eqref{e003}.
First, it follows from \eqref{aae1}-\eqref{aae3} that
\be\label{aae19}
\partial_{j} m=\xi_{kj}\partial_k w,\quad
\partial_{j} \tilde{c}=\xi_{kj}\partial_k h,\quad
\partial_{j}^2m= \xi_{kj}\partial_{k}(\xi_{lj}\partial_l w),\quad
\partial_{j}^2\tilde{c}= \xi_{kj}\partial_{k}(\xi_{lj}\partial_l h),
\ee
where the derivatives on the left-(or right-)hand side of each equality are with respect to the coordinates in $\Om_t$ (or $\Om$) and repeated indices are summed.
Using \eqref{aae2}-\eqref{aae18} we deduce from \eqref{aae19} that
\be\label{aae5}
\begin{split}
\na m=&(
\partial_1 w-J^{-1}\al\partial_3 w,\
\partial_2 w-J^{-1}\beta\partial_3 w,\
J^{-1}\partial_3 w
),\\
\na \tilde{c}=&(
\partial_1 h-J^{-1}\al\partial_3 h,\
\partial_2 h-J^{-1}\beta\partial_3 h,\
J^{-1}\partial_3 h
)
\end{split}
\ee
and that
\be\label{aae6}
\begin{split}
\De m=&\De w+(J^{-2}-1)\partial_3^2 w
+J^{-1}\partial_3(J^{-1})\partial_3w
-\partial_1(J^{-1}\al\partial_3 w)
-\partial_2(J^{-1}\beta\partial_3 w)\\
&-J^{-1}\al\partial_3\left(\partial_1 w-J^{-1}\al\partial_3w\right)
-J^{-1}\beta\partial_3\left(\partial_2 w-J^{-1}\beta\partial_3w\right),\\
\De \tilde{c}=&\De h+(J^{-2}-1)\partial_3^2 h
+J^{-1}\partial_3(J^{-1})\partial_3h
-\partial_1(J^{-1}\al\partial_3 h)
-\partial_2(J^{-1}\beta\partial_3 h)\\
&-J^{-1}\al\partial_3\left(\partial_1 h-J^{-1}\al\partial_3h\right)
-J^{-1}\beta\partial_3\left(\partial_2 h-J^{-1}\beta\partial_3h\right).
\end{split}
\ee
Noting that $\theta_3(x_1,x_2,y,t)=\bar{\eta}+y(1+\bar{\eta}/b)$ depends on $t$, one gets from \eqref{aae1} and \eqref{aae5} that
\be\label{aae7}
\begin{split}
w_t=&m_t+(\partial_3 m)\times (\partial_t\theta_3)
=m_t+(J^{-1}\partial_3w)\times(1+y/b)\bar{\eta}_t,\\
h_t=&\tilde{c}_t+(\partial_3 \tilde{c})\times (\partial_t\theta_3)
=\tilde{c}_t+(J^{-1}\partial_3h)\times(1+y/b)\bar{\eta}_t.
\end{split}
\ee
Substituting \eqref{aae5}-\eqref{aae7} and \eqref{aae8} into \eqref{e01}-\eqref{e03} we get
\be\label{aae9}
\left\{
\begin{array}{lll}
w_t-\Delta w-\nabla\cdot(w\nabla h)=F_4(w,h,\vec{v},\bar{\eta}),\quad (x_1,x_2,y,t)\in \Omega\times(0,\infty),\\
h_t-\Delta h-w=F_5(h,\vec{v},\bar{\eta}),\\
(w,h)(x_1,x_2,y,0)=(w_0,h_0)(x_1,x_2,y),\\
\partial_3 w+w\partial_3 h=G_4(w,h,\bar{\eta}),\quad h=0\quad\text{on}\ \ \Gamma\times(0,\infty),\\
w=0,\quad\partial_3 h=0 \quad \text{on}\ \ S_B\times(0,\infty),
\end{array}
\right.
\ee
where
\ben
\begin{split}
F_4=&[J^{-2}(\al^2+\beta^2+1)-1]\partial_3(\partial_3 w+w\partial_3h)
-2J^{-1}\al (\partial_3 \partial_1w+w\partial_3\partial_1 h)
-2J^{-1}\beta (\partial_3 \partial_2w+w\partial_3\partial_2 h)\\
&-J^{-1}v_1(\partial_1w-J^{-1}\al\partial_3 w)
-J^{-1}v_2(\partial_2w-J^{-1}\beta\partial_3 w)
-J^{-1}(J^{-1}\al v_1+J^{-1}\beta v_1+v_3)\partial_3w\\
&-\partial_1(J^{-1}\al)\partial_3w
-\partial_2(J^{-1}\beta)\partial_3w
+J^{-1}\al \partial_{3}(J^{-1}\al)\partial_3w
+J^{-1}\beta \partial_{3}(J^{-1}\beta)\partial_3w\\
&-J^{-1}\al(\partial_1w\partial_3 h+\partial_3 w\partial_1 h)
-J^{-1}\beta(\partial_2w\partial_3 h+\partial_3 w\partial_2 h)
-w\partial_{1}(J^{-1}\al)\partial_3 h\\
&-w\partial_{2}(J^{-1}\beta)\partial_3 h
+wJ^{-1}\al\partial_3(J^{-1}\al)\partial_3h
+wJ^{-1}\beta\partial_3(J^{-1}\beta)\partial_3h
+J^{-1}\partial_3(J^{-1})\partial_3w\\
&+(J^{-1}\partial_3w)\times(1+y/b)\bar{\eta}_t
\end{split}
\enn
and
\ben
\begin{split}
F_5=&[J^{-2}(\al^2+\beta^2+1)-1]\partial_3^2 h
-2J^{-1}\al \partial_3 \partial_1h
-2J^{-1}\beta \partial_3 \partial_2h\\
&-J^{-1}v_1(\partial_1h-J^{-1}\al\partial_3 h)
-J^{-1}v_2(\partial_2h-J^{-1}\beta\partial_3 h)
-J^{-1}(J^{-1}\al v_1+J^{-1}\beta v_1+v_3)\partial_3h\\
&-\partial_1(J^{-1}\al)\partial_3h
-\partial_2(J^{-1}\beta)\partial_3h
+J^{-1}\al \partial_{3}(J^{-1}\al)\partial_3h
+J^{-1}\beta \partial_{3}(J^{-1}\beta)\partial_3h\\
&-(\partial_1 h-J^{-1}\al\partial_3 h)^2
-(\partial_2 h-J^{-1}\beta\partial_3 h)^2
-J^{-2}(\partial_3 h)^2
+J^{-1}\partial_3(J^{-1})\partial_3h\\
&+(J^{-1}\partial_3h)\times(1+y/b)\bar{\eta}_t
\end{split}
\enn
and
\be\label{aae21}
G_4=J(\al^2+\beta^2+1)^{-1}[\partial_1 \eta (\partial_1w+w\partial_1 h)
+\partial_2 \eta (\partial_2w+w\partial_2 h)].
\ee
It follows from \eqref{aae3}-\eqref{aae4} that
\be\label{aae20}
\partial_{j} u_i=\xi_{kj}\partial_k (J^{-1}v_m\partial_m \theta_i ),\quad
\partial_{j}^2u_i= \xi_{kj}\partial_{k}[\xi_{lj}\partial_l ( J^{-1}v_m\partial_m\theta_i)],
\quad \partial_i p=\xi_{ki}\partial_k q,
\quad\partial_i \Phi=\xi_{ki}\partial_k \phi.
\ee
Thus
\be\label{aae11}
\begin{split}
[\vec{u}\cdot\na \vec{u}+\nabla p+m\nabla \Phi-\Delta\vec{u}]_i
=&
J^{-1}v_l\partial_l \theta_j
\xi_{kj}\partial_k (J^{-1}v_m \partial_m\theta_i )
+\xi_{ki}\partial_k q
+w \xi_{ki}\partial_k \phi\\
&-
 \xi_{kj}\partial_{k}[\xi_{lj}\partial_l ( J^{-1}v_m\partial_m\theta_i)],
\end{split}
\ee
where we have used $[\vec{u}\cdot\na \vec{u}+\nabla p+m\nabla \Phi-\Delta\vec{u}]_i$ to denote the $i$-th component of $[\vec{u}\cdot\na \vec{u}+\nabla p+m\nabla \Phi-\Delta\vec{u}]$.
 Differentiating \eqref{aae4} with respect to $t$ one derives
\ben
u_{it}+(\partial_3 u_i)(\partial_t \theta_3)=\partial_t[J^{-1}(\partial_j \theta_i)v_j]
=J^{-1}(\partial_j \theta_i)v_{jt}
+J^{-1}(\partial_j \theta_{it})v_j
-J^{-2}J_t(\partial_j \theta_i)v_j,
\enn
which, along with \eqref{aae1} and \eqref{aae18} leads to
\be\label{aae10}
\begin{split}
u_{it}=&J^{-1}(\partial_j \theta_i)v_{jt}
+J^{-1}(\partial_j \theta_{it})v_j
-J^{-2}[\bar{\eta}_t/b+\partial_3\bar{\eta}_t(1+y/b)]
(\partial_j \theta_i)v_j\\
&-\xi_{k3}\partial_k (J^{-1}v_m\partial_m \theta_i )(1+y/b)\bar{\eta}_t
\end{split}
\ee
with $\theta_{1t}=0$, $\theta_{2t}=0$ and $\theta_{3t}=(1+y/b)\bar{\eta}_t$.
Substituting \eqref{aae11}-\eqref{aae10} into the third equation of \eqref{e01} and multiplying the resulting equality by $J(\xi_{ij})_{3\times 3}$ we arrive at
\be\label{aae15}
\begin{split}
\vec{v}_t-\Delta\vec{v}+\nabla q +w\nabla \phi=\vec{F}(w,\vec{v},\na q,\bar{\eta}),
\end{split}
\ee
where $\vec{F}=(F_1,F_2,F_3)$ with
\ben
\begin{split}
F_1=&
[J^{-2}(\al^2+\beta^2+1)-1]\partial_3^2 v_1
-2J^{-1}\al \partial_3 \partial_1v_1
-2J^{-1}\beta \partial_3 \partial_2v_1\\
&+2J\partial_1(J^{-1})\partial_1v_1
+J\partial_1^2(J^{-1})v_1
+2J\partial_2(J^{-1})\partial_2v_1
+J\partial_2^2(J^{-1})v_1
+\partial_3[J^{-1}\partial_3(J^{-1})v_1]\\
&+\partial_3(J^{-2})\partial_3v_1
-J\partial_1[\al J^{-1} v_1\partial_3(J^{-1})]
-J\partial_1(\al J^{-2})\partial_3 v_1
-J\partial_2[\beta J^{-1} v_1\partial_3(J^{-1})]\\
&-J\partial_2(\beta J^{-2})\partial_3 v_1
-\al \partial_3[v_1\partial_1(J^{-1})]
-\al \partial_3(J^{-1})\partial_1v_1
-\beta \partial_3[v_1\partial_2(J^{-1})]
-\beta \partial_3(J^{-1})\partial_2v_1\\
&+\al \partial_3[J^{-1}\al v_1\partial_3(J^{-1})]
+\al \partial_3(J^{-2}\al)\partial_3v_1
+\beta \partial_3[J^{-1}\beta v_1\partial_3(J^{-1})]
+\beta \partial_3(J^{-2}\beta)\partial_3v_1\\
&-v_1\partial_1(J^{-1}v_1)-v_2\partial_2(J^{-1}v_1)
-v_3\partial_3(J^{-1}v_1)\\
&+J^{-1}v_1[\bar{\eta}_t/b+\partial_3 \bar{\eta}_t(1+y/b)]
+\partial_3(J^{-1}v_1)(1+y/b)\bar{\eta}_t\\
&+w\al \partial_3 \phi+w(1-J)\partial_1 \phi
+\al \partial_3 q+(1-J)\partial_1 q.
\end{split}
\enn
One can derive $F_2$ by replacing the terms in line 6 and line 8 of $F_1$ by
\ben
-v_1\partial_1(J^{-1}v_2)-v_2\partial_2(J^{-1}v_2)
-v_3\partial_3(J^{-1}v_2)
+w\beta \partial_3 \phi+w(1-J)\partial_2 \phi
+\beta \partial_3 q+(1-J)\partial_2 q
\enn
and replacing the $v_1$ in other terms of $F_1$ by $v_2$. For brevity, we shall not write out the explicit expression of $F_2$. The third component $F_3$ is as follows:
\ben
\begin{split}
F_3=&[J^{-2}(\al^2+\beta^2+1)-1]\partial_3^2 v_3
-2J^{-1}\al \partial_3 \partial_1v_3
-2J^{-1}\beta \partial_3 \partial_2v_3\\
&+J^{-1}\partial_3(J^{-1})\partial_3 v_3
-\partial_1(\al J^{-1})\partial_3v_3
-\partial_2(\beta J^{-1})\partial_3v_3\\
&+\al J^{-1}\partial_3(\al J^{-1})\partial_3v_3
+\beta J^{-1}\partial_3(\beta J^{-1})\partial_3v_3\\
&+v_1J^{-1}\partial_1v_3+v_2J^{-1}\partial_2v_3
+v_3J^{-1}\partial_3v_3
+v_1^2J^{-2}\partial_1 \al+v_1v_2J^{-2}(\partial_1 \beta
+\partial_2 \al)\\
&+v_2^2J^{-2}\partial_2 \beta
+v_1v_3J^{-2}\partial_3 \al
+v_2v_3J^{-2}\partial_2 \beta\\
&+J^{-1}(v_1J^{-1}\partial_3\al+v_2J^{-1}\partial_3\beta +\partial_3v_3)(1+y/b)\bar{\eta}_t
-J^{-1}v_1(1+y/b)\partial_1\bar{\eta}_t
-J^{-1}v_2(1+y/b)\partial_2\bar{\eta}_t\\
&+w\al\partial_1\phi +w\beta\partial_2\phi +w[1-J^{-1}(\al^2+\beta^2+1)]\partial_3\phi
+\al\partial_1q +\beta\partial_2q +[1-J^{-1}(\al^2+\beta^2+1)]\partial_3q.
\end{split}
\enn
Substitute \eqref{aae20} into the last boundary condition in \eqref{e02} we derive on $\G\times(0,\infty)$ that
\be\label{aae12}
q N_i-[\xi_{kj}\partial_k (J^{-1}v_m\partial_m \theta_i)+\xi_{ki}\partial_k (J^{-1}v_m\partial_m \theta_j)]N_j=\left\{
\gamma\eta-\sigma\na_0\cdot\left(
\frac{\na_0 \eta}{\sqrt{1+|\na_0\eta|^2}}
\right)
\right\}N_i,
\ee
where
\ben
\vec{N}:=\vec{n}\circ\theta=\frac{(-\partial_1\eta,-\partial_2\eta,1)}
{\sqrt{1+(\na_0\eta)^2
}}.
 \enn
Taking the inner product of \eqref{aae12} with $T_1:=(1,0,\partial_1\eta)$ one deduces that
\be\label{aae13}
\partial_3v_1+\partial_1 v_3=G_1(\vec{v},\bar{\eta})\quad \text{on}\ \ \G\times(0,\infty),
\ee
with
\ben
\begin{split}
G_1=&2[\partial_{1}(J^{-1}v_1)
-J^{-1}\partial_1\eta\partial_3(J^{-1}v_1)]\partial_1\eta
+[\partial_2(J^{-1}v_1)
-J^{-1}\partial_2\eta\partial_3(J^{-1}v_1)]\partial_2\eta\\
&+[\partial_1(J^{-1}v_2)
-J^{-1}\partial_1\eta\partial_3(J^{-1}v_2)]\partial_2\eta
+(1-J^{-2})\partial_3v_1\\
&-[J^{-1}v_1\partial_3(J^{-1})+\partial_1(J^{-1}v_1\partial_1\eta
+J^{-1}v_2\partial_2\eta)]
+[J^{-1}\partial_1\eta \partial_3 (v_1J^{-1}\partial_1\eta
+v_2J^{-1}\partial_2\eta+v_3)]\\
&+\![J^{-1}\partial_3(J^{-1}v_1)
\!+\!\partial_1(J^{-1} v_1\partial_1\eta
\!+\!J^{-1} v_2\partial_2\eta\!+\!v_3)
\!-\!J^{-1}\partial_1\eta\partial_3(J^{-1} v_1\partial_1\eta
\!+\!J^{-1} v_2\partial_2\eta\!+\!v_3
)](\partial_1\eta)^2\\
&+\![J^{-1}\partial_3(J^{-1}v_2)
\!+\!\partial_2(J^{-1} v_1\partial_1\eta
\!+\!J^{-1} v_2\partial_2\eta\!+\!v_3)
\!-\!J^{-1}\partial_2\eta\partial_3(J^{-1} v_1\partial_1\eta
\!+\!J^{-1} v_2\partial_2\eta\!+\!v_3
)]\partial_1\eta\partial_2\eta\\
&-2J^{-1}\partial_3(J^{-1} v_1\partial_1\eta
+J^{-1} v_2\partial_2\eta+v_3
)\partial_1\eta.
\end{split}
\enn
Taking the inner product of \eqref{aae12} with $T_{2}:=(0,1,\partial_2\eta)$ to have
\be\label{aae16}
\partial_3v_2+\partial_2v_3=G_2(\vec{v},\bar{\eta})\quad \text{on}\ \ \G\times(0,\infty),
\ee
where $G_2$ has an expression similar to that of $G_1$ and we shall not write out the explicit form of $G_2$ for brevity. Taking the inner product of \eqref{aae12} with $\vec{N}$ one gets
\be\label{aae17}
q-2\partial_3 v_3=\gamma\eta-\sigma\Delta_0\eta-G_3(\vec{v},\bar{\eta})\quad \text{on}\ \ \G\times(0,\infty),
\ee
where $G_3=\sigma\na_0\cdot\left(
\frac{\na_0 \eta}{\sqrt{1+|\na_0\eta|^2}}
\right)-\sigma\De_0\eta+\tilde{G}_3$ with $\tilde{G}_3$ defined in the following way:
\ben
\begin{split}
[1+&(\na_0\eta)^2
]\tilde{G}_3\\
=&
-2[\partial_{1}(J^{-1}v_1)
-J^{-1}\partial_1\eta\partial_3(J^{-1}v_1)](\partial_1\eta)^2
-2[\partial_{2}(J^{-1}v_2)
-J^{-1}\partial_2\eta\partial_3(J^{-1}v_2)](\partial_2\eta)^2
\\
&-2[\partial_2(J^{-1}v_1)
-J^{-1}\partial_2\eta\partial_3(J^{-1}v_1)
+\partial_1(J^{-1}v_2)
-J^{-1}\partial_1\eta\partial_3(J^{-1}v_2)
]
(\partial_2\eta)(\partial_1\eta)\\
&+[J^{-1}\partial_3(J^{-1}v_1)+\partial_1(J^{-1}v_1\partial_1\eta
+J^{-1}v_2\partial_2\eta+v_3)
]
(\partial_1\eta)
\\
&-[J^{-1}\partial_1\eta \partial_3 (v_1J^{-1}\partial_1\eta
+v_2J^{-1}\partial_2\eta+v_3)](\partial_1\eta)\\
&+[J^{-1}\partial_3(J^{-1}v_2)+\partial_2(J^{-1}v_1\partial_1\eta
+J^{-1}v_2\partial_2\eta+v_3)
]
(\partial_2\eta)
\\
&-[J^{-1}\partial_2\eta \partial_3 (v_1J^{-1}\partial_1\eta
+v_2J^{-1}\partial_2\eta+v_3)](\partial_2\eta)\\
&+[J^{-1}\partial_3(J^{-1}v_1)
+\partial_1(J^{-1} v_1\partial_1\eta
+J^{-1} v_2\partial_2\eta+v_3)
-J^{-1}\partial_1\eta\partial_3(J^{-1} v_1\partial_1\eta
+J^{-1} v_2\partial_2\eta+v_3
)](\partial_1\eta)\\
&+[J^{-1}\partial_3(J^{-1}v_2)
+\partial_2(J^{-1} v_1\partial_1\eta
+J^{-1} v_2\partial_2\eta+v_3)
-J^{-1}\partial_2\eta\partial_3(J^{-1} v_1\partial_1\eta
+J^{-1} v_2\partial_2\eta+v_3
)](\partial_2\eta)\\
&-2J^{-1}\partial_3(J^{-1} v_1\partial_1\eta
+J^{-1} v_2\partial_2\eta)
+2[1+(\na_0\eta)^2
-J^{-1}](\partial_3v_3)
.
\end{split}
\enn
Noting that $\al=\partial_1\eta$ and $\beta=\partial_2\eta$ on $\G\times(0,\infty)$, we deduce from \eqref{aae8} and the third boundary condition in \eqref{e02} that
\be\label{aae14}
\eta_t=v_3\quad \text{on}\ \ \G\times(0,\infty).
\ee
Moreover, by  rewriting \eqref{aae8} as
\be\label{aae22}
\begin{split}
&v_1(x_1,x_2,y,t)=J u_1(\theta(x_1,x_2,y,t),t),
\quad
v_2(x_1,x_2,y,t)=J u_2(\theta(x_1,x_2,y,t),t),\\
&v_3(x_1,x_2,y,t)=u_3(\theta(x_1,x_2,y,t),t)-\al u_1(\theta(x_1,x_2,y,t),t)
-\beta u_2(\theta(x_1,x_2,y,t),t)
\end{split}
\ee
and applying a direct computation to \eqref{aae22}, one can easily deduces from $\na\cdot\vec{u}=0$ that
\be\label{aae23}
\na\cdot\vec{v}=0\quad \text{in}\ \ \Om\times(0,\infty).
\ee
Collecting \eqref{aae9}, \eqref{aae15}, \eqref{aae13}-\eqref{aae17}, \eqref{aae14} and \eqref{aae23} we derive \eqref{e001}-\eqref{e003}.

\endProof

\section*{Acknowledgements}
This work is supported by National Natural Science Foundation of China (No. 11901139), China Postdoctoral Science Foundation
(No. 2019M651269, No. 2020T130151).

\setlength{\bibsep}{0.5ex}
\bibliography{rf}
\bibliographystyle{plain}

\end{document}